\documentclass{amsart}
\usepackage{mathpreamble}

\begin{document}

\title{Scalar curvature of blow-ups of compact K\"ahler manifolds along complex submanifolds}

%\begin{comment}
\author{Zeqing Miao}
\address{School of Mathematical Sciences, Xiamen University, Xiamen, Fujian, 361005, China.}
\email{{19020240157493@stu.xmu.edu.cn}}
\author{Bo Yang}
\thanks{The second-named author is partially supported by National Natural Science Foundation of China
with the grant numbers: 11801475, 12141101, and 12271451, and Natural Science Foundation of Fujian Province of China with the grant No.2019J05012.}    
\address{School of Mathematical Sciences, Xiamen University, Xiamen, Fujian, 361005, China.}
\email{{boyang@xmu.edu.cn}}    
%\end{comment}

\begin{abstract}
Let $(M,\omega)$ be a compact K\"ahler manifold with its scalar curvature $S(\omega)$, Assume that $M$ has complex dimension at least $3$ and contains a complex submanifold $X$ of complex codimension at least $2$. Let $\sigma: Bl_{X} M \rightarrow M$ denote the blow-up of $M$ along $X$. We show that $Bl_{X} M$ admits a sequence of K\"ahler metrics $\{\widetilde{\omega}_{i}\}_{i \geq 1}$ whose scalar curvatures $S(\widetilde{\omega}_i)$ converge to $\sigma^{\ast} (S(\omega))$ in the $C^0(Bl_{X} M)$ norm. Our work is motivated by a recent result of Brown, who established the corresponding result for blow-ups at a point. The proof is based on the gluing method for constructing extremal K\"ahler metrics on blow-ups, together with new analytic tools and several modifications needed in our setting.
\end{abstract}


\date{Version of 08/23/2026}

\maketitle

\markleft{Scalar curvature of blow-ups along complex submanifolds}

\markright{Scalar curvature of blow-ups along complex submanifolds}

\setcounter{tocdepth}{1}
\tableofcontents

\begin{comment}
{\color{blue}
checklist:

}    
\end{comment}

\section{Introduction}

\subsection{Statement of the main result}

A classical question in Riemannian geometry, posed in \cite{KW1974, KW1975}, asks whether a given function on a smooth manifold $M$ can be the scalar curvature of some Riemannian metric on $M$. While this general question has been investigated in various settings, in this work we focus on the behavior of scalar curvature under blow-ups of compact K\"ahler manifolds. Our main theorem is as follows.

\begin{theorem}\label{MainThm_intro}
    Let $(M^m,\omega)$ be a compact K\"{a}hler manifold of complex dimension $m \geq 3$. Assume that $M$ admits a complex submanifold $X$ of complex codimension $2 \leq k \leq m-1$. Let $\sigma: Bl_{X}M \rightarrow M$ be the blow-up of $M$ along $X$ with the exceptional divisor $E=\sigma^{-1}(X)$. For any $\widetilde{\epsilon}>0$, there exists $\epsilon_0>0$ such that for any $\epsilon \in (0, \epsilon_0]$, there exists a K\"{a}hler metric $\widetilde{\omega}$ on $Bl_{X}M$ in the class $\sigma^{*}[\omega]-\epsilon^{2}[E]$ with $\|S(\widetilde{\omega})-\sigma^{\ast} (S(\omega))\|_{C^0(Bl_{X} M)} \leq \widetilde{\epsilon}$. Here, $[E] \coloneqq c_1(\mathcal{O}_{Bl_{X}M}(E))$ denotes the first Chern class of the line bundle associated with $E$, and $S(\widetilde{\omega})$ and $S(\omega)$ are the scalar curvatures of $\widetilde{\omega}$ and $\omega$, respectively.    
\end{theorem}

It is well known that the blow-up of a compact K\"ahler manifold along a complex submanifold is K\"ahler; see \cite[Theorem II.6, p.~202]{Blanchard1956} or \cite[Proposition 3.24, p.~80]{Voisin2007_1}. More precisely, for sufficiently small $\epsilon>0$, the class
\(
\sigma^{*}[\omega]-\epsilon^{2}[E]
\)
is K\"ahler and can be represented by an explicit background metric $\omega_{\epsilon}$ on $Bl_{X}M$, defined in \eqref{backmetric3} and \eqref{backmetric2}. Roughly speaking, Theorem \ref{MainThm_intro} shows that the scalar curvature on $Bl_{X}M$ can be arbitrarily close to that of $M$. It is motivated by an important recent result of Brown \cite{Brown}, where he proved the same conclusion for the blow-up of a compact K\"ahler manifold at one point.

Following \cite{Brown}, our approach to Theorem \ref{MainThm_intro} is inspired by the gluing method for constructing constant scalar curvature K\"ahler (cscK) metrics and extremal K\"ahler metrics on blow-ups of compact K\"ahler manifolds. For related results on the gluing construction, we refer to \cite{AP2006,AP2009,APS2011,Sze2012,Sze2014,BR2015,SS2020} and the references therein. In particular, the gluing construction of extremal metrics on blow-ups of compact K\"ahler manifolds along submanifolds of codimension \(k\geq 3\) was carried out in \cite{SS2020}. The key observation in \cite{Brown} is that the construction of a K\"ahler metric with $S>0$ can be reduced to solving a PDE closely related to the cscK equation. Roughly speaking, with respect to the background metric $\omega_{\epsilon}$ on $Bl_{X}M$, we solve for $\phi$ in the equation
\begin{equation}\label{main_eq_intro}
S(\omega)-S\bigl(\omega_{\epsilon}+\sqrt{-1}\,\partial\dbar{\phi}\bigr)=\mathcal{E}_{\epsilon}.
\end{equation}
Here the error term $\mathcal{E}_{\epsilon}$ converges to zero, in an appropriate norm, as $\epsilon\to 0$. We follow \cite{Brown} and solve for $\phi$ in \eqref{main_eq_intro} by applying the contraction mapping principle in suitable weighted H\"older spaces. Since the weighted norm of the solution $\phi$ to \eqref{main_eq_intro} is bounded by a suitable positive power of $\epsilon$, the desired metric $\widetilde{\omega}=\omega_{\epsilon}+\sqrt{-1}\partial\dbar{\phi}$ in Theorem \ref{MainThm_intro} can be viewed as a small perturbation of $\omega_{\epsilon}$.

Notably, our method differs from the approaches of \cite{SS2020} and \cite{Brown} in the following aspects.

\begin{enumerate}[label=(\Roman*)]

   \item   In the case of a blow-up at one point, the background metric $\omega_{\epsilon}$ is constructed using holomorphic normal coordinates and a cut-off technique. For blow-ups along a submanifold, a similar construction of $\omega_{\epsilon}$ can be carried out; see \cite[p.~174]{SS2020} for the case $k\geq 3$. We also define an analogous metric $\omega_{\epsilon}$ in the case $k=2$. Since the distance function to the submanifold enters the construction, the corresponding holomorphic ``normal'' coordinates, as in \cite[Lemma 5, p.~175]{SS2020} or Lemma \ref{holosubcoor}, have weaker smoothness properties. Consequently, the behavior of $\omega_{\epsilon}$ requires careful estimates; see Lemmas \ref{MC3} and \ref{MC2}.

   \item  In \cite{Brown}, the contraction property related to \eqref{main_eq_intro} is established by studying the adjoint of a modified linear operator $\widetilde{L}_{\omega_{\epsilon}}$ arising from the scalar curvature function; see \cite[Lemma 4.1]{Brown}. It is unclear whether the similar strategy applies to blow-ups along a submanifold. Instead, we work directly with $\widetilde{L}_{\omega_{\epsilon}}$ and solve \eqref{main_eq_intro} using a different choice of the error term $\mathcal{E}_{\epsilon}$. A crucial step is to show the uniform bound for the inverse of $\widetilde{L}_{\omega_{\epsilon}}$; see Propositions \ref{Inverse} in the case $k \geq 3$. We prove this bound by a scaling argument, together with removable singularity results for higher-order elliptic PDEs. When $k \geq 3$, the proof of the related result \cite[Proposition 9, p.~181]{SS2020} relies on a Liouville theorem for certain fourth-order linear equations on the product space $Bl_{0}\mathbb{C}^k \times \mathbb{C}^{m-k}$. By contrast, our argument avoids the need for such a Liouville theorem through an appropriate choice of scaling parameters.

   \item   On the blow-up of a K\"ahler surface at a point, Brown \cite{Brown} replaces the background metric \(\omega_{\epsilon}\) appearing in \eqref{main_eq_intro} with a modified metric \(\widetilde{\omega}_{\epsilon}\). His argument is based on the corresponding analysis for the cscK problem; see \cite[Section 8.4]{Sze2014}. In the $k=2$ case of Theorem \ref{MainThm_intro}, we adopt a new and possibly conceptually simpler approach: we introduce a different weighted H\"older space (Definition \ref{WeightedBLxM2}) and modify the error term $\mathcal{E}_{\epsilon}$ in \eqref{main_eq_intro} to cancel certain unfavorable asymptotic terms. We refer to \eqref{Equation2} for the precise from of the new equation. Using the newly introduced weighted H\"older norms, we show in Proposition \ref{Inverse} a uniform bound for the inverse of $\widetilde{L}_{\omega_{\epsilon}}$. To solve \eqref{main_eq_intro}, we further prove uniform estimates for the error term $\mathcal{E}_{\epsilon}$ on the right-hand side; see Proposition \ref{F}.

   \item   Our approach to Theorem \ref{MainThm_intro} can be applied to give an alternative proof of the main result in \cite{Brown}. This point may be of particular interest in the case of blow-ups of a surface at a point, in light of the remark following \cite[Corollary 4.14, p.~12]{Brown}. Indeed, on the blow-up of a K\"ahler surface at a point, we may find a new K\"ahler metric in the same K\"ahler class as the metric constructed in \cite{Brown}, each of whose scalar curvature is $C^{0}$ close to $\sigma^{\ast}(S(\omega))$.

\end{enumerate}

For more recent developments on extremal K\"ahler metrics on blow-ups, we refer the reader to \cite{DS2021,Hallam, Naja2025,BJT2026}.

\subsection{Relations to recent work on positive scalar curvature}

In this subsection, we discuss how Theorem \ref{MainThm_intro} relates to compact K\"ahler manifolds with positive scalar curvature ($S>0$ for short).

The K\"ahler geometry of uniruled surfaces were first studied by Yau and Hitchin in 1970s. For example, Yau \cite{Yau1974} proved that any compact K\"ahler surface of positive total scalar curvature is uniruled. Recall that the total scalar curvature of a compact K\"ahler manifold $(M^m, \omega)$ (see \cite{Yau1974}) is defined as
\begin{equation}
    \int_{M}S(\omega) \frac{\omega^m}{m!}=\frac{2\pi}{(m-1)!} \int_{M}c_{1}(M)\wedge \omega^{m-1}.
\end{equation}
It is an invariant of a given K\"ahler class. Hitchin \cite{Hitchin1975} proved that a compact K\"ahler surface of positive total scalar curvature is rational if it is regular, that is, if $h^{0,1}=0$. On the other hand, it is proved in \cite{Yau1974} that any uniruled K\"ahler surface admits a K\"ahler metric of positive total scalar curvature, and any minimal ruled K\"ahler surface admits a K\"ahler metric $S>0$. In complex dimension $m \geq 3$, Hitchin proved that the blow-up of any compact K\"ahler manifold $M^m$ with $S>0$ at one point admits a K\"ahler metric with $S>0$. He also showed that the same result holds for some particular rational surfaces. These surfaces include $\mathbb{CP}^{2}$ with the standard Fubini--Study metric and a Hirzebruch surface
\(F_k=\mathbb{P}\bigl(\mathcal{O}(k)\oplus \mathcal{O}_{\mathbb{P}^{1}}\bigr)\) ($k \geq 1$)
equipped with a suitably chosen $U(2)$-invariant K\"ahler metric. As an important corollary of his main result in \cite{Brown}, Brown proved that any uniruled K\"ahler surface admits a K\"ahler metric with $S>0$. Therefore we have a precise connection between uniruled surfaces and the existence of K\"ahler metrics with $S>0$. We also note that the recent result of Zhang \cite{Zhang2026} that any rational surface admits a K\"aher metric with positive holomorphic sectional curvature.

The second-named author learned from \cite{Sun_PersonalComm} that it is natural to investigate whether $S>0$ is preserved after birational transformations. Theorem \ref{MainThm_intro}, focusing on a typical case of a bimeromorphic transformation. shows that $S>0$ is preserved after blow-up along complex submanifolds. After completing our work, we learned of the very recent arXiv preprints by Sha \cite{Sha2026_1, Sha2026_2}. He proves that on a compact K\"ahler manifold with a K\"ahler class $[\alpha]$ of positive total scalar curvature, either there exists no K\"ahler metrics with $S>0$ in $[\alpha]$ or the space of K\"ahler metrics with $S>0$ in $[\alpha]$ is path-connected. He constructs examples such that the former case indeed could happen. He also shows that the existence of K\"ahler metrics with $S>0$ in $[\alpha]$ is equivalent to some kind of stablility condition, we refer to \cite[p.~2]{Sha2026_2} for a precise definition. We observe that, under the assumptions of Theorem \ref{MainThm_intro}, the K\"ahler class $\sigma^{*}[\omega]-\epsilon^{2}[E]$ has the same sign of total scalar curvature as the class $[\omega]$ on $M$ for any sufficiently small $\epsilon>0$. This observation follows from a calculation reminiscent of the proof of \cite[Proposition 3]{Yau1974}. However, we do not know whether one can verify $\sigma^{*}[\omega]-\epsilon^{2}[E]$ satisfies the stability condition formulated in \cite{Sha2026_2} directly. This might be possible in view of \cite{BJT2026}. If so, it leads to another proof that $S>0$ is preserved after blowing up along submanifolds. In any case, we want to emphasize that the gluing method in the proof of Theorem \ref{MainThm_intro} shows that a K\"ahler metric with $S>0$ on the blow-up can be constructed as a perturbation of the background metric $\omega_{\epsilon}$ in \eqref{backmetric3} and \eqref{backmetric2}.

\begin{comment}
To see this, recall from \cite[p.~608]{GH} that the first Chern class of $Bl_XM$ satisfies
\begin{equation*}
c_1(Bl_XM)=\sigma^*c_1(M)-(k-1)[E].
\end{equation*}
Then we solve the total scalar curvature $\sigma^*[\omega]-\epsilon^2[E]$ as follows.
\begin{equation}\label{total_S_codim}
    \begin{split}
        &\int_{Bl_{X}M}\Big( \sigma^{*}c_{1}(M)-(k-1)[E] \Big)\wedge \Big( \sigma^{*}[\omega]-\epsilon^2 [E] \Big)^{m-1}\\
        =& \int_{M}c_{1}(M)\wedge[\omega]^{m-1}+(-1)^{m-1}\epsilon^{2(m-1)}\int_{Bl_{X}M} \Big(\sigma^{*}c_{1}(M) \wedge [E]^{m-1}-(k-1)[E]^{m}\Big)\\
        &+\sum_{l=1}^{m-2}C_{m-1}^{l}(-1)^{l} \epsilon^{2l}\int_{Bl_{X}M}\Big( \sigma^{*}c_{1}(M)\wedge (\sigma^{*}[\omega])^{m-1-l}\wedge [E]^{l}-(k-1)(\sigma^{*}[\omega])^{m-1-l}\wedge [E]^{l+1} \Big)\\
        =& \int_{M}c_{1}(M)\wedge[\omega]^{m-1}+O(\epsilon^{2(k-1)}).
    \end{split}
\end{equation} Here we use $[E]|_E=c_1(\mathcal{O}_E(-1))$ and $\int_E \sigma^{\ast} \alpha \wedge (c_1(\mathcal{O}_E(1)))^r=0$ for $\alpha \in \Omega^{(m-1-r, m-1-r)}(X)$  and $0 \leq r \leq k-2$. The above calculation is reminiscent of the proof of \cite[Proposition 3]{Yau1974}, where Yau considered the total scalar curvature of blow-ups of compact K\"ahler surfaces at one point.
\end{comment}

\subsection{Further discussions}

In this subsection, we discuss the behavior of scalar curvature under operations in K\"ahler geometry related to blow-ups.

Given a compact complex manifold $\widehat{M}$, there is a basic complex-analytic criterion \cite{Nakano1970,FN1971} for recognizing when it can be obtained by blowing up another complex manifold $M$ along a complex submanifold. It is well-known that $M$ does not need to K\"ahler if $\widehat{M}$ is a K\"ahler manifold of complex dimension $\geq 3$. We list three examples.
\begin{enumerate}[label=(\arabic*)]
    \item  The $3$-fold constructed in \cite[Example 2 on p.~504]{Hiro1975} is non-K\"ahler. 
    
    \item  $M$ is a small resolution of generic quintic threefold $Z$ with an ordinary double point in \cite{Clemens1983}. The exceptional set $F \cong \mathbb{P}^1$ can be represented by a zero class in $H^2(Y, \mathbb{C})$. Hence $M$ is non-K\"ahler. Let $\widehat{M}$ be the blow up of $Z$ at the double point. Note that the corresponding exceptional divisor $E=\mathbb{P}^1 \times \mathbb{P}^1 \subset \widehat{M}$. We may define a blow down map $\sigma: \widehat{M} \rightarrow M$ contracts along one of two rulings of $E$. 
    
    \item  See \cite[p.~144]{BT2000} for an example in which $\widehat{M}$ is Fano and $M$ is Moishezon. 
\end{enumerate}

\begin{comment}
 \begin{proposition}[{\cite{Nakano1970,FN1971}}]\label{blow_fact}
Let $\widehat{M}$ be a complex manifold and let
$E\subset \widehat{M}$ be a compact smooth hypersurface. Assume that
there exists a holomorphic fiber bundle
\(
p:E\longrightarrow F
\)
over a complex manifold $F$, whose fibers are isomorphic to
$\mathbb{P}^{r-1}$ with $r\geq 2$. Assume that
\[
\left.\mathcal{O}_{\widehat{M}}(E)\right|_{p^{-1}(x)}
\cong \mathcal{O}_{\mathbb{P}^{r-1}}(-1),\ \ \ \text{for every}\ x\in F.
\]
Then there exist a complex manifold $M$, containing $F$ as a complex
submanifold of codimension $r$, and a surjective holomorphic map
\(
\pi:\widehat{M}\longrightarrow M
\)
such that:
\begin{enumerate}[label=(\roman*)]
    \item $\pi|_E=p$ and $\pi(E)=F$;
    \item $\pi:\widehat{M}\setminus E\longrightarrow M\setminus F$
    is biholomorphic;
    \item $\widehat{M}$ is biholomorphic to $Bl_FM$, and $\pi$
    corresponds to the blow-down map.
\end{enumerate}
\end{proposition}   

According to \cite{Hiro1975}, if \(X\) and \(Y\) are bimeromorphic compact complex manifolds, then there exists a compact complex manifold \(Z\) which dominates both \(X\) and \(Y\) by proper modifications.
\end{comment}

The blowing up and down along a smooth submanifold are standard examples of a proper modification or a bimeromorphic map. According to the weak factorization theorem in \cite{AKMW}, any bimeromorphic map is a composition of a finite number of blow up and blow down maps along smooth submanifolds. There are many studies on algebraic and geometric properties that are preserved after a blow up/down map or a modification. We list a few examples below.
\begin{enumerate}[label=(\arabic*)]
    \item  Given a proper modification $\sigma: \widehat{M} \rightarrow M$ between two complex manifolds. If $\widehat{M}$ satisfies the $\partial \bar{\partial}$-lemma, so does $M$; see \cite[Theorem 5.22]{DGMS} and \cite[Theorem 12.9, Chapter VI]{Demailly_book}.
    \item  The existence of balanced metrics is preserved under proper modifications (Alessandrini-Bassanelli \cite{AB1993}, \cite{AB1995}, and \cite{AB1996}).
    \item  The existence of a balanced metric with positive total scalar curvature is a bimeromorphically invariant property (Chiose-R{\u a}sdeaconu-{\c S}uvaina \cite{CRS2019}).
\end{enumerate}

\begin{comment}
        \item  The property that the canonical line bundle is not pseudo-effective is preserved in a blow up/down map.
\end{comment}

In view of results in \cite{CRS2019} and Theorem \ref{MainThm_intro}, we propose the following general question concerning the behavior of scalar curvature under a blow-down map.

\begin{question}\label{blow_down_intro}
    Given a compact manifold $M$ of complex dimension $m$ with a complex submanifold $X$, and $\sigma: \widehat{M}\coloneqq Bl_{X}M \rightarrow M$ the blowing up map. If $\widehat{M}$ admits a K\"ahler metric $\widetilde{\omega}$, Let $\mathcal{B}_{M}$ be the balanced cone on $M$; see, for example, \cite[p.~340]{CRS2019} for the definition, and $S(\omega)$ denote the Chern scalar curvature of a balanced metric $\omega$ with $\omega^{m-1} \in \mathcal{B}_{M}$. Can we estimate $\inf_{\omega^{m-1} \in \mathcal{B}_{M}}\|S(\widetilde{\omega})-\sigma^{\ast} (S(\omega))\|_{C^0(Bl_{X} M)}$? In particular, if $S(\widetilde{\omega})>0$, does there exist a balanced metric $\omega$ on $M$ with $S(\omega)>0$?
\end{question}

Returning to compact K\"ahler manifolds with $S>0$. There has been significant progress in understanding the connection between uniruledness and K\"ahler manifolds with $S>0$ in higher dimensions. Heier--Wong \cite{HW2012} showed that the canonical bundle $K_M$ of any projective manifold $M$ admitting a K\"ahler metric with positive total scalar curvature is not pseudo-effective. Hence it is uniruled by a result of Boucksom--Demailly--P{\u a}un--Peternell \cite{BDPP2013}, As the recent work of Ou \cite{Ou} generalizes the result in \cite{BDPP2013} to compact K\"ahler manifolds, we know that any compact K\"ahler manifolds with positive total scalar curvature is uniruled. On the other hand, Yang \cite{Yang2019} and Brown \cite{Brown} proposed the following metric characterization of compact uniruled K\"ahler manifolds.

\begin{conjecture}[{\cite[Conjecture 4.7]{Yang2019}} and {\cite[Conjecture 7.4]{Brown}}]\label{brown_conj}
Does a compact uniruled K\"ahler manifold admits a K\"ahler metric with positive (total) scalar curvature? 
\end{conjecture}

\begin{comment}
Answers to these questions would deepen our understanding of compact K\"ahler manifolds with $S>0$. Together with Theorem \ref{MainThm_intro}, an affirmative answer to Question \ref{blow_down_intro} would suggest that K\"ahler or balanced metrics with $S>0$ may be constructed via bimeromorphic transformations.    
\end{comment}

A possible strategy for proving Conjecture \ref{brown_conj} is as follows. Any compact uniruled K\"ahler manifold is expected to be associated with some Mori fiber space under a finite number of bimeromorphic transformations. We refer to \cite{KM1998} for the definition and basic properties of Mori fiber spaces. Roughly speaking, a Mori fiber space is a normal variety $X$ with mild singularity and a morphism $f: X \rightarrow B$ so that $K_X$ is $f$-ample (positive in the direction of fibers). Then we study whether any Mori fiber space arising in this way admits a (possibly singular) K\"ahler or balanced metrics with $S>0$. Our hope is to construct a K\"ahler metric with $S>0$ on the original uniruled K\"ahler manifold, starting from a K\"ahler metric with $S>0$ on its Mori fiber space. At present, the conjecture remains open, and we do not have a detailed approach. Instead, following an idea of \cite{Yau1974} (see Remark \ref{Rem_OnYau}), we observe the following result for smooth Fano fibrations.

\begin{proposition}\label{Fanofib}
Let $f: X \rightarrow B$ be a holomorphic submersion between compact complex manifolds. Assume that each fiber $X_b \coloneqq f^{-1}(b)$ of a point $b \in B$ is a Fano manifold. If $B$ admits a K\"ahler metric, then $X$ admits a K\"ahler metric with $S>0$. 
\end{proposition}

We also refer the reader to \cite{Sha2026, Tsiamis2026} for recent progress on the systolic geometry of compact K\"ahler manifolds with $S>0$.

The structure of the paper is the following. In Section \ref{section2} we introduce the background metric on the blow-up of a compact K\"ahler manifold. Weighted H\"older spaces and the corresponding linearized operators are defined in Sections \ref{section3} and \ref{section4}. In particular, Propositions \ref{Inverse}, which are crucial for setting up the contraction mapping needed in the proof of Theorem \ref{MainThm_intro}, are proved in Section \ref{section4}. Finally, we complete the proof of Theorem \ref{MainThm_intro} in Sections \ref{section5} and \ref{section6}. Proposition \ref{Fanofib} is proved in Section \ref{section7}. The relevant estimates are collected in Appendices \ref{holo_nor_app}--\ref{App_D}.

\section{A background metric and related estimates}\label{section2}

\subsection{A background metric}
Let $Bl_0 \mathbb{C}^k$ denote the blow up of $\mathbb{C}^k$ at the origin. Recall that the Burns-Simanca metric is a complete scalar ﬂat K\"ahler metric on $Bl_0 \mathbb{C}^k$ constructed by Burns \cite{Lebrun1988} and Simanca \cite{Simanca}. We consider a suitable scaling of the Burns-Simanca metric. Let $\eta$ denote its corresponding K\"{a}hler form and $z=(z_1, \cdots, z_k)$ a holomorphic coordinate on $\C^{k}$. Then for $z \in \C^{k} \setminus \{0\}$, $\eta$ satisfies
\begin{equation}
    \eta=\begin{cases}
    \sqrt{-1}\di\dbar(|z|^{2}+\log(|z|^{2})), \   \ \ \text{if}\ k=2, \\
   \sqrt{-1}\di\dbar\Big(|z|^{2}+\gamma(|z|)\log |z|^{2}+\psi(|z|^{2})\Big), \  \   \text{if}\ k\geq 3.   
   \end{cases}  \label{BSdim23}
\end{equation}
Here $\psi \in C^{\infty}[0,\infty)$ satisfies
\begin{equation}
    \psi(t)= -c_1 t^{2-k}+c_2 t^{1-k}+O(t^{3-2k}), \ t\rightarrow{\infty},\  k\geq 3 \label{psidef}
\end{equation} for two positive constants $c_1$ and $c_2$ which depend on $k$. Moreover, $\gamma$ is a smooth function on $[0, +\infty)$ such that $\gamma(t)=1$ if $t<1$ and $\gamma(t)=0$ if $t>2$. One may check that $\eta$ extends across the exceptional divisor and defines a complete K\"ahler metric on $Bl_0 \mathbb{C}^n$. We follow Seyyedali--Sz\'ekelyhidi \cite[p.~174]{SS2020} in writing the particular form of $\eta$ in \eqref{BSdim23}; see Sz\'ekelyhidi \cite[Lemma 26]{Sze2012} for a more refined asymptotic expansion of the function $\psi$ in \eqref{psidef}.

Let $(M, \omega)$ be a compact K\"ahler manifold of complex dimension $m$, and let $X \subset M$ be a complex submanifold of complex codimension $k$. For short, we write $\codim X=k$. Let $d$ denote the distance from a point $q\in M$ to $X$ with respect to $\omega$, namely,
\begin{equation}
    d(q)=\inf_{p\in X} d_{\omega}(p,q).
    \label{disfuc}
\end{equation}

In the case of $k \geq 3$, a background metric $\omega_{\epsilon}$ on $M \setminus X$ is introduced in \cite[p.~174]{SS2020}. It is obtained by gluing the scaled potential of the Burns-Simanca metric to $\omega$, namely,
\begin{equation}\label{backmetric3}
\omega_{\epsilon}=
\omega+\epsilon^{2}\sqrt{-1}\di\dbar\Big(\gamma(\frac{d}{r_{\epsilon}})\Big[\gamma(\frac{d}{\epsilon})\log (\frac{d^{2}}{\epsilon^{2}})+\psi\bigl(\frac{d^{2}}{\epsilon^{2}}\bigr)\Big]\Big), 
\end{equation}
where we fix $r_{\epsilon}=\epsilon^{\frac{k}{k+1}}$.

If $k=2$, we define a background metric on $M\setminus X$ by
\begin{equation}\label{backmetric2}
    \omega_{\epsilon}=\omega+\epsilon^{2}\sqrt{-1}\di\dbar\Big( \gamma\bigl(\frac{d}{\epsilon^{\theta}}\bigr)\log\bigl(\frac{d^{2}}{\epsilon^{2}}\bigr) \Big), 
\end{equation}
where $\theta$ is any value in $(0, \frac{1}{2})$. The forms of \eqref{backmetric3} and \eqref{backmetric2} are motivated by the background metric used in the blow-up at a point; see \cite[pp.~158 and 175]{Sze2014}. Note that in the case $k \geq 3$, our choice of $r_{\epsilon}$ is slightly different from that in \cite[p.~174]{SS2020}. When $k=2$, we also make a specific choice of $\theta$, which will play a crucial role in Section \ref{section6}.

\begin{lemma}[{\cite[Proposition 4]{SS2020}}]\label{SSepsilon}
If $k \geq 3$, then, for all sufficiently small $\epsilon>0$, the form $\omega_{\epsilon}$ defined in \eqref{backmetric3} extends to a smooth K\"ahler metric on $Bl_{X}M$.
\end{lemma}

In Lemma \ref{MC2} we establish asymptotic estimates on $\omega_{\epsilon}$ given by (\ref{backmetric2}). In particular, we have the following result.

\begin{lemma}\label{MC2_half}
If $k=2$, for any sufficiently small $\epsilon>0$, $\omega_{\epsilon}$ defined in (\ref{backmetric2}) extends to a smooth K\"ahler metric on $Bl_{X}M$.
\end{lemma}

One readily checks that, in both cases $k\geq 3$ and $k=2$, the corresponding K"ahler metric $\omega_{\epsilon}$ on $Bl_{X}M$ lies in the K\"ahler class $\sigma^{*}[\omega]-\epsilon^{2}[E]$.

\subsection{Notations}\label{subsec_nota}

We introduce some notations which will be used throughout the paper. 

\begin{enumerate}[label=(\Roman*)]
    \item  Given $\sigma:Bl_{X}M\rightarrow{M}$ the blow-up map and $E=\sigma^{-1}(X)$. Let 
    \[
    (U,(z^{1},...,z^{k},w^{1},...,w^{m-k}))
    \]
     be a holomorphic coordinate chart containing a point $p \in X$, chosen so that 
     \[
     U\cap X=\{z^{1}=\cdots=z^{k}=0\}.
     \]
    For any fixed $1 \leq i \leq k$, let $(U_{(i)},(\zi, w))$ denote the corresponding affine chart of the blow-up associated with the coordinate $z_i$. In this chart, the blow-up map is given by
\begin{equation}
    \sigma(\zi^{1}, \cdots, \zi^{k},w)=(\zi^{i}\zi^{1}, \cdots, \zi^{i}\zi^{i-1},\zi^{i},\zi^{i}\zi^{i+1}, \cdots, \zi^{i}\zi^{k},w)=(z,w).   \label{BLcoord}
\end{equation}
In particular, the exceptional divisor is locally given by \(E\cap U_{(i)}=\{\zi^{i}=0\}\).

\item By $f=O(g)$ we mean that there exist a sufficiently small $\epsilon_{0}>0$ and a constant $C>0$ independent of any $\epsilon \in (0, \epsilon_{0})$, such that $|f| \leq  C|g|$ holds under additional asymptotic conditions. Those conditions will be made clear in the context. 

\item At every point of $X \subset M$, there exist holomorphic ``normal'' coordinates centered at that point. Such coordinates are constructed in \cite[Lemma 5]{SS2020} and will be used throughout the paper. For the convenience of the reader, we recall their result in Lemma \ref{holosubcoor} and include a detailed proof in Appendix \ref{holo_nor_app}. We introduce the multi index $I=(I_{1},...,I_{k},\ol{J_{1}},...,\ol{J_{l}})$ with $|I|=k+l$. With respect to `normal' coordinates in Lemma \ref{holosubcoor}, derivatives of the distance function $d$ (defined in (\ref{disfuc})) can be estimated. When $|z|+|w|$ is sufficiently small, we have the following
\begin{align}
    &\di_{w}^{K} d (z,w)=|z|\di_{w}^{K}\sqrt{1+\rho (z,w)} = O(d),  \label{diff_d_w}\\
    &\di_{z}^{I}d(z,w)=\begin{cases}
      (\di_{z}^{I}|z|) \sqrt{1+\rho(z,w)}+|z|\di_{z}^{I}\sqrt{1+\rho(z,w)}=O(1), \ |I|=1,\\
   O\Big(\sqrt{1+\rho(z,w)}\di_{z}^{I}(|z|)\Big)=O \Big( \frac{1}{|z|^{|I|-1}}\Big), \ |I|\geq 2.  
    \end{cases}    \label{diff_d_z}
\end{align}
Due to the above relation, without additional statement, we always take $|z|+|w|$ sufficiently small, such that $\frac{|z|}{2}\leq d\leq \frac{3|z|}{2}$ holds in the chart.

\item According to (\ref{psidef}), we may take $R_{0}>>2$ such that the following holds
\begin{equation}\label{BSpotential}
    \psi(\frac{d^{2}}{\epsilon^{2}})=-c_1(\frac{d^2}{\epsilon^2})^{2-k}+O\Big(\big(\frac{d^2}{\epsilon^2}\big)^{1-k}\Big),  \ \ d \geq R_{0}\epsilon.
\end{equation}

\end{enumerate}

\subsection{Related estimates on the background metric}\label{rel_more_est}

In this subsection, we state related estimates on the comparison of the background metric $\omega_{\epsilon}$ with several model metrics.

Following \cite[Proposition 4]{SS2020}, consider a model space $(Bl_{0}\C ^{k}\times \C ^{m-k},\omega_{\eta})$, where $\omega_{\eta}=\eta+\omega_{euc}$ is the product of the Burns-Simanca metric $\eta$ on $Bl_{0}\C ^{k}$ and the Euclidean metric on $\C ^{m-k}$. Let $\omega_{\eta, \epsilon}=\eta_{\epsilon}+\omega_{euc}$ denote the scaled product metric 
\begin{equation}  \label{ScaledPMetric}
\omega_{\eta,\epsilon}=
\begin{cases} 
\sqrt{-1}\,\di\dbar\Big(|z|^{2}+|w|^{2}+\epsilon^{2}[\gamma(\frac{|z|}{\epsilon})\log(\frac{|z|^{2}}{\epsilon^{2}})+\psi(\frac{|z|^{2}}{\epsilon^{2}})] \Big), & (z,w)\in \mathbb{C}^{k} \setminus
\{0\}\times \mathbb{C}^{m-k},\quad k\geq 3, \\[0.5em]
\sqrt{-1}\,\di\dbar\Big(|z|^{2}+|w|^{2}+\epsilon^{2} \log(\frac{|z|^{2}}{\epsilon^{2}})\Big), & (z,w)\in \mathbb{C}^{2}\setminus\{0\}\times \mathbb{C}^{m-2}, \quad k=2.
\end{cases}
\end{equation}
Note that, $\eta_{\epsilon}$ is obtained from $\eta$ by the coordinate dilation
$z\longmapsto \frac{z}{\epsilon}$ followed by rescaling by a factor of $\epsilon^2$. In order to write down the local components of $\eta_{\epsilon}$. We choose a global holomorphic coordinate $Z=(Z^1, \cdots, Z^k)$ on $ \C^{k}\setminus\{0\}$ and write 
\begin{equation}
\eta=\sqrt{-1} (g_{\eta})_{i\overline{j}} dZ^i \wedge d\overline{Z}^j,\ \ \text{where}\ \ (g_{\eta})_{i\ol{j}}(Z)=(g_{\eta})(\frac{\partial}{\partial Z^{i}},\frac{\partial}{\partial \ol{Z}^{j}})(Z).
\end{equation}
It follows that
\begin{equation}
    \eta_{\epsilon}(z)=
        \sqrt{-1}(g_{\eta})_{i\overline{j}}\bigl(\dfrac{z}{\epsilon}\bigr)dz^{i}\wedge d\ol{z}^{j}, \ \ \ z=\epsilon Z, 
\end{equation}

The following estimates were established in \cite{SS2020}. For our purposes, we present them with suitable modifications.

\begin{lemma}[Adapted from {\cite[pp.~175--177]{SS2020}}]\label{MC3}
    When $k\geq 3$, the background metric $\omega_{\epsilon}$ satisfies the following estimates.

    When $R_{0}\epsilon\leq d\leq6r_{\epsilon}$,
    \begin{align}
        &\di_{z}^{I}\di_{w}^{J}(g_{\omega_\epsilon}-g_{\omega})=O(\epsilon^{2k-2}d^{2-2k-|I|}),\label{Compare1} \\
        &\di_{z}^{I}\di_{w}^{J}(g_{\omega_{\epsilon}}-g_{\omega_{\eta,\epsilon}})(z,w)=O(\di_{z}^{I}\di_{w}^{J}(|z|+|w|))+O\Big(\epsilon^{2k-2}|z|^{2-2k-|I|}(|z|+|w|) \Big).\label{Compare5}
    \end{align}

    When $d\leq 3R_{0}\epsilon$,
    \begin{align}
        &(g_{\omega_\epsilon}-g_{\omega_{\eta,\epsilon}})(\zi,w)=O(|\zi^{i}|(1+\sum_{j\neq i}|\zi^{j}|^{2})^{\frac{1}{2}}+|w|),\label{Compare3} \\
       &\di_{\zi}^{I}\di_{w}^{J}(g_{\omega_\epsilon}-g_{\omega_{\eta,\epsilon}})(\zi,w)=O(1)+O\Big(\epsilon^{-|I|}(|\zi^{i}|(1+\sum_{j\neq i}|\zi^{j}|^{2})^{\frac{1}{2}}+|w|)\Big).\label{Compare3.5}
    \end{align}

    For $R_{0}\epsilon<d<r_{0}$, with $r_{0}>0$ fixed and sufficiently small,
    \begin{equation}
    \di_{z}^{I}\di_{w}^{J}(g_{\omega_{\epsilon}}-g_{euc})(z,w)=O(\di_{z}^{I}\di_{w}^{J}(|z|+|w|))+O(\epsilon^{2k-2}d^{2-|I|-2k}).\label{Compare4}
    \end{equation}
    \end{lemma}
 
    By \eqref{Compare1}, we mean that there exist a constant $C>0$ and a sufficiently small $\epsilon_{0}>0$ such that, for every $\epsilon\in(0,\epsilon_{0}]$ and any fixed holomorphic coordinate chart $U$ as in Lemma \ref{holosubcoor}, we have
    \begin{equation}
        \sup_{U \cap \{R_0\epsilon<d<2r_{\epsilon}\}}\sup_{i, j}|\di_{z}^{I}\di_{w}^{J}(g_{\omega_\epsilon,i\ol{j}}-g_{\omega,i\ol{j}})|\leq C\epsilon^{2k-2}d^{2-2k-|I|}.
    \end{equation}

By (\ref{Compare1}) and (\ref{Compare3}), $\omega_{\epsilon}$ is a well-defined K\"{a}hler metric of $Bl_{X}M$. Detailed calculations on Lemma \ref{MC3} can be found in Appendix \ref{GE}. We state a similar result in the case $k=2$.

\begin{lemma}\label{MC2}

If $k=2$, $\omega_{\epsilon}$ defined in (\ref{backmetric2}) satisfies the following estimates.

When $d\leq \epsilon^{\theta}$,
\begin{equation}\label{BSCompare}
    \di_{z_{(i)}}^{I}\di_{w}^{J}(g_{\omega_{\epsilon}}-g_{\omega_{\eta,\epsilon}})(z_{(i)},w)=O\Big(\di_{z_{(i)}}^{I}\di_{w}^{J}(|z_{(i)}^{i}|+|w|) \Big)+O(\epsilon^{2}), \  \forall |I|,|J|\geq 0.
\end{equation}

When $\frac{1}{2}\epsilon^{\theta} \leq d \leq 2\epsilon^{\theta}$,
\begin{equation}\label{OriCompare}
    \di_{z}^{I}\di_{w}^{J}(g_{\omega_{\epsilon}}-g_{\omega})(z,w)=O\Big(\frac{\epsilon^{2}}{(\epsilon^{\theta})^{2+|I|}}\Big)\gamma\bigl(\frac{d}{\epsilon^{\theta}}\bigr)\log\frac{1}{\epsilon}, \ \forall |I|,|J|\geq 0.
\end{equation}

When $\epsilon \leq d \leq r_{0}$ for some fixed constant $r_{0}>0$ independent of $\epsilon$,
\begin{equation}\label{EucCompare}
    \di_{z}^{I}\di_{w}^{J}(g_{\omega_{\epsilon}}-g_{euc})(z,w)=O\Big(\di_{z}^{I}\di_{w}^{J}(|z|+|w|)\Big)+O\Big(\frac{\epsilon^{2}}{|z|^{2+|I|}}\Big)\gamma\bigl(\frac{d}{\epsilon^{\theta}}\bigr)\log\frac{1}{\epsilon}, \ \forall |I|,|J|\geq 0.
\end{equation}
\end{lemma}

\section{Analysis in weighted H\"{o}lder spaces}\label{section3}

According to Lemma \ref{MC3}, the background metric (\ref{backmetric3}) can be seen as a perturbation of certain model metrics in different regions in $Bl_{X}M$. In this section, we recall several weighted H\"{o}der norms for different model spaces that are introduced in \cite[Section 8.2]{Sze2014}. When $k\geq 3$, we use the weighted norm introduced in \cite[Section 3.1]{SS2020}. When $k=2$, we introduce a modified weight function in (\ref{WeightFunc2}) to define additional weighted H\"older spaces.

For a domain $\Omega\subset\C^{k}$, the usual $C^{l,\alpha}$ H\"{o}lder norm is denoted by:
\begin{equation}\label{holder1}
    \|f\|_{C^{l,\alpha}(\Omega)}=\sum_{|I|=0}^{l}\sup_{\Omega}|\di^{I}f|+\sup_{x,y\in\Omega}\frac{|\di^{I}f(x)-\di^{I}f(y)|}{|x-y|^{\alpha}}.
\end{equation}
By $f\in C^{l,\alpha}_{loc}(\Omega)$ we mean that $\|f\|_{C^{l,\alpha}(K)}$ is finite for any compact subset $K\subset \Omega$. On any compact manifold $M$, we choose a finite open cover of coordinate charts $\{U_{i} \}_{i=1}^{n}$ and define $C^{l,\alpha}(M)$ by
\begin{equation}\label{holder2}
    \|f\|_{C^{l,\alpha}(M)}=\sum_{i}\|f\|_{C^{l,\alpha}(U_{i})}.
\end{equation}
We may also define H\"older norms for tensors on $M$ following (\ref{holder2}).

\begin{definition}[{\cite
[p.~160]{Sze2014}}]\label{WeightedEuc}
    Let $k$, $l$ be positive integers, $\alpha\in(0,1)$, and $\delta\in\R$. The weighted H\"{o}lder space $C^{l,\alpha}_{\delta}(\C^{k}\setminus\{0 \})$ consists of all $f\in\C^{l,\alpha}_{loc}(\C^{k}\setminus\{0 \})$ with a finite
 \begin{equation}\label{holder3}
    \|f\|_{\C^{l,\alpha}_{\delta}(\C^{k}\setminus\{ 0\})}=\sup_{r>0}r^{-\delta}\|f(rz)\|_{\C^{l,\alpha}(B_{2}\setminus B_{1})}.
    \end{equation}
    Here $B_{r}$ is the ball centered at the origin with radius $r$.
\end{definition}

According to \cite[p.~160]{Sze2014}, $C^{l,\alpha}_{\delta}(\C^{k}\setminus\{0 \})$ with the weighted norm (\ref{holder3}) is a Banach space. Note that (\ref{holder3}) controls both the growth (or decay) rates at the origin and the infinity. Namely, for any $f \in C^{l,\alpha}_{\delta}(\C^{k}\setminus\{0 \})$, there exists a constant $C>0$ such that
\begin{equation}
    |\di^{I}f(z)|\leq C|z|^{\delta-|I|}, \ |z|\to\infty,\ \text{and}\  \ |z|\to 0.
\end{equation}

We recall the following mapping property of the Laplacian between weighted H\"older spaces.

\begin{theorem}[{\cite[Theorem 8.3]{Sze2014}}]\label{Laplace}
  If $\delta\notin \Z\setminus(2-2k,0)$, then
  \begin{equation}
      \Delta:C^{l,\alpha}_{\delta}(\C^{k}\setminus\{0 \})\longrightarrow C^{l-2,\alpha}_{\delta-2}(\C^{k}\setminus\{0 \})
  \end{equation}
  is an isomorphism.
\end{theorem}
We refer to \cite[pp.~160-163]{Sze2014} for a simplified proof and \cite[Theorem 1.7]{Bartnik} for a detailed proof of surjectivity. There is a more general result 
\cite[Theorem 7.4]{LMc1985} which deals with the Fredholm properties of higher order elliptic operators (systems) satisfying certain asymptotic behavior.

By a composition in Theorem \ref{Laplace}, we have the following result. 
\begin{corollary}[{\cite[Exercise 8.8]{Sze2014}}]\label{Iso_Bilap}
If $\delta\notin \Z\setminus(4-2k,0)$ for $k\geq 3$, and $\delta \notin \Z$ for $k=2$ respectively, then
\begin{equation}
    \Delta^{2}:C^{l,\alpha}_{\delta}(\C^{k}\setminus\{0 \})\longrightarrow C^{l-4,\alpha}_{\delta-4}(\C^{k}\setminus\{0 \}) 
\end{equation}
is an isomorphism.
\end{corollary}

A similar weighted space on $Bl_{0}\C^{k}$ is introduced in \cite[p.~162]{Sze2014} which focuses on growth (decay) at $\infty$. Fix a coordinate system $z\in \C^{k}\setminus\{0 \}$ and let $B_{r}^{k}$ denote the ($2k$-dimensional) Euclidean ball centered at the origin in $\mathbb{C}^k$. Let $\widetilde{B}_{r}^{k}=\sigma^{*}(B_{r}^{k})$ where $\sigma:Bl_{0}\C^{k} \rightarrow \C^{k}$ is the projection map. 
 \begin{definition}\label{WHS_Bl0}
     Let $l$ be positive integers, $\alpha\in(0,1)$, and $\delta\in\R$. The weighted H\"{o}lder space $C^{l,\alpha}_{\delta}(Bl_{0}\C^{k})$ is defined by all $f\in C^{l,\alpha}_{loc}(Bl_{0}\C^{k})$ with finite weighted norm:
     \begin{equation}
         \|f\|_{C^{l, \alpha}_{\delta}(Bl_{0}\C^{k})}=\|f\|_{C^{l,\alpha}(\widetilde{B}^{k}_{1})}+\sup_{r>1}r^{-\delta}\|f(rZ)\|_{C^{l,\alpha}(B^{k}_{2}\setminus B^{k}_{1})}.
     \end{equation}
     Recall that the term $\|f\|_{C^{l,\alpha}(\widetilde{B}^{k}_{1})}$ is defined by (\ref{holder2}).
 \end{definition}

The main goal of Section \ref{section3} is to define suitable weighted H\"older spaces on $Bl_{X}M$.
Note that for each $p \in X$, we choose a holomorphic coordinate chart $(U_p; (z, w))$ of the type described in Lemma \ref{holosubcoor}. Moreover, we may choose $r_{0}>0$ sufficiently small such that
\begin{equation}\label{choose_r_0}
\{(z, w)\mid \max(|z|, |w|) \leq 100 r_0\} \subset U_p,\ \ \text{and}\ \frac{2}{3}|z| \leq d(q) \leq \frac{3}{2}|z|,\  \text{for any}\ q \in U_p.    
\end{equation}
Then the metric tubular neighborhood $\{p\in M \mid d(p)<4r_{0}\}$ of $X$ in $M$ is covered by the family of open sets $\{U_{p}\}_{p\in X}$. We further choose $\epsilon$ sufficiently small so that $0<\epsilon<r_0$. The preimage of $\{p\in M \mid d(p)<4r_{0}\}$ under the blow-up map $\sigma\colon Bl_{X}M\longrightarrow M$ is covered by the blow-up coordinate charts defined in \eqref{BLcoord}. We henceforth \textbf{fix these coordinate charts and use them to define weighted H\"older norms} for functions and tensors. We begin by recalling the weight function introduced in \cite[p.~179]{SS2020}. When $k=2$, we also introduce an additional weight function.

\begin{definition}[Weight functions]
For the choice of $r_0$ which satisfies \eqref{choose_r_0}, we pick any $\epsilon<r_0$ and recall the weight function $\tau_{\epsilon}: \R \rightarrow \R$ defined in \cite[p.~179]{SS2020}. Note that $r_0$ is chosen to be $1$ in \cite{SS2020}. 
\begin{equation}\label{WeightFunc3}
    \tau_{\epsilon}(t)=
    \begin{cases}
        r_0, \ \ t\geq r_{0},\\
        t, \ \ \epsilon<t<r_{0},\\
        \epsilon, \ \ t\leq\epsilon.
    \end{cases}
\end{equation}
When $k=2$, we also introduce the following additional weight function. 
\begin{equation}\label{WeightFunc2}
    \ti{\tau}_{\epsilon}(t)=
    \begin{cases}
        1, \ \ \ \ \ t \geq [\log\frac{1}{\epsilon}]^{-1},\\
        t\log\frac{1}{\epsilon}, \ \ \ \epsilon<t<[\log\frac{1}{\epsilon}]^{-1},\\
        \epsilon\log\frac{1}{\epsilon}, \ \ \ \ t\leq\epsilon.
    \end{cases}
\end{equation}
Here we choose $\epsilon>0$ sufficiently small so that $\epsilon<[\log\frac{1}{\epsilon}]^{-1}<\frac{r_0}{2}$.

For $q\in M\setminus X$, define $\tau_{\epsilon}(q)=\tau_{\epsilon}(d(q))$ and $\widetilde{\tau}_{\epsilon}(q)=\widetilde{\tau}_{\epsilon}(d(q))$. Under the natural identification $Bl_{X}M\setminus E\cong M\setminus X$, both functions extend continuously across $E$ and hence belong to $C^{0}(Bl_{X}M)$.
\end{definition}

\begin{definition}[Based on {\cite[p.~179]{SS2020}}]\label{WeightedBLxM3}
Let $k \geq 2$ and $l \geq 0$ be two integers, $\alpha \in(0,1)$, and $\delta\in\R$. The weighted H\"{o}lder space $C^{l,\alpha}_{\delta}(Bl_{X}M,\omega_{\epsilon})$ consists of all functions $f\in C^{l,\alpha}_{loc}(Bl_{X}M)$ for which the following weighted norm is finite:
\begin{equation}\label{holder_bl_f}
\begin{split}
    \|f\|_{C^{l,\alpha}_{\delta}(Bl_{X}M, \omega_{\epsilon})}
    =&\|f\|_{C^{l,\alpha}(M\setminus \{d<r_{0} \})}+ \sup_{p \in X}\sup_{\epsilon\leq r \leq r_{0}}r^{-\delta}\|f(rZ_{p},rW_{p})\|_{C^{l,\alpha}((B_{2}^{k}\setminus B_{1}^{k})\times B_{2}^{m-k})}\\
    &+ \sup_{p \in X}\sum_{j=1}^{k}\epsilon^{-\delta}\|f(\epsilon Z_{(j),p},\epsilon W_{p})\|_{C^{l,\alpha}(\widetilde{B}_{(j),2}^{k}\times B_{2}^{m-k})}.
\end{split}
\end{equation} Here we define
\begin{equation}\label{def_B_j}
\widetilde{B}^k_{(j),1} \times B_{2}^{m-k}=\Big\{(Z_{(j)}, W) \in U_{(j)}\ \Big|\ |Z_{(j)}^j|^2(1+\sum_{1 \leq l \neq j }^{k} |Z_{(j)}^l|^2)<1,\ \ \sum_{p} |w_p|^2<4\Big\}.    
\end{equation}
We also introduce
\begin{equation}\label{holder_bl_f_0}
\|f\|_{C^{0}_{\delta}(Bl_{X}M, \omega_{\epsilon})}=\sup_{q \in Bl_{X}M}  (\tau_{\epsilon}(q))^{-\delta}|f(q)| 
\end{equation} where the weight function $\tau_{\epsilon}$ is defined in \eqref{WeightFunc3}. Note that, when $l=\alpha=0$, the norm in \eqref{holder_bl_f_0} is equivalent to the weighted norm in \eqref{holder_bl_f}.
\end{definition}

In the case of blow-ups at a point, the norm in \eqref{holder_bl_f} can be defined in a less coordinate-dependent way that more directly reflects the background metric $\omega_{\epsilon}$; see \cite[p.~1430]{Sze2012}. For the purposes of carrying out estimates, however, we adopt the formulation in \eqref{holder_bl_f}.

Let $g_{\omega_{\epsilon},i\ol{j}}$ denote the local coefficient of the metric $\omega_{\epsilon}$ in a fixed local coordinate chart. Then we introduce the following weighted norm of $g_{\omega_{\epsilon}}$ following (\ref{holder_bl_f}).
\begin{equation}
\label{WeightedTensor}
\begin{split}
    \|g_{\omega_{\epsilon}}\|_{C^{l,\alpha}_{\delta}(Bl_{X}M,\omega_{\epsilon})}
    =&\sum_{i, j}\Big(\|g_{\omega_{\epsilon},i\ol{j}}\|_{C^{l,\alpha}(M\setminus \{d<r_{0} \})}\\
    &+ \sup_{p \in X}\sup_{\epsilon \leq r \leq r_{0}}r^{-\delta}\|g_{\omega_{\epsilon},i\ol{j}}(rZ_{p},rW_{p})\|_{C^{l,\alpha}((B_{2}^{k}\setminus B_{1}^{k})\times B_{2}^{m-k})}\\
    &+ \sup_{p \in X}\sum_{j=1}^{k}\epsilon^{-\delta}\|g_{\omega_{\epsilon},i\ol{j}}(\epsilon Z_{(j),p},\epsilon W_{p})\|_{C^{l,\alpha}(\ti{B}_{(j),2}^{k}\times B_{2}^{m-k})}\Big).
\end{split}
\end{equation}
We also define
$|g_{\epsilon}^{-1}|_{C^{l,\alpha}_{\delta}(Bl_XM,\omega_{\epsilon})}$
by taking the corresponding weighted norm of the components
$g_{\omega_{\epsilon}}^{i\overline{j}}$ of the inverse metric. The norms
$\|g\|_{C^{l,\alpha}(M)}$ and $\|g^{-1}\|_{C^{l,\alpha}(M)}$ are defined analogously.

When the codimension is $k=2$, we also introduce an additional weighted H\"older norm associated with the weight function \eqref{WeightFunc2}.

\begin{definition}\label{WeightedBLxM2}
Let $k=2$, $l \in \mathbb{N}$, $t \in \mathbb{Z}$, $\alpha \in(0,1)$, and $\delta\in \R$. Along the submanifold $X$, we choose holomorphic coordinate charts as in Lemma \ref{holosubcoor}. The weighted H\"{o}lder space $\ti{C}^{l,\alpha}_{t,\delta}(Bl_{X}M,\omega_{\epsilon})$ consists of all functions $f \in C^{l,\alpha}_{loc}(Bl_{X}M)$ for which the following weighted norm is finite:
\begin{equation}\label{WeightedBLxM2_norm}
\begin{split}
    \|f\|_{\ti{C}^{l,\alpha}_{t,\delta}(Bl_{X}M,\omega_{\epsilon})}
    =\,&\|f\|_{C^{l,\alpha}(M\setminus \{d\leq r_{0} \})}\\
    &+\sup_{p\in X} \Big\{ 
    \sup_{\epsilon\leq r\leq r_{0}} 
    [\ti{\tau}(r)]^{-\delta}r^{-t}\|f(rZ_{p},rW_{p})\|_{C^{l,\alpha}(B_{2}^{2}\setminus B_{1}^{2}\times B_{2}^{m-2})} \\
    &+\sum_{j=1}^{k}[\ti{\tau}(\epsilon)]^{-\delta}\epsilon^{-t}
    \|f(\epsilon Z_{(j),p},\epsilon W_{p})\|_{C^{l,\alpha}(\ti{B}_{(j),2}^{2}\times B_{2}^{m-2})} 
    \Big\}.
\end{split}
\end{equation}
Similarly, we introduce a weighted norm equivalent to $\ti{C}^{0}_{0,\delta}$ in \eqref{WeightedBLxM2_norm}.
\begin{equation}\label{holder_bl_f_0_2}
\|f\|_{\widetilde{C}^{0}_{\delta}(Bl_{X}M, \omega_{\epsilon})}=\sup_{q \in Bl_{X}M}  (\widetilde{\tau}_{\epsilon}(q))^{-\delta}|f(q)|. 
\end{equation}The corresponding weighted norm for tensors are defined analogously to \eqref{WeightedTensor}.

\end{definition}

\begin{comment}

In practice, we restrict our attention to $f\in \ti{C}^{4,\alpha}_{0,\delta}$.  

linearized operators are then viewed as maps
$\ti{C}^{4,\alpha}_{0,\delta} \longrightarrow \ti{C}^{0,\alpha}_{-4,\delta}$.

\begin{remark}
   With the new weight function $\widetilde{\tau}_{\epsilon}$ used in Definition \ref{WeightedBLxM2}, any uniformly bounded sequence $\{\phi_{\epsilon}\} \in \ti{C}^{1,\alpha}_{0,\delta}(Bl_{X}M, \omega_{\epsilon})$ with respect to \eqref{WeightedBLxM2_norm} subconverges, as $\epsilon\to 0$, to a bounded function on $M\setminus X$.
\end{remark}

We point out that, when $\epsilon$ is fixed, the norm still gives polynomial control on the usual H\"older norm of order $r^{\delta}$ in local coordinates on $M \setminus X$. The difference is that, as $\epsilon\to 0$, any uniformly bounded sequence ${\phi_{\epsilon}}$ with respect to this norm is forced to converge, after passing to a subsequence, to a bounded function on $M\setminus X$.    
\end{comment}

Motivated by \cite[Lemma 8.12]{Sze2014}, we observe a uniform estimate for the weighted norm of components of the background metric. The proof is included in Appendix \ref{UUB}.

Recall that the curvature components of a K\"ahler metric $\widetilde{\omega}$ on $Bl_{X}M$ are given by
\begin{equation}
    R_{\ti{\omega},i\ol{j}k\ol{l}}=-\di_{k}\di_{\ol{l}} g_{\ti{\omega},i\ol{j}}+g_{\ti{\omega}}^{p\ol{q}}(\di_{i}g_{\ti{\omega}, k\ol{q}})(\di_{\ol{j}}g_{\ti{\omega}, p\ol{l}}).
\end{equation}
The weighted H\"older norm
\(
\|R_{\tilde{\omega}}\|_{C^{l,\alpha}_{\delta}(Bl_{X} M)}
\)
is defined analogously to \eqref{WeightedTensor}. 

\begin{proposition}\label{metric_Uni}
    There exists a constant $C>0$, such that for any $k\geq 2$ and $\epsilon\in(0,\epsilon_{0}]$ with $\epsilon_{0}$ sufficiently small, we have
    \begin{align}
        &\|g_{\omega_{\epsilon}}\|_{C^{2,\alpha}_{0}(Bl_{X}M)}, \ \|g_{\omega_{\epsilon}}^{-1}\|_{C^{2,\alpha}_{0}(Bl_{X}M)}\leq C.
    \end{align}
    As a consequence, we have the following estimates.
\begin{equation}
        \|R_{\omega_{\epsilon}}\|_{C^{0,\alpha}_{2}}, \ 
        \|Ric(\omega_{\epsilon})\|_{C^{0,\alpha}_{2}}, \ 
        \|S(\omega_{\epsilon})\|_{C^{0,\alpha}_{2}}\leq C.
    \end{equation}
\end{proposition}

Next, we study the perturbation of the background metric $\omega_{\epsilon}$ to $\omega_{\epsilon}+\sqrt{-1}\di\dbar\vphi$ for $\vphi$ in a suitable weighted H\"older space. 
Using Proposition \ref{metric_Uni}, we obtain the following uniform estimates, analogous to those in \cite[Lemma 8.13]{Sze2014} and \cite[Lemma 3.7]{Brown}. We refer to Appendix \ref{UEC} for the proof of Proposition \ref{metric_Uni}.

\begin{proposition}[Based on {\cite[Lemma 8.13]{Sze2014}} and {\cite[Lemma 3.7]{Brown}}]\label{UniEst}
    Fix $\delta<0$. There exists sufficiently small constants $c>0$, $\epsilon_{0}>0$ and a constant $C>0$, such that when $\vphi\in C^{4,\alpha}_{2}(Bl_{X}M)$ with $\|\vphi\|_{C^{4,\alpha}_{2}(Bl_{X}M)}\leq c$, 
    \(
    \omega_{\vphi,\epsilon}=\omega_{\epsilon}+\sqrt{-1}\di\dbar\vphi
    \)
    is a well defined K\"{a}hler metric on $Bl_{X}M$. Moreover, the following estimates holds for any $\epsilon \in (0,\epsilon_{0}]$.
    \begin{align}
        &\|g_{\omega_{\vphi,\epsilon}}\|_{C^{2,\alpha}_{0}}, \|g_{\omega_{\vphi,\epsilon}}^{-1}\|_{C^{2,\alpha}_{0}} \leq C, \\
        &\|g_{\omega_{\vphi,\epsilon}}-g_{\omega_{\epsilon}}\|_{C^{2,\alpha}_{0}(Bl_{X}M)}, \ \|g_{\omega_{\vphi,\epsilon}}^{-1}-g_{\omega_{\epsilon}}^{-1}\|_{C^{2,\alpha}_{0}(Bl_{X}M)}\leq C \|\vphi\|_{C^{4,\alpha}_{2}(Bl_{X}M)},   \label{metric_inv_est}\\
        & \|R_{\omega_{\vphi,\epsilon}}-R_{\omega_{\epsilon}}\|_{C^{0,\alpha}_{\delta-4}(Bl_{X}M)}\leq C\|g_{\omega_{\vphi,\epsilon}}-g_{\omega_{\epsilon}}\|_{C^{2,\alpha}_{\delta-2}(Bl_{X}M)} \leq C\|\vphi\|_{C^{4,\alpha}_{\delta}(Bl_{X}M)},  \label{curv_est} \\
        &\|L_{\omega_{\vphi,\epsilon}}-L_{\omega_{\epsilon}}\|_{C^{4,\alpha}_{\delta}(Bl_{X}M)\to C^{0,\alpha}_{\delta-4}(Bl_{X}M)}\leq C\|\vphi\|_{C^{4,\alpha}_{2}(Bl_{X}M)},  \label{opera_est} \\
        &\|R_{\omega_{\vphi,\epsilon}}-R_{\omega_{\epsilon}}\|_{\ti{C}^{0,\alpha}_{-4,\delta}(Bl_{X}M)}\leq C\|\vphi\|_{\ti{C}^{4,\alpha}_{0,\delta}(Bl_{X}M)},\ \ \text{when}\ k=2, \label{curv_est_2}\\
        &\|L_{\omega_{\vphi,\epsilon}}-L_{\omega_{\epsilon}}\|_{\ti{C}^{4,\alpha}_{0,\delta}(Bl_{X}M)\to \ti{C}^{0,\alpha}_{-4,\delta}(Bl_{X}M)}\leq C\|\vphi\|_{C^{4,\alpha}_{2}(Bl_{X}M)},\ \ \text{when}\ k=2.  \label{opera_est_2}
    \end{align} 
    
\end{proposition}

In \eqref{opera_est} and \eqref{opera_est_2}, $L_{\omega}$ denotes the linearized operator of the scalar curvature at a K\"{a}hler metric $\omega$ within a K\"{a}hler class. According to the proof of \cite[Theorem 4.2]{Sze2014}, we have
\begin{equation}\label{def_L}
    L_{\omega}\phi=\frac{d}{dt}\Big|_{t=0}S(\omega+t\sqrt{-1}\di\dbar \phi)=-\Delta_{\omega}^{2}\phi-Ric_{\omega}\cdot \di\dbar\phi.
\end{equation} We emphasize that the derivation of \eqref{def_L} is entirely local and therefore applies equally to a perturbation of a metric defined on an open subset.

\section{Estimates on linearized operators}\label{section4}
In this section, we consider the linearized operator of the scalar curvature given by \eqref{def_L}.
The main goal is to prove Proposition \ref{Inverse}, establishing the boundedness of the inverse of a modified operator.

\subsection{Mapping properties of linearized operators.}
In this subsection, we discuss mapping properties of different linearized operators in certain weighted spaces. They are mainly concerned with properties of the bi-Laplacian operators with respect to certain model metrics. We refer to \cite[Section 8.2]{Sze2014}, \cite[Section 2]{LMc1985} and \cite[Theorem 1.7]{Bartnik} for more information.

\begin{proposition}\label{Fredholm0}
The linearized operator
    \begin{equation}
        L_{\omega}:C^{4,\alpha}(M)\to C^{0,\alpha}(M)
    \end{equation}
    is a Fredholm operator with index $0$.
\end{proposition}

\begin{proof}[Proof of Proposition \ref{Fredholm0}]
   The simple observation is that $L_{\omega}$ is a compact perturbation  of $\Delta_{\omega}^{2}$. According to \cite[Theorem 2.13]{Sze2014},
   \begin{equation}
       \Delta_{\omega}:(\ker\Delta_{\omega})^{\perp}\cap C^{4,\alpha}(M)\to(\ker\Delta_{\omega}^{*})^{\perp}\cap C^{2,\alpha}(M)
   \end{equation}
   is an isomorphism. Hence, since $\Delta_{\omega}$ is $L^{2}$ self-adjoint and $\dim(\ker\Delta_{\omega})=1$, we have
   \begin{equation}
       \dim(\coker\Delta_{\omega})=\dim(\ker\Delta_{\omega}^{*})=1.
   \end{equation}
   Therefore, $\Delta_{\omega}:C^{4,\alpha}(M)\to C^{2,\alpha}(M)$ is a Fredholm operator with index $0$. Then, as a composition, $\Delta_{\omega}^{2}:C^{4,\alpha}(M)\to C^{0,\alpha}(M)$ is a Fredholm operator with index $0$.

   Consider a bounded linear operator
   \begin{equation}
    K:\,C^{4,\alpha}(M) \ni \phi \to Ric_{\omega}\cdot\di\dbar\phi \in C^{0,\alpha}(M).
   \end{equation}
   Take any bounded sequence $\|\phi_{i}\|_{C^{4,\alpha}(M)}\leq C$, we have
   \begin{equation}
       \|K\phi_{i}\|_{C^{2,\alpha}(M)}\leq C_{1}\|\phi_{i}\|_{C^{4,\alpha}(M)}\leq \ti{C}.
   \end{equation}
   Then, according to the Arzela-Ascoli Theorem (\cite[Theorem 2.7.]{Sze2014}), there exists a subsequence $\phi_{i_{k}}$ such that $K\phi_{i_{k}}$ converges in $C^{0,\alpha}(M)$. Hence, $K:C^{4,\alpha}(M)\to C^{0,\alpha}(M)$ is a compact operator. As a consequence, $L_{\omega}=-\Delta_{\omega}^{2}-K$ is a Fredhom operator with
   \(
   \ind(L_{\omega})=\ind(\Delta_{\omega}^{2})=0.
   \)
\end{proof}

For the $L^{2}$ formal adjoint operator $L^{*}_{\omega}$, a similar discussion will yield the same conclusion that $L^{*}_{\omega}$ is also a Fredholm operator with index zero, hence, it has a finite dimensional kernel. Note that $\operatorname{Ker} L_{\omega}$ is non-trivial as it contains constant functions. Following an observation in \cite{Brown}, we consider a modified operator $\ti{L}_{\omega}$. 

\begin{proposition}[Based on {\cite[Lemma 4.1]{Brown}}]\label{Iso_L}
    Fix an $L^{2}$ orthogonal basis $f_{1},...,f_{d}$ for the kernel of $L^{*}_{\omega}$, there exist points $q_{1},...,q_{d}\in M\setminus X$, such that
    \begin{equation}
    \ti{L}_{\omega}:\,C^{4,\alpha}(M) \ni \phi \to L_{\omega}\phi-\sum_{i=1}^{d}\phi(q_{i})f_{i} \in C^{0,\alpha}(M)
    \end{equation}
    is an isomorphism.
\end{proposition}

Proposition \ref{Iso_L} follows from the same argument of the proof of {\cite[Lemma 4.1]{Brown}}. We note that each $f_i \in C^{\infty} (M)$ by the elliptic regularity theory.

Given $q_{i}$ and $f_{i}$ as in Proposition \ref{Iso_L}, we define 
\begin{equation}
    \ti{L}_{\omega_{\epsilon}}\phi=L_{\omega_{\epsilon}}\phi-\sum_{i=1}^{d}\phi(q_{i})f_{i}.
\end{equation}
We observe that for a fixed $\epsilon>0$, both weighted spaces $C^{4,\alpha}_{\delta}(Bl_{X}M)$ and $\ti{C}^{4,\alpha}_{0,\delta}(Bl_{X}M)$ are equivalent to the usual H\"{o}lder space $C^{4,\alpha}(Bl_{X}M)$. Therefore, when $\epsilon$ is fixed, the same proof of Proposition \ref{Fredholm0} implies that
\begin{proposition}\label{Fredholm1}
    The operator
    \begin{align}
    &\ti{L}_{\omega_{\epsilon}}:C^{4,\alpha}_{\delta}(Bl_{X}M,\omega_{\epsilon})\to C^{0,\alpha}_{\delta-4}(Bl_{X}M,\omega_{\epsilon}), \ \text{if } k\geq 3; \\
    &\ti{L}_{\omega_{\epsilon}}:\ti{C}^{4,\alpha}_{0,\delta}(Bl_{X}M,\omega_{\epsilon})\to \ti{C}^{0,\alpha}_{-4,\delta}(Bl_{X}M,\omega_{\epsilon}), \ \text{if } k=2.
    \end{align}
    are Fredholm operators with index zero.
\end{proposition}

We recall an important fact with respect to the weighted space in Definition \ref{WHS_Bl0}. It is concerned with the linearized operator $L_{\eta}$ of the Burns-Simanca metric.
\begin{proposition}[{\cite[Proposition 8.9]{Sze2014}}]\label{Iniec_Leta}
    If $\delta<0$, the operator
    \begin{equation}
        L_{\eta}:C^{4,\alpha}_{\delta}(Bl_{0}\C^{k})\to C^{0,\alpha}_{\delta-4}(Bl_{0}\C^{k})
    \end{equation}
    is injective.
\end{proposition}

\subsection{Bounds on the inverse operator}

We begin with a uniform Schauder estimate on $\ti{L}_{\omega_{\epsilon}}$ with respect to the weighted norms. Such type of estimates are established in various weighted spaces. For example, we refer to \cite[Lemma 4.6]{Brown} for a closely related estimate and \cite[Theorem 1.10]{Bartnik} for more general results. The techniques used in the proof are standard, and we include the details in Appendix \ref{PS} for the reader’s convenience.

\begin{proposition}[Based on {\cite[Lemma 4.6]{Brown}}]\label{Schauder}
    For any $\delta<0$ there exists a constant $C>0$, independent of sufficiently small $\epsilon>0$, such that the operator $\ti{L}_{\omega_{\epsilon}}$ has the following uniform Schauder estimate:
    \begin{align}
        &\|f\|_{C^{4,\alpha}_{\delta}(Bl_{X}M,\omega_{\epsilon})}\leq C\Big( \|f\|_{C^{0}_{\delta}(Bl_{X}M,\omega_{\epsilon})}+\|\ti{L}_{\omega_{\epsilon}}f\|_{C^{0,\alpha}_{\delta-4}(Bl_{X}M,\omega_{\epsilon})} \Big), \ \text{when }k\geq 3;\\
        &\|f\|_{\ti{C}^{4,\alpha}_{0,\delta}(Bl_{X}M,\omega_{\epsilon})}\leq C\Big( \|f\|_{\ti{C}^{0}_{\delta}(Bl_{X}M,\omega_{\epsilon})}+\|\ti{L}_{\omega_{\epsilon}}f\|_{\ti{C}^{0,\alpha}_{-4,\delta}(Bl_{X}M,\omega_{\epsilon})} \Big), \ \text{when  } k=2.
    \end{align}
\end{proposition}

The main result of this subsection is Proposition \ref{Inverse}, where we prove a uniform upper bound of the operator norm of the inverse operator $\ti{L}_{\omega_{\epsilon}}^{-1}$. Such a result is crucial to the application of the contraction mapping theorem. We use the scaling method similar to the proof of \cite[Theorem 8.14]{Sze2014}. This approach was used in \cite{BR2015} to study the smoothing of singular cscK metrics. We follow the proofs of \cite[Theorem 8.14]{Sze2014} and \cite[Proposition 4.8]{Brown} in the case of blow-ups at a point. In the case of blow-ups along submanifolds of codimension $k\geq 3$, a similar result was established in \cite[Proposition 9]{SS2020}. However, unlike the proof of \cite[Proposition 9]{SS2020}, our argument reduces to a purely local analysis, allowing us to avoid the use of a Liouville theorem for product spaces; see \cite[Lemma 11]{SS2020}. In the case $k=2$, those modified weighted norms introduced in Definition \ref{WeightedBLxM2} play a crucial role in the proof. 

\begin{proposition}\label{Inverse}
Fix $\delta\in(-1,0)$, consider the following operators:
\begin{align}
    &\tilde{L}_{\omega_{\epsilon}}: C^{4,\alpha}_{\delta}(Bl_{X}M,\omega_{\epsilon})\longrightarrow C^{0,\alpha}_{\delta-4}(Bl_{X}M,\omega_{\epsilon}), \ \text{when } k\geq 3;
    \\
    &\tilde{L}_{\omega_{\epsilon}}: \ti{C}^{4,\alpha}_{0,\delta}(Bl_{X}M,\omega_{\epsilon})\longrightarrow \ti{C}^{0,\alpha}_{-4,\delta}(Bl_{X}M,\omega_{\epsilon}), \ \text{when } k=2;
\end{align}
There exist $\epsilon_{0}>0$ and $K>0$ independent of any $\epsilon \in (0, \epsilon_{0}]$ such that when $k\geq 3$, for any $\phi\in C^{4,\alpha}_{\delta}(Bl_{X}M, \omega_{\epsilon})$ we have
    \begin{equation}
        \|\phi\|_{C^{4,\alpha}_{\delta}(Bl_{X}M,\omega_{\epsilon})}\leq K \|\ti{L}_{\omega_{\epsilon}}\phi \|_{C^{0,\alpha}_{\delta-4}(Bl_{X}M,\omega_{\epsilon})}.
    \end{equation}
And when $k=2$, for any $\phi\in \ti{C}^{4,\alpha}_{0,\delta}(Bl_{X}M,\omega_{\epsilon})$,
    \begin{equation}
        \|\phi\|_{\ti{C}^{4,\alpha}_{0,\delta}(Bl_{X}M,\omega_{\epsilon})}\leq K \|\ti{L}_{\omega_{\epsilon}}\phi \|_{\ti{C}^{0,\alpha}_{-4,\delta}(Bl_{X}M,\omega_{\epsilon})}.
    \end{equation}
\end{proposition}
\begin{proof}[Proof of Proposition \ref{Inverse}]
    Assume that there exists $\epsilon_{i} \rightarrow 0$ as $i \rightarrow \infty$ such that
    \begin{align}
        &\|\phi_{i}\|_{C^{4,\alpha}_{\delta}(Bl_{X}M,\omega_{\epsilon_{i}})}=1, \quad \|\ti{L}_{\omega_{\epsilon_{i}}}\phi_{i}\|_{C^{0,\alpha}_{\delta-4}(Bl_{X}M,\omega_{\epsilon_{i}})}<\frac{1}{i}, \ \ k \geq 3;\label{contar_assume}
        \\
        &\|\phi_{i}\|_{\ti{C}^{4,\alpha}_{0,\delta}(Bl_{X}M,\omega_{\epsilon_{i}})}=1, \quad \|\ti{L}_{\omega_{\epsilon_{i}}}\phi_{i}\|_{\ti{C}^{0,\alpha}_{-4,\delta}(Bl_{X}M,\omega_{\epsilon_{i}})}<\frac{1}{i},\ \ \ \ k=2. \label{contra_ass_2}
    \end{align}

\vskip 0.2cm

\textbf{Step 1.} Point picking from a normalization sequence in the case $k\geq 3$.

\vskip 0.2cm

    Note that $\|\phi_{i}\|_{C^{4,\alpha}_{\delta}(Bl_{X}M,\omega_{\epsilon_{i}})}=1$ implies that on every compact subset $K\subset M\setminus X$, $\|\phi_{i}\|_{C^{4,\alpha}(K)}$ has a uniform upper bound. Then, according to Arzela-Ascoli Theorem (\cite[Theorem 2.7]{Sze2014}), we have, by replacing $\alpha$ with a smaller number,
    \begin{equation}\label{subconverge}
        \phi_{i}\xrightarrow{C^{4,\alpha}_{loc}(M\setminus X)}\phi_{\infty}, \ (i\to\infty).
    \end{equation}
    The convergence gives an asymptotic behavior of $\phi_{\infty}$. For any fixed $r \leq r_{0}$, when $i$ is sufficiently large, there exists a constant $C$ independent of $\epsilon$, such that
    \begin{equation}\label{phi_infty_asy}
    \begin{split}
         \|\phi_{\infty}\|_{C^{0}(\{r\leq d\leq 2r\})}
         \leq& 
         \|\phi_{\infty}-\phi_{i}\|_{C^{0}(\{r\leq d\leq 2r\})}+\|\phi_{i}\|_{C^{0}(\{r\leq d\leq 2r\})}
         \\
         \leq&
         \|\phi_{\infty}-\phi_{i}\|_{C^{0}(\{r\leq d\leq 2r\})}
         +C\|\phi_{i}\|_{C^{4,\alpha}_{\delta}(Bl_{X}M,\omega_{\epsilon_{i}})}r^{\delta}
         \leq
         2Cr^{\delta}.
    \end{split}
    \end{equation}
    On the other hand, when $i$ is large enough, we have $\omega_{\epsilon_{i}}|_{K}=\omega|_{K}$. Then, the second inequality in \eqref{contar_assume} implies that $\ti{L}_{\omega}\phi_{\infty}=0$ in $M\setminus X$. According to Bochner's results \cite{Bochner}, $\phi_{\infty}$ extends to a solution (in sense of distributions) to $\ti{L}_{\omega}\phi_{\infty}=0$ on $M$ if
    \begin{equation}\label{ext_cri}
        \lim_{a \rightarrow 0+} \frac{1}{a^{4}} \int_{V_{a}} |\phi_{\infty}| d\operatorname{Vol}\,(\omega_{\epsilon})=0,\ \ \text{where}\ V_a=\{ q \in M,\ |\ d(q)=d(q, w)<a\ \text{and}\ w \in V \}
    \end{equation} holds for any sufficiently small closed subset $V \subset X$. Now that 
    \eqref{ext_cri} follows from \eqref{phi_infty_asy} and $k \geq 3$. Applying the elliptic regularity theory on $\ti{L}_{\omega}\phi_{\infty}=0$ and Proposition \ref{Iso_L}, we have $\phi_{\infty} \equiv 0$ on $M$. It follows from \eqref{subconverge} that
    \begin{equation}\label{Condition3}
        \phi_{i}\xrightarrow{C^{4,\alpha}_{loc}(M\setminus X)}0, \ (i\to\infty).
    \end{equation}

    By the Schauder estimate in Proposition \ref{Schauder}, we have a constant $C>0$ independent of $\epsilon$ such that 
    \begin{equation}
        1=\|\phi_{i}\|_{C^{4,\alpha}_{\delta}(Bl_{X}M,\omega_{\epsilon_{i}})}\leq C\Big(\|\phi_{i}\|_{C^{0}_{\delta}(Bl_{X}M,\omega_{\epsilon_{i}})}+\|\ti{L}_{\omega_{\epsilon_{i}}}\phi_{i}\|_{C^{0,\alpha}_{\delta-4}(Bl_{X}M,\omega_{\epsilon_{i}})} \Big).
    \end{equation}
    Hence, $\|\phi_{i}\|_{C^{0}_{\delta}(Bl_{X}M,\omega_{\epsilon_{i}})}$ have a uniform positive lower bound. Then, by letting
    \begin{equation}
        \psi_{i}=\frac{\phi_{i}}{\|\phi_{i}\|_{C^{0}_{\delta}(Bl_{X}M,\omega_{\epsilon_{i}})}},
    \end{equation}
    we have the following properties:
    \begin{align}
        &\|\psi_{i}\|_{C^{0}_{\delta}(Bl_{X}M,\omega_{\epsilon_{i}})}=1, \quad \|\psi_{i}\|_{C^{4,\alpha}_{\delta}(Bl_{X}M,\omega_{\epsilon_{i}})}\leq C; \label{Condition1} \\
        &\|\ti{L}_{\omega_{\epsilon_{i}}}\psi_{i}\|_{C^{0,\alpha}_{\delta-4}(Bl_{X}M,\omega_{\epsilon_{i}})}\rightarrow 0, \ (i\rightarrow \infty); \label{Condition2} 
    \end{align}
    For each $i$, we recall the weight function defined in (\ref{WeightFunc3})
    \begin{equation}
    \tau_{i}(q)=
    \begin{cases}
        r_0, \ d(q)\geq r_{0};\\
        d(q), \ \epsilon_{i}<d(q)<r_{0};\\
        \epsilon_{i}, \ d(q)\leq\epsilon_{i}.
    \end{cases}
    \end{equation}
    Then by \eqref{holder_bl_f_0}, \eqref{Condition3}, and \eqref{Condition1}, there exists a sequence $\{x_{i}\}_{i=1}^{\infty} \subset Bl_{X}M$ such that
    \begin{align}
        &\tau_{i}^{-\delta}(x_{i})\psi_{i}(x_{i})=1, \ \ \forall\ i, \label{MaxPoint} \\
        &d(x_{i})\rightarrow 0, \ (i\rightarrow\infty). 
    \end{align}

\vskip 0.2cm

    \textbf{Step 1.5.}  Point picking from a normalization sequence in the case $k=2$.

\vskip 0.2cm

    For any fixed $r>0$, when $i$ is large enough we may assume that $[\log\frac{1}{\epsilon_{i}}]^{-1}\leq r$, and hence, $\ti\tau(r)=1$. According to the definition of the weight function in Definition \ref{WeightedBLxM2}, we have an estimate of the usual H\"{o}lder norm,
    \begin{align}
        &
        \|\phi_{i}\|_{C^{4,\alpha}(\{r\leq d \leq 2r\})}
        \leq 
        \|\phi_{i}\|_{\ti{C}^{4,\alpha}_{0,\delta}(Bl_{X}M)}r^{-4-\alpha}=r^{-4-\alpha}, \ \text{if } r\leq r_{0};
        \\
        &
        \|\phi_{i}\|_{C^{4,\alpha}(\{r\leq d\leq 2r\})}
        \leq
        \|\phi_{i}\|_{\ti{C}^{4,\alpha}_{0,\delta}(Bl_{X}M)}=1, \ \text{if } r> r_{0}.
    \end{align}
    As in \eqref{subconverge}, by the Arzela--Ascoli Theorem and a diagonal sequence argument, $\phi_{i}$ subconverges to $\phi_{\infty} \in C_{loc}^{4,\alpha}$ uniformly in every compact subset of $M\setminus X$, with a smaller $\alpha$. Moreover, after taking $i \rightarrow \infty$, we get
    \begin{equation}
        |\phi_{\infty}(q)|\leq C\ti{\tau}_{i}(q)+|(\phi_{i}-\phi_{\infty})(q)| \leq C, \ \forall \ q\ \in M\setminus X.
    \end{equation}
    Hence $\phi_{\infty}$ is bounded in $M\setminus X$. When $i$ is sufficiently large, 
    \(
    \omega_{\epsilon_{i}}=\omega
    \) in $\{p\in M\ |\ r \leq d \leq 2r\}$. By \eqref{contra_ass_2}, we have $\ti{L}_{\omega}\phi_{\infty}=0$ in $M\setminus X$. Since $\phi_{\infty}$ is bounded, according to Harvey--Polking's theorem (\cite[Theorem 6.4]{HP1970}), we know that $\ti{L}_{\omega}\phi_{\infty}=0$ actually holds on $M$. It follows from Proposition \ref{Iso_L} that $\phi_{\infty}\equiv 0$, i.e. $\phi_{i}$ converges to 0 uniformly in every compact subset of $M\setminus X$.
    
    The rest of the argument is analogous to that in Step 1. We consider
    \begin{equation}
        \ti{\psi}_{i}=\frac{\phi_{i}}{\|\phi_{i}\|_{\ti{C}^{0}_{\delta}(Bl_{X}M, \omega_{\epsilon_{i}})}}.
    \end{equation}
    Then we obtain the following
    \begin{align}
        &\|\ti{\tau}_{i}^{-\delta}\ti{\psi}_{i}\|_{C^{0}(Bl_{X}M,\omega_{\epsilon_{i}})}=1, \quad \|\ti{\psi}_{i}\|_{\ti{C}^{4,\alpha}_{0,\delta}(Bl_{X}M,\omega_{\epsilon_{i}})}\leq C; \label{Condition4} \\&\|\ti{L}_{\omega_{\epsilon_{i}}}\ti{\psi}_{i}\|_{\ti{C}^{0,\alpha}_{-4,\delta}(Bl_{X}M,\omega_{\epsilon_{i}})}\rightarrow 0, \ (i\rightarrow \infty). \label{Condition5} 
    \end{align}
    Suppose that $x_{i}$ is the maximum point of $\ti{\tau}_{i}^{-\delta}\psi_{i}$ for each $i$. Then,
    \begin{align}
        &\ti{\tau}_{i}^{-\delta}(x_{i})\psi_{i}(x_{i})=1, \ \forall i, \label{MaxPoint2} \\
        &d(x_{i})\rightarrow 0, \ (i\rightarrow\infty). 
    \end{align}
    
    According to the distance of $\{x_{i}\}$ to the exceptional divisor, we discuss two cases in Steps 2 and 3. For simplicity, we write
$\|\psi_{i}\|_{C^{4,\alpha}_{\delta}(Bl_{X}M)}$
and
$\|\ti{\psi}_{i}\|_{\ti{C}^{4,\alpha}_{0,\delta}(Bl_{X}M)}$
for
$\|\psi_{i}\|_{C^{4,\alpha}_{\delta}(Bl_{X}M,\omega_{\epsilon_{i}})}$
and
$\|\ti{\psi}_{i}\|_{\ti{C}^{4,\alpha}_{0,\delta}(Bl_{X}M,\omega_{\epsilon_{i}})}$,
respectively.

\vskip 0.2cm

    \textbf{Step 2}: There exists some $R>0$ such that $\frac{d(x_{i})}{\epsilon_{i}}<R, \ \forall i>>1$.

\vskip 0.2cm

    The assumption implies that the sequence of points $x_i$ accumulates near the exceptional divisor. We would cut-off $\psi_{i}$ and $\ti{\psi}_{i}$ around each $x_{i}$ and pull them back to the model space $(Bl_{0}\C^{k}\times \C^{m-k},\omega_{\eta})$. Using Proposition \ref{Iniec_Leta}, a scaling analysis will lead to a contradiction.

    For each $x_{i}$, there exists a coordinate neighborhood $U_{i} \subset Bl_{X} M$ such that $\sigma(U_i) \subset M$ is constructed from Lemma \ref{holosubcoor}. Moreover, we may assume that
    \begin{equation}\label{Maxpoint}
    x_{i}=(z_{(l_{i}),i},0)\in \ti{B}_{(l_{i}),r_{0}}^{k}\times B_{r_{0}}^{m-k}\subset \ti{B}_{2r_{0}}^{k}\times B_{2r_{0}}^{m-k}\subset U_{i}.
    \end{equation}
    Note that since $d(x_{i})\to 0$, we can choose constants $r_{0}$ to be uniform for $i$. Then, from $\frac{\tau_{i}(x_{i})}{\epsilon_{i}}<R$, we know that
    \begin{equation}
        |\sigma(z_{(l_{i}),i},0)|=|z_{(l_{i}),i}^{l_{i}}|(1+\sum_{j\neq i}|z_{(l_{i}),i}^{j}|^{2})^{\frac{1}{2}}<R\epsilon_{i}, \quad \text{if}   \ |z_{(l_{i}),i}^{l_{i}}|\neq 0.
    \end{equation}

    Take cut-off functions $\mu\in C^{\infty}_{c}(\C^{k})$ and $\nu\in C^{\infty}_{c}(\C^{m-k})$ such that $0\leq\mu,\nu\leq 1$, and
    \begin{align}
        &\mu(z)\equiv 1, \ z\in \ol{B_{1}^{k}}; \ \mu(z)\equiv 0, \ z\in \mC^{k}\setminus B_{2}^{k}; \label{Def_Cutoff} \\
        &\nu(w)\equiv 1, \ w\in \ol{B_{1}^{m-k}}; \ \nu(w)\equiv 0, \ w\in \mC^{m-k}\setminus B_{2}^{m-k}; \label{cutoff_nu}
    \end{align}

     Let $\chi(z,w)=\mu(\frac{z}{r_{0}})\nu(\frac{w}{r_{0}})$. Note that $\sigma^{\ast}\chi$ extends naturally to the exceptional divisor in any coordinate charts, then, we may consider $(\sigma^{\ast}\chi)\psi_{i}|_{U_{i}}$. We introduce a scaling function
     \begin{equation}
         \Lambda_{i}(Z,W)=(\epsilon_{i}Z,\epsilon_{i}^{2}W).
     \end{equation}
     and take coordinates $(Z_{(j)},W)\in \ti{B}_{(j),\frac{2r_{0}}{\epsilon_{i}}}^{k}\times B_{\frac{2r_{0}}{\epsilon_{i}^{2}}}^{m-k}$ with $\sigma(Z_{(j)},W)=(Z,W)\in B_{\frac{2r_{0}}{\epsilon_{i}}}^{k}\times B_{\frac{2r_{0}}{\epsilon_{i}^{2}}}^{m-k}$ if $Z_{(j)}\neq 0$. 
    Then, we define
    \begin{equation}
        \Psi_{i}(Z_{(j)},W)
        =\epsilon_{i}^{-\delta}\Lambda_{i}^{*}(\sigma^{\ast}\chi)\psi_{i}|_{U_{i}}(Z_{(j)},W)
        =\epsilon_{i}^{-\delta} \mu(\frac{\epsilon_{i}Z}{r_{0}})\nu(\frac{\epsilon_{i}^{2}W}{r_{0}})\psi_{i}|_{U_{i}}(\epsilon_{i}Z_{(j)},\epsilon_{i}^{2}W), \ j=1,..,k.
    \end{equation}
    as a function defined on $ Bl_{0}\C^{k}\times\C^{m-k}$, with $\operatorname{supp}(\Psi_{i})\subset \ti{B}_{\frac{2r_{0}}{\epsilon_{i}}}^{k}\times B_{\frac{2r_{0}}{\epsilon_{i}^{2}}}^{m-k}$. 

    When $k=2$, the corresponding function is chosen as follows.
    \begin{equation}
        \ti{\Psi}_{i}(Z_{(j)},W)
        =[\ti{\tau}_{i}(\epsilon_{i})]^{-\delta}\Lambda_{i}^{*}(\sigma^{\ast}\chi)\ti{\psi}_{i}|_{U_{i}}(Z_{(j)},W), \ j=1,2.
    \end{equation}

    A direct observation shows that for any compact subset $B\subset Bl_{0}\C^{k}\times \C^{m-k}$ there exists a constant $C$ independent of sufficiently large $i$ such that for any $|I|+|J|\leq 4$,
    \begin{align}
        &\|\di_{Z_{(1)}}^{I}\di_{W}^{J}\Psi_{i}\|_{C^{0,\alpha}(B)}\leq C\epsilon_{i}^{|J|}\|\psi_{i}\|_{C^{4,\alpha}_{\delta}(Bl_{X}M)}, \label{S2E0}
        \\        &\|\di_{Z_{(1)}}^{I}\di_{W}^{J}\ti{\Psi}_{i}\|_{C^{0,\alpha}(B)}\leq C\epsilon_{i}^{|J|}\|\ti{\psi}_{i}\|_{\ti{C}^{4,\alpha}_{0,\delta}(Bl_{X}M)}. \label{S2E0_2}
    \end{align}
    We refer to Appendix \ref{EL} for the derivation of \eqref{S2E0}. According to \cite[Theorem 2.7]{Sze2014}, we have a subsequence choosing by diagonal principle such that
    \begin{equation}
        \Psi_{i}\xrightarrow{C^{4,\alpha_{1}}_{loc}} \Psi_{\infty}, \ \ \ \
        \ti{\Psi}_{i}\xrightarrow{C^{4,\alpha_{1}}_{loc}} \ti{\Psi}_{\infty} \ (i\longrightarrow\infty). \label{LocConverge1}
    \end{equation}
    with $\Psi_{\infty}$ defined on $Bl_{0}\C^{k}\times\C^{m-k}$, $\ti{\Psi}_{\infty}$ defined on $Bl_{0}\C^{2}\times\C^{m-2}$ and $\alpha_{1}<\alpha$. 
    
    Furthermore, we have, for any compact subset $B\subset \ti{B}_{R_{1}}^{k}\times B_{S_{1}}^{m-k}\subset Bl_{0}\C^{k}\times\C^{m-k}$,
    \begin{equation}
        \begin{split}
            \|\di_{Z}^{I}\di_{W}^{J}\Psi_{\infty}\|_{C^{0}(B)}
            \leq& \|\di_{Z}^{I}\di_{W}^{J}(\Psi_{\infty}-\Psi_{i})\|_{C^{0}(B)}+\|\di_{Z}^{I}\di_{W}^{J}\Psi_{i}\|_{C^{0}(B)}\\
            \leq& \|\Psi_{\infty}-\Psi_{i}\|_{C^{4,\alpha_{1}}(B)}+C\epsilon_{i}^{|J|}\|\psi_{i}\|_{C^{4,\alpha}_{\delta}(Bl_{X}M)},
        \end{split}
    \end{equation}
    with $C$ depends only on the compact subset $B$. Let $i\to\infty$, we have $\di_{Z}^{I}\di_{W}^{J}\Psi_{\infty}\equiv0$ for any $|I|+|J|\leq 4$,  $|J|\neq 0$, which means $\Psi_{\infty}$ is independent of the $\C^{m-k}$ factor. Using \eqref{S2E0_2}, a similar analysis leads to a similar result for $\ti{\Psi}_{\infty}$, i.e. we have shown
    \begin{align}
        &\Psi_{\infty}(Z,W)=\Psi_{\infty}(Z,0), \ \ \forall\ (Z,W)\in Bl_{0}\C^{k}\times \C^{m-k},
        \\
        &\ti{\Psi}_{\infty}(Z,W)=\ti{\Psi}_{\infty}(Z,0), \ \ \forall\ (Z,W)\in Bl_{0}\C^{2}\times \C^{m-2}.
    \end{align}
    Recall that we introduced the scaled Burns-Simanca metric on $Bl_{0}\C^{k}$ in \eqref{ScaledPMetric}.
    \begin{equation}
        \eta_{\epsilon_{i}}(z)=
        \begin{cases}
        \epsilon_{i}^{2}\sqrt{-1}\di\dbar\Big( |\frac{z}{\epsilon_{i}}|^{2}+\gamma(|\frac{z}{\epsilon_{i}}|)\log|\frac{z}{\epsilon_{i}}|^{2}+\psi(|\frac{z}{\epsilon_{i}}|^{2}) \Big),\ \ \ \ k \geq 3,\\
         \epsilon_{i}^{2}\sqrt{-1}\di\dbar\Big( |\frac{z}{\epsilon_{i}}|^{2}+\log|\frac{z}{\epsilon_{i}}|^{2} \Big),\ \ \ \ k = 2.
        \end{cases}
    \end{equation}
    In fact, $\eta=\epsilon_{i}^{-2}\Lambda_{i}^{*}\eta_{\epsilon_{i}}$. Define a function $\ti{\Lambda}_{i}(z,w)=(\frac{z}{\epsilon_{i}},\frac{w}{\epsilon_{i}^{2}})$ inverse to $\Lambda_{i}$, for the linearized operator,
    \begin{equation}\label{Scale_L_eta}
        L_{\eta}\Psi
        =L_{\epsilon_{i}^{-2}\Lambda_{i}^{*}\eta_{\epsilon_{i}}}\Lambda_{i}^{*}\ti{\Lambda}_{i}^{*}\Psi
        =\epsilon_{i}^{4}\Lambda_{i}^{*}(L_{\eta_{\epsilon_{i}}}\ti{\Lambda}_{i}^{*}\Psi)
        =\epsilon_{i}^{4}\Lambda_{i}^{*}(L_{\omega_{\eta,\epsilon_{i}}}\ti{\Lambda}_{i}^{*}\Psi), \ \forall\ \Psi\in C^{4,\alpha}_{loc}(Bl_{0}\C^{k})
    \end{equation}
    The derivation of this equality can be found in Appendix \ref{EL}. By \eqref{Condition2}, we can prove that for any $R_{1}>0$,
    \begin{equation}\label{S2E1}
    \begin{split}
        \|L_{\eta}\Psi_{\infty}(\cdot,0)\|_{C^{0}(\ti{B}_{R_{1}}^{k})}
        =&\|L_{\omega_{\eta,\epsilon_{i}}}\epsilon_{i}^{4}\ti{\Lambda}_{i}^{*}\Psi_{\infty}(\cdot,0)\|_{C^{0}(\ti{B}_{\epsilon_{i}R_{1}}^{k})}
        \\
        \leq& \|\ti{L}_{\omega_{\epsilon_{i}}}\psi_{i}\|_{C^{0,\alpha}_{\delta-4}(Bl_{X}M)}
        \\
        &+C\|g_{\omega_{\epsilon_{i}}}-g_{\omega_{\eta,\epsilon_{i}}}\|_{C^{2,\alpha}_{0}(\ti{B}_{R_{1}\epsilon_{i}}^{k})}\|\Psi_{i}-\Psi_{\infty}\|_{C^{0}(\ti{B}_{R_{1}}^{k})}
        \to 0, \ (i\to\infty).
    \end{split}
    \end{equation}
    See Appendix \ref{EL} for a derivation of \eqref{S2E1}. By a similar argument, we have
    \begin{equation}
        \|L_{\eta}\ti{\Psi}_{\infty}(\cdot,0)\|_{C^{0}(\ti{B}_{R_{1}}^{2})}\leq \|\ti{L}_{\omega_{\epsilon_{i}}}\ti{\psi}_{i}\|_{\ti{C}^{4,\alpha}_{-4,\delta}(Bl_{X}M)}+C\|\ti{\Psi}_{i}-\ti{\Psi}_{\infty}\|_{C^{0}(\ti{B}_{R_{1}}^{2})}
        \to 0, \ (i\to\infty).
    \end{equation}
    As a result, 
    \begin{align}
        &L_{\eta}\Psi_{\infty}(\cdot,0)=0, \ \text{on } \ Bl_{0}\C^{k}, 
        \\
        &L_{\eta}\ti{\Psi}_{\infty}(\cdot,0)=0, \ \text{on } \ Bl_{0}\C^{2}.
    \end{align}
    Consider the decay (growth) of $\Psi_{\infty}(\cdot,0)$. Fix $r>1$, when $i$ is sufficiently large
    \begin{equation}
        \begin{split}
            &\|(\di_{Z}^{I}\Psi_{\infty})(r\ti{Z},0)\|_{C^{0,\alpha_{1}}(B_{2}^{k}\setminus B_{1}^{k})}\\
            \leq& \|(\di_{Z}^{I}[\Psi-\Psi_{i}])(r\ti{Z},r\ti{W})\|_{C^{0,\alpha_{1}}(B_{2}^{k}\setminus B_{1}^{k}\times B_{2}^{m-k})}+\|\epsilon_{i}^{|I|-\delta}(\di_{z}^{I}\psi_{i})(\epsilon_{i}r\ti{Z},\epsilon_{i}^{2}r\ti{W}) \|_{C^{0,\alpha_{1}}(B_{2}^{k}\setminus B_{1}^{k}\times B_{2}^{m-k})}\\
            \leq& \|\Psi-\Psi_{i}\|_{C^{4,\alpha_{1}}(B_{2r}^{k}\setminus B_{r}^{k}\times B_{2r}^{m-k})}+\|\psi_{i}\|_{C^{4,\alpha}_{\delta}(Bl_{X}M)}\epsilon_{i}^{|I|-\delta}(\epsilon_{i}r)^{\delta-|I|} \\
            \leq& Cr^{\delta-|I|}, \ \forall |I|\leq 4.
        \end{split}
    \end{equation}
    Similarly,
    \begin{equation}
        \begin{split}
            &\|(\di_{Z}^{I}\ti{\Psi}_{\infty})(r\ti{Z},0)\|_{C^{0,\alpha_{1}}(B_{2}^{2}\setminus B_{1}^{2})}
            \\
            \leq& 
            \|(\di_{Z}^{I}[\ti{\Psi}-\ti{\Psi}_{i}])(r\ti{Z},r\ti{W})\|_{C^{0,\alpha_{1}}(B_{2}^{2}\setminus B_{1}^{2}\times B_{2}^{m-2})}
            \\
            &+
            \|[\ti{\tau}_{i}(\epsilon_{i})]^{-\delta}\epsilon_{i}^{|I|}(\di_{z}^{I}\psi_{i})(\epsilon_{i}r\ti{Z},\epsilon_{i}^{2}r\ti{W}) \|_{C^{0,\alpha_{1}}(B_{2}^{2}\setminus B_{1}^{2}\times B_{2}^{m-2})}
            \\
            \leq& 
            \|\ti{\Psi}-\ti{\Psi}_{i}\|_{C^{4,\alpha_{1}}(B_{2r}^{2}\setminus B_{r}^{2}\times B_{2r}^{m-2})}+\|\ti{\psi}_{i}\|_{C^{4,\alpha}_{0,\delta}(Bl_{X}M)}[\ti{\tau}_{i}(\epsilon_{i})]^{-\delta}\epsilon_{i}^{|I|}[\ti{\tau}_{i}(r\epsilon_{i})]^{\delta}(\epsilon_{i}r)^{-|I|} 
            \\
            \leq& 
            Cr^{\delta-|I|}, \ \forall |I|\leq 4.
        \end{split}
    \end{equation}
    Hence, we know that $\Psi_{\infty}(\cdot,0)\in C^{4,\alpha}_{\delta}(Bl_{0}\C^{k})$ and $\ti{\Psi}_{\infty}(\cdot,0)\in C^{4,\alpha}_{\delta}(Bl_{0}\C^{2})$. According to Proposition \ref{Iniec_Leta}, we have $\Psi_{\infty}=\ti{\Psi}_{\infty}\equiv 0$.

    Recall our assumption (\ref{MaxPoint}) and coordinates of those $x_i$ defined in (\ref{Maxpoint}). Then we have the following.
    \begin{align}
        &\Psi_{i}(\frac{z_{(j_{i}),i}}{\epsilon_{i}},0)=\epsilon_{i}^{-\delta}\psi_{i}|_{U_{i}}(z_{(j_{i}),i},0)=(\frac{\tau_{i}(x_{i})}{\epsilon_{i}})^{\delta}>R^{\delta},
        \\
        &\ti{\Psi}_{i}(\frac{z_{(j_{i}),i}}{\epsilon_{i}},0)=[\ti{\tau}_{i}(\epsilon_{i})]^{-\delta}\ti{\psi}_{i}|_{U_{i}}(z_{(j_{i}),i},0)=(\frac{\ti{\tau}_{i}(x_{i})}{\ti{\tau}_{i}(\epsilon_{i})})^{\delta}\geq (\frac{d(x_{i})}{\epsilon_{i}})^{\delta}>R^{\delta}.
    \end{align}
    Note that $(\frac{z_{(j_{i}),i}}{\epsilon_{i}},0)\in \ti{B}_{(1),R}^{k}\times B_{1}^{m-k}$, then there exists a subsequence such that:
    \begin{equation}
        (\frac{z_{(j_{i}),i}}{\epsilon_{i}},0)\to (Z_{(j_{\infty}),\infty},0)\in \ol{\ti{B}_{(1),R}^{k}\times B_{1}^{m-k}}
    \end{equation}
    Therefore,
    \begin{align}
        &\Psi_{\infty}(Z_{(j_{\infty}),\infty},0)=\lim_{i\to\infty}\Psi_{i}(\frac{z_{(j_{i}),i}}{\epsilon_{i}},0)\geq R^{\delta}>0,
        \\
        &\ti{\Psi}_{\infty}(Z_{(j_{\infty}),\infty},0)=\lim_{i\to\infty}\ti{\Psi}_{i}(\frac{z_{(j_{i}),i}}{\epsilon_{i}},0)\geq R^{\delta}>0.
    \end{align}
    which contradicts $\Psi_{\infty}(\cdot,0)=0$.

\vskip 0.2cm

    \textbf{Step 3.} $\frac{d(x_{i})}{\epsilon_{i}}$ is not bounded.

\vskip 0.2cm

    The assumption implies that the sequence of points $x_i$ accumulates in the ``annular" region away from the exceptional divisor. A similar scaling analysis in the model space $(\C^{k}\setminus\{0\}\times \C^{m-k},\omega_{\mathrm{euc}})$ will lead to another contradiction.

    By Definition (\ref{WeightFunc3}) and the fact that $\psi_{i},\ti{\psi}_{i}\xrightarrow{C^{4,\alpha}_{loc}(M\setminus X)}0$, we know that 
    \begin{equation}
        \epsilon_{i}<d(x_{i})\to 0, \ (i\to\infty).
    \end{equation}
    Then, arguing as in \textbf{Step 2}, we may assume that for each $x_{i}$, there exists a coordinate chart $B_{2r_{0}}^{k}\times B_{2r_{0}}^{m-k}\subset U_{i}$ of the type described in Lemma \ref{holosubcoor}, such that
    \begin{equation}
        x_{i}=(z_{i},0)\in B_{r_{0}}^{k}\setminus B_{\frac{\epsilon_{i}}{2}}^{k} \times B_{r_{0}}^{m-k}
    \end{equation}

    Next, we consider $x_{i}=(z_{i},0)$. Take $R_{i}>0$ and $r_{i}>0$, such that
    \begin{align}
        &R_{i}\to\infty, \ R_{i}|z_{i}|\to 0;   \label{choose_R_i}\\
        &r_{i}\to 0, \ \frac{\epsilon_{i}}{r_{i}|z_{i}|}\to 0.
    \end{align}
    Then, we have $(z_{i},0)\in (B_{R_{i}|z_{i}|}^{k}\setminus B_{r_{i}|z_{i}|}^{k})\times B_{r_{0}}^{m-k}\subset U_{i}$, which also ensures that $r_{i}|z_{i}|$ are away from $\epsilon_{i}$ and $R_{i}|z_{i}|$ are shrinking. 

    Take a smooth cutoff function $\kappa_{i}\in C_{c}^{\infty}(\C^{k})$, $0\leq\kappa\leq 1$, such that
    \begin{equation}
        \kappa_{i}(z)=
        \begin{cases}
            1,\ \ r_{i}|z_{i}|\leq |z|\leq R_{i}|z_{i}|; \\
            0,\ \ |z|\leq \frac{r_{i}|z_{i}|}{2}, \ |z|\geq 2R_{i}|z_{i}|.
        \end{cases}
    \end{equation}
    We take the same cutoff function $\nu$ defined in (\ref{cutoff_nu}), and define $\lambda_{i}(z,w)=\kappa_{i}(z)\nu(\frac{w}{r_{0}})$ and a scaling function
    \begin{equation}
        \Theta_{i}(Z,W)=(|z_{i}|Z,|z_{i}|^{2}W).
    \end{equation}
    Now we set
    \begin{align}
         &\Phi_{i}(Z,W)
         =|z_{i}|^{-\delta}\Theta_{i}^{*}\lambda_{i}\psi_{i} (Z,W)
         =|z_{i}|^{-\delta}\kappa_{i}(|z_{i}|Z)\nu(\frac{|z_{i}|^{2}W}{r_{0}})\psi(|z_{i}|Z,|z_{i}|^{2}W),
         \\
         &\ti{\Phi}_{i}(Z,W)
         =[\ti{\tau}_{i}(|z_{i}|)]^{-\delta}\Theta_{i}^{*}\lambda_{i}\ti{\psi}_{i}(Z,W),
    \end{align}
    with $\operatorname{supp}\Phi_{i}\subset B_{2R_{i}}^{k}\setminus B_{\frac{r_{i}}{2}}^{k}\times B_{\frac{2r_{0}}{|z_{i}|^{2}}}^{m-k}$ and $\operatorname{supp}\ti{\Phi}_{i}\subset B_{2R_{i}}^{2}\setminus B_{\frac{r_{i}}{2}}^{2}\times B_{\frac{2r_{0}}{|z_{i}|^{2}}}^{m-2}$.

    Similarly as in \textbf{Step 2}, we have for any fixed constant $r>0$ and multi-index $|I|+|J|\leq 4$,
    \begin{align}
        &\|(\di_{Z}^{I}\di_{W}^{J}\Phi_{i})(r\ti{Z},r\ti{W})\|_{C^{0,\alpha}(B_{2}^{k}\setminus B_{1}^{k}\times B_{2}^{k})}
        \leq  2\|\psi_{i}\|_{C^{4,\alpha}_{\delta}(Bl_{X}M)}r^{\delta-|I|-|J|}|z_{i}|^{|J|}, \label{D_psi_euc}
        \\
        &\|(\di_{Z}^{I}\di_{W}^{J}\ti{\Phi}_{i})(r\ti{Z},r\ti{W})\|_{C^{0,\alpha}(B_{2}^{2}\setminus B_{1}^{2}\times B_{2}^{2})}
        \leq  2\|\ti{\psi}_{i}\|_{\ti{C}^{4,\alpha}_{0,\delta}(Bl_{X}M)}r^{\delta-|I|-|J|}|z_{i}|^{|J|}. \label{D_psi_euc_2}
    \end{align}
    As a consequence, we can choose a subsequence of $\Phi_{i}$ by the diagonal principle such that
    \begin{align}
        &\Phi_{i}\xrightarrow{C^{4,\alpha_{1}}_{loc}} \Phi_{\infty}, \ (i\to\infty),
        \\
        &\ti{\Phi}_{i}\xrightarrow{C^{4,\alpha_{1}}_{loc}} \ti{\Phi}_{\infty}, \ (i\to\infty),
    \end{align}
    for some $\Phi_{\infty}\in C^{4,\alpha_{1}}_{loc}(\C^{k}\setminus\{0\}\times\C^{m-k})$, $\ti{\Phi}_{\infty}\in C^{4,\alpha_{1}}_{loc}(\C^{2}\setminus\{0\}\times\C^{m-2})$ and $\alpha_{1}<\alpha$. Furthermore, by (\ref{D_psi_euc}), we have, for any $j=1,...,m-k$,
    \begin{equation}\label{Dw_Phi}
        \begin{split}
        &\|\di_{W^{j}}\Phi_{\infty}\|_{C^{0}(B_{2r}^{k}\setminus B_{r}^{k}\times B_{2r}^{k})}\\
        \leq& \|\di_{W^{j}}(\Phi_{\infty}-\Phi_{i})\|_{C^{0,\alpha_{1}}(B_{2r}^{k}\setminus B_{r}^{k}\times B_{2r}^{k})}+\|(\di_{W^{j}}|z_{i}|^{-\delta}\Theta_{i}^{*}\lambda_{i}\psi_{i})\|_{C^{0,\alpha_{1}}(B_{2r}^{k}\setminus B_{r}^{k}\times B_{2r}^{k})}\\
        \leq& \|\Phi_{\infty}-\Phi_{i}\|_{C^{4,\alpha_{1}}(B_{2r}^{k}\setminus B_{r}^{k}\times B_{2r}^{k})}+Cr^{\delta}|z_{i}|.
        \end{split}
    \end{equation}
    After we take $i\to \infty$ in the right-hand side, we have $\di_{W^{j}}\Phi_{\infty}\equiv 0$ for any $j$. Hence,
    \begin{equation}
        \Phi_{\infty}(Z,W)=\Phi_{\infty}(Z,0), \ \forall(Z,W)\in\C^{k}\setminus\{0\}\times\C^{m-k}.
    \end{equation}
    Similarly, by \eqref{D_psi_euc_2} we also have the following.
    \begin{equation}
        \ti{\Phi}_{\infty}(Z,W)\equiv \Phi_{\infty}(Z,0) \ \forall(Z,W)\in\C^{2}\setminus\{0\}\times\C^{m-2}.
    \end{equation}
    On the other hand, by \eqref{D_psi_euc} and \eqref{D_psi_euc_2}, when $i$ is sufficiently large, we also have
    \begin{align}
        \|(\di_{Z}^{I}\Phi_{\infty})(r\ti{Z},0)\|_{C^{0,\alpha_{1}}(B_{2}^{k}\setminus B_{1}^{k})}, \ \|(\di_{Z}^{I}\ti{\Phi}_{\infty})(r\ti{Z},0)\|_{C^{0,\alpha_{1}}(B_{2}^{2}\setminus B_{1}^{2})}\leq Cr^{\delta-|I|},
    \end{align}
    which means $\Phi_{\infty}\in C^{4,\alpha_{1}}_{\delta}(\C^{k}\setminus\{0\})$ and $\ti{\Phi}_{\infty}\in C^{4,\alpha_{1}}_{\delta}(\C^{2}\setminus\{0\})$.

    Again, by our assumption that $\|\ti{L}_{\omega_{\epsilon_{i}}}\psi_{i}\|_{C^{0,\alpha}_{\delta-4}(Bl_{X}M)}\to 0$ and $\|\ti{L}_{\omega_{\epsilon_{i}}}\ti{\psi}_{i}\|_{\ti{C}^{0,\alpha}_{-4,\delta}(Bl_{X}M)}\to 0$, a direct calculation in Appendix \ref{EL} shows that
    \begin{equation}\label{S3E1}
        \Delta^{2}_{\mathrm{euc}}\Phi_{\infty}=0, \ \ \ \ \Delta^{2}_{\mathrm{euc}}\ti{\Phi}_{\infty}=0.
    \end{equation}
    Then, it follows from Corollary \ref{Iso_Bilap} that $\Phi_{\infty}(Z,W)=\Phi_{\infty}(Z,0)=0$ and $\ti{\Phi}_{\infty}(Z,W)=\ti{\Phi}_{\infty}(Z,0)=0$.

    However, notice that we can choose a subsequence such that
    \begin{equation}
        (\frac{z_{i}}{|z_{i}|},0)\to (Z_{\infty},0)\in\C^{k}\setminus\{0\}\times \C^{m-k}.
    \end{equation}
    Then, we have
    \begin{align}
        &\Phi_{\infty}(Z_{\infty},0)=\lim_{i\to\infty}|z_{i}|^{-\delta}\psi_{i}|_{U_{i}}(z_{i},0)=\lim_{i\to\infty}(\frac{\tau(x_{i})}{|z_{i}|})^{\delta}=\lim_{i\to\infty}(\frac{d(x_{i})}{|z_{i}|})^{\delta}\geq 2^{\delta},
        \\
        &\widetilde{\Phi}_{\infty}(Z_{\infty},0)=\lim_{i\to\infty}[\ti{\tau}_{i}(|z_{i}|)]^{-\delta}\psi_{i}|_{U_{i}}(z_{i},0)=\lim_{i\to\infty}(\frac{\ti{\tau}_{i}(x_{i})}{\ti{\tau}_{i}(|z_{i}|)})^{\delta}=\lim_{i\to\infty}(\frac{d(x_{i})}{|z_{i}|})^{\delta}\geq 2^{\delta}.
    \end{align} 
    Here we use \eqref{MaxPoint}. These are contradictions to $\Phi_{\infty}=0$ and $\widetilde{\Phi}_{\infty}=0$.

    Based on the contradictions we obtained in \textbf{Steps 2 and 3}, we prove Proposition \ref{Inverse}.

\end{proof}

Proposition \ref{Inverse} are sufficient to show that $\ti{L}_{\omega_{\epsilon}}$ is injective in the corresponding weighted spaces. In the meantime, Proposition \ref{Fredholm1} shows that $\ti{L}_{\omega_{\epsilon}}$ is a Fredholm operator with index zero. Therefore we immediately have the uniform upper bound on the inverse operator.
\begin{corollary}\label{UniBound}
    Fix $\epsilon>0$ sufficiently small. both operators
    \begin{equation}
    \begin{split}
        &\ti{L}_{\omega_{\epsilon}}:C^{4,\alpha}_{\delta}(Bl_{X}M)\to C^{0,\alpha}_{\delta-4}(Bl_{X}M)\ \ \text{when}\ k \geq 3,\\
        &\ti{L}_{\omega_{\epsilon}}:\ti{C}^{4,\alpha}_{0,\delta}(Bl_{X}M)\to \ti{C}^{0,\alpha}_{-4,\delta}(Bl_{X}M)\ \ \text{when}\ k = 2.
    \end{split}
    \end{equation}
    are isomorphisms. And there exists a constant $K>0$ independent of $\epsilon$ such that
    \begin{equation}
    \begin{split}
        &\|(\ti{L}_{\omega_{\epsilon}})^{-1}\|_{C^{0,\alpha}_{\delta-4}(Bl_{X}M)\to C^{4,\alpha}_{\delta}(Bl_{X}M)}\leq K, \ \text{if} \ k\geq3, \\
        &\|(\ti{L}_{\omega_{\epsilon}})^{-1}\|_{\ti{C}^{0,\alpha}_{-4,\delta}(Bl_{X}M)\to \ti{C}^{4,\alpha}_{0,\delta}(Bl_{X}M)}\leq K, \ \text{if} \ k=2.
    \end{split}
    \end{equation}
\end{corollary}

\section{The case of codimension at least \texorpdfstring{$3$}{3}}\label{section5}

In this section, we follow \cite[Section 8.3]{Sze2014} and \cite[Section 6.1]{Brown} closely to solve the following equation
\begin{equation}\label{Equation0}
    S(\omega)-S(\omega_{\epsilon}+\sqrt{-1}\di\dbar\phi)=-\sum_{i=1}^{d}\phi(q_{i})f_{i}.
\end{equation}
Recall the definition of the linearized operator \eqref{def_L}. We  have
\begin{equation}
S(\omega_{\epsilon}+\sqrt{-1}\di\dbar\phi)=S(\omega_{\epsilon})+L_{\omega_{\epsilon}}\phi+Q_{\omega_{\epsilon}}(\phi),
\end{equation}
where $Q_{\omega_{\epsilon}}$ denotes the higher order terms. Then, the equation (\ref{Equation0}) is equivalent to 
\begin{equation}
    S(\omega)-S(\omega_{\epsilon})-Q_{\omega_{\epsilon}}(\phi)=\ti{L}_{\omega_{\epsilon}}(\phi).
\end{equation}
Then we are led to consider applying the contraction mapping principle to the following map
\begin{equation}
    \begin{split}
       \mathcal{N}_{\epsilon}:C^{4,\alpha}_{\delta}(Bl_{X}M)&\to C^{4,\alpha}_{\delta}(Bl_{X}M),\\
       \phi&\to (\ti{L}_{\omega_{\epsilon}})^{-1}\Big(S(\omega)-S(\omega_{\epsilon})-Q_{\omega_{\epsilon}}(\phi)\Big).
    \end{split}
\end{equation}
In view of Proposition \ref{Inverse}, we need to bound the weighted norm of $S(\omega)-S(\omega_{\epsilon})-Q_{\omega_{\epsilon}}(\phi)$. Observe that
\begin{equation}
    \mathcal{N}_{\epsilon}(\phi_{1})-\mathcal{N}_{\epsilon}(\phi_{2})=(\ti{L}_{\omega_{\epsilon}})^{-1}\Big(Q_{\omega_{\epsilon}}(\phi_{1})-Q_{\omega_{\epsilon}}(\phi_{2})\Big).
\end{equation}
Therefore, to obtain the contraction mapping property of $\mathcal{N}_{\epsilon}$, we need to estimate $Q_{\omega_{\epsilon}}(\phi_{1})-Q_{\omega_{\epsilon}}(\phi_{2})$.

\begin{proposition}[{\cite[Lemma 8.18]{Sze2014}}] \label{Qestimate}
    Suppose $\delta\in(-1,0)$ and $c>0$ is the constant in Proposition \ref{UniEst}. Then there exists a constant $C>0$ which is independent of $\epsilon$ such that
    \begin{equation}
    \begin{split}
        \|Q_{\omega_{\epsilon}}(\phi_{0})-Q_{\omega_{\epsilon}}(\phi_{1})\|_{C^{0,\alpha}_{\delta-4}(Bl_{X}M)}
        \leq C\Big(\|\phi_{0}\|_{C^{4,\alpha}_{2}(Bl_{X}M)}+\|\phi_{1}\|_{C^{4,\alpha}_{2}(Bl_{X}M)} \Big)\|\phi_{0}-\phi_{1}\|_{C^{4,\alpha}_{\delta}(Bl_{X}M)},
    \end{split}
    \end{equation}
    for any $\|\phi_{0}\|_{C^{4,\alpha}_{2}(Bl_{X}M)},\|\phi_{1}\|_{C^{4,\alpha}_{2}(Bl_{X}M)}\leq c$.
\end{proposition}

Therefore, when the constant $c$ in Proposition \ref{Qestimate} is chosen sufficiently small, we have the contraction property.
\begin{corollary}[{\cite[Lemma 8.18]{Sze2014}}]\label{Contrac}
    Suppose $\delta\in(-1,0)$, there exists a constant $c>0$ independent of $\epsilon$ such that for any $\|\phi_{i}\|_{C^{4,\alpha}_{2}(Bl_{X}M)}\leq c$,
    \begin{equation}
        \|\mathcal{N}_{\epsilon}(\phi_{1})-\mathcal{N}_{\epsilon}(\phi_{2})\|_{C^{4,\alpha}_{\delta}(Bl_{X}M)}\leq \frac{1}{2}\|\phi_{1}-\phi_{2}\|_{C^{4,\alpha}_{\delta}(Bl_{X}M)}.
    \end{equation}
\end{corollary}

We then estimate $S(\omega)-S(\omega_{\epsilon})$ in the spirit of \cite[Lemma 8.19]{Sze2014}.

\begin{proposition}\label{CompareS}
    Suppose $\delta\in(-1,0)$, and $\epsilon$ is fixed sufficiently small, then there exists a constant $C>0$ independent of $\epsilon$ such that
    \begin{equation}\label{S_diff_est}
        \|S(\omega)-S(\omega_{\epsilon})\|_{C^{0,\alpha}_{\delta-4}(Bl_{X}M)}\leq Cr_{\epsilon}^{3-\delta}.
    \end{equation}
\end{proposition}

\begin{proof}[Proof of Proposition \ref{CompareS}]
   Fix a point $p\in X$ and a coordinate chart $B_{2r}^{k}\times B_{2s}^{m-k}$ centered at $p$, as in Lemma \ref{holosubcoor}. We establish the corresponding estimates separately in the different regions of $Bl_X M$.
   
    Case 1: When $|z_{(j)}^{j}|\sqrt{1+\sum_{i\neq j}|z_{(j)}^{i}|}\leq 2\epsilon$, we have $d(z_{(j)},w)\leq 3\epsilon$. Combining (\ref{Compare3}), (\ref{Compare3.5}), Proposition \ref{metric_Uni}, and the fact that $\omega_{\eta,\epsilon}$ is scalar flat, we have
    \begin{equation}
        \begin{split}
            &\epsilon^{4-\delta}\|(S(\omega)-S(\omega_{\epsilon}))(\epsilon Z_{(j)},\epsilon W)\|_{C^{0,\alpha}(\ti{B}_{(j),2}^{k}\times B_{2}^{m-k})}\\
            \leq& \epsilon^{4-\delta}\|S(\omega)(\epsilon Z_{(j)},\epsilon W)\|_{C^{0,\alpha}(\ti{B}_{(j),2}^{k}\times B_{2}^{m-k})}+\epsilon^{4-\delta}\|S(\omega_{\epsilon})-S(\omega_{\eta,\epsilon})(\epsilon Z_{(j)},\epsilon W)\|_{C^{0,\alpha}(\ti{B}_{(j),2}^{k}\times B_{2}^{m-k})}\\
            \leq& C_{1}\epsilon^{4-\delta}+C_{2}\epsilon^{4-\delta}\epsilon^{-2}\Big(|\epsilon Z_{(j)}^{j}|(1+\sum_{l\neq j}|\epsilon Z_{(j)}^{l}|^{2})^{\frac{1}{2}}+|\epsilon W|\Big)\\
            \leq& C\epsilon^{3-\delta}.
        \end{split}
    \end{equation}
    
    Case 2: Fix $\epsilon\leq r\leq 2R_{0}\epsilon$, then $\frac{\epsilon}{2}\leq d\leq 3R_{0}\epsilon$. According to (\ref{Compare3.5}), and Proposition \ref{metric_Uni}, we have
    \begin{equation}
        \begin{split}
            &r^{4-\delta}\|(S(\omega)-S(\omega_{\epsilon}))(r\ti{Z},r\ti{W})\|_{C^{0,\alpha}(B_{2}^{k}\setminus B_{1}^{k}\times B_{2}^{m-k})}\\
            \leq& C_{1}r^{4-\delta}+r^{4-\delta}\|(S(\omega_{\epsilon})-S(\omega_{\eta,\epsilon}))(r\ti{Z},r\ti{W})\|_{C^{0,\alpha}(B_{2}^{k}\setminus B_{1}^{k}\times B_{2}^{m-k})}\\
            \leq& C_{1}r^{4-\delta}+C_{2}r^{4-\delta}r^{-2}(|r\ti{Z}|+|r\ti{W}|)\\
            \leq& C\epsilon^{3-\delta}.
        \end{split}
    \end{equation}
    
    Case 3: When $2R_{0}\epsilon\leq r\leq \frac{4}{3}r_{\epsilon}$, we have $R_{0}\epsilon\leq d\leq 2r_{\epsilon}$. Then according to (\ref{Compare5}),
    \begin{equation}
        \begin{split}
            &r^{4-\delta}\|(S(\omega)-S(\omega_{\epsilon}))(r\ti{Z},r\ti{W})\|_{C^{0,\alpha}(B_{2}^{k}\setminus B_{1}^{k}\times B_{2}^{m-k})}\\
            \leq& C_{1}r^{4-\delta}+C_{2}r^{4-\delta}\epsilon^{2k-2}r^{-2k}(|r\ti{Z}|+|r\ti{W}|)\\
            \leq& Cr_{\epsilon}^{3-\delta}.
        \end{split}
    \end{equation}
    
    Case 4: When $\frac{4}{3}r_{\epsilon}\leq r\leq 4r_{\epsilon}$ and $R_0 \epsilon \leq d \leq 2\epsilon$, we apply (\ref{Compare1}) and estimate
    \begin{equation}
        \begin{split}
             &r^{4-\delta}\|(S(\omega)-S(\omega_{\epsilon}))(r\ti{Z},r\ti{W})\|_{C^{0,\alpha}(B_{2}^{k}\setminus B_{1}^{k}\times B_{2}^{m-k})}\leq C_{1}r^{4-\delta}\epsilon^{2k-2}r^{-2k}\leq Cr_{\epsilon}^{3-\delta}.
        \end{split}
    \end{equation}

    Note that in Cases 3 and 4, we use $r_{\epsilon}=\epsilon^{\frac{k}{k+1}}$.
    
    Case 5: When $r>4r_{\epsilon}$, we may assume that $d>2r_{\epsilon}$, then $\omega=\omega_{\epsilon}$ in this region. In view of these estimates, we have the desired conclusion.
\end{proof}

\begin{proposition}\label{Solv}
    Given any $\delta \in (-1,0)$. For every $\epsilon$ sufficiently small, (\ref{Equation0}) admits a solution $\phi$ which satisfies $\|\phi\|_{C^{4,\alpha}_{\delta}(Bl_{X}M)}\leq c\epsilon^{2-\delta}$ with the constant $c$ as in Corollary \ref{Contrac}.
\end{proposition}
\begin{proof}[Proof of Proposition \ref{Solv}]
    Note that $\|\phi\|_{C^{4,\alpha}_{\delta}(Bl_{X}M)}\leq c\epsilon^{2-\delta}$ implies that $\phi\in C^{4,\alpha}_{2}(Bl_{X}M)$. In fact, we get $\|\phi\|_{C^{4,\alpha}_{2}(Bl_{X}M)}\leq c$. According to Corollary \ref{Contrac}, $\mathcal{N}_{\epsilon}$ satisfies the contraction property when restricted in the $C^{4,\alpha}_{\delta}$-norm ball $B_{c\epsilon^{2-\delta}}$. It remains to show that $\mathcal{N}_{\epsilon}: B_{c\epsilon^{2-\delta}}\to B_{c\epsilon^{2-\delta}}$.

    Note that $\mathcal{N}_{\epsilon}(0)=(L_{\omega_{\epsilon}})^{-1}(S(\omega)-S(\omega_{\epsilon}))$
    Then, according to Proposition \ref{CompareS} and \ref{Inverse}, we have
    \begin{equation}
        \|\mathcal{N}_{\epsilon}(0)\|_{C^{4,\alpha}_{\delta}(Bl_{X}M)}\leq Cr_{\epsilon}^{3-\delta}=C\epsilon^{2-\delta}\epsilon^{\frac{k}{k+1}(3-\delta)-2+\delta}
    \end{equation}
    We have that $\frac{k}{k+1}(3-\delta)-2+\delta>0$ as $k\geq 3$ and $\delta>-1$. Then, according to Corollary \ref{Contrac},
    \begin{equation}
        \begin{split}
            \|\mathcal{N}_{\epsilon}(\phi)\|_{C^{4,\alpha}_{\delta}(Bl_{X}M)}
            \leq& \|\mathcal{N}_{\epsilon}(\phi)-\mathcal{N}_{\epsilon}(0)\|_{C^{4,\alpha}_{\delta}(Bl_{X}M)}+\|\mathcal{N}_{\epsilon}(0)\|_{C^{4,\alpha}_{\delta}(Bl_{X}M)}\\
            \leq& \frac{1}{2}\|\phi\|_{C^{4,\alpha}_{\delta}(Bl_{X}M)}+Cr_{\epsilon}^{3-\delta}.
        \end{split}
    \end{equation}
    Take $\epsilon$ sufficiently small such that $C\epsilon^{\frac{k}{k+1}(3-\delta)-2+\delta}\leq\frac{1}{2}c$. Then for any $\|\phi\|_{C^{4,\alpha}_{\delta}(Bl_{X}M)}\leq c\epsilon^{2-\delta}$,
    \begin{equation}
        \|\mathcal{N}_{\epsilon}(\phi)\|_{C^{4,\alpha}_{\delta}(Bl_{X}M)}\leq c\epsilon^{2-\delta}.
    \end{equation}
    Therefore, by the contraction mapping theorem, $\mathcal{N}_{\epsilon}$ has a fixed point $\phi$ that solves (\ref{Equation0}). 
\end{proof}

\section{The case of codimension \texorpdfstring{$2$}{2}}
\label{section6}

When $k=2$, the logarithmic term in \eqref{backmetric2} introduces additional technical difficulties in deriving an estimate analogous to \eqref{S_diff_est}. Consequently, a direct application of the contraction mapping theorem as in Proposition \ref{Solv} is not available. However, we observe that the undesirable asymptotic behavior is confined to a small ``annular" region, and that the terms responsible for it can be singled out. To that end, we consider the following modified equation
\begin{equation}\label{Equation2}
\begin{split}
    S(\omega)-S(\omega_{\epsilon}+\sqrt{-1}\di\dbar\phi)=-\sum_{i=1}^{d}\phi(q_{i})f_{i}&+\Big(1-\gamma(\frac{2d}{\epsilon^{\theta}})\Big)\Big(S(\omega)-S(\omega_{\epsilon})\Big)\\
    &-\gamma(\frac{2d}{\epsilon^{\theta}})\chi_{\epsilon^{\beta}}Q_{\omega}\Bigl(\epsilon^{2}\gamma\bigl(\frac{d}{\epsilon^{\theta}}\bigr)\log\bigl(\frac{d^{2}}{\epsilon^{2}}\bigr)\Bigr).
\end{split}
\end{equation}
Here $\chi_{\epsilon^{\beta}}: [0, \infty) \rightarrow [0, 1]$ is a smooth cut-off function such that
\begin{equation}
    \chi_{\epsilon^{\beta}}(d)=
    \begin{cases}
      1, \ 2\epsilon^{\beta}<d<\epsilon^{\theta}, \\
      0, \ d<\epsilon^{\beta}, \ d>2\epsilon^{\theta},
    \end{cases}
\end{equation}
and $\beta$ is any number in $(\frac{1}{2},\frac{2}{3})$. Recall that $\theta \in (0, \frac{1}{2})$ is introduced in \eqref{backmetric2}.

To solve $\phi$ such that $S(\omega_{\epsilon}+\sqrt{-1}\di\dbar \phi)$ is of the same sign with $S(\omega)$, we need to show that the $C^{0}$ norms of the right-hand side in \eqref{Equation2} tend to $0$ when $\epsilon\to 0$. We estimate these terms in the following.

\begin{proposition}\label{cut_S-S}
For any $\theta\in(0,\frac{1}{2})$, we have
\begin{equation}
    \Big\| \Big(1-\gamma(\frac{2d}{\epsilon^{\theta}})\Big)\Big(S(\omega)-S(\omega_{\epsilon})\Big) \Big\|_{C^{0}(Bl_{X}M)}\to 0\ \ \text{as}\ \epsilon\to 0.
\end{equation}
\end{proposition}
\begin{proof}[Proof of Proposition \ref{cut_S-S}]
By Propositions \ref{metric_Uni} and \ref{UniEst},
\begin{equation}
    \begin{split}
        &\Big\| (1-\gamma(\frac{2d}{\epsilon^{\theta}}))\Big(S(\omega)-S(\omega_{\epsilon})\Big) \Big\|_{C^{0}(Bl_{X}M)}\\
        \leq& C_{0}[\log\frac{1}{\epsilon}]^{\delta}(\epsilon^{\theta})^{\delta-4}\|S(\omega)-S(\omega_{\epsilon})\|_{\ti{C}^{0}_{-4,\delta}(\{ \frac{1}{2}\epsilon^{\theta}\leq d\leq \epsilon^{\theta} \})} \\
        \leq& C_{1}[\log\frac{1}{\epsilon}]^{\delta}(\epsilon^{\theta})^{\delta-4}\| \epsilon^{2}\gamma\bigl(\frac{d}{\epsilon^{\theta}}\bigr)\log\frac{d^{2}}{\epsilon^{2}} \|_{\ti{C}^{4}_{0,\delta}(\{ \frac{1}{2}\epsilon^{\theta}\leq d\leq \epsilon^{\theta} \})}\\
        \leq& C\epsilon^{2-4\theta}\log\frac{1}{\epsilon}\to 0 \ \text{as}\ \epsilon\to 0.
    \end{split}
\end{equation}
In the last step, we use an estimate on $\| \epsilon^{2}\gamma\bigl(\frac{d}{\epsilon^{\theta}}\bigr)\log\frac{d^{2}}{\epsilon^{2}} \|_{\ti{C}^{4}_{0,\delta}(\{ \frac{1}{2}\epsilon^{\theta}\leq d\leq \epsilon^{\theta} \})}$. See \eqref{Sec6_need2} in Appendix \ref{App_D} for details.
\end{proof}

We next state a lemma analogous to \cite[Lemma 8.18]{Sze2014}; see also Proposition \ref{Qestimate}. It will be used to estimate the last term on the right-hand side of \eqref{Equation2}. Since its proof is essentially the same as that of \cite[Lemma 8.18]{Sze2014}, we defer the details to Appendix \ref{PQ2}.

\begin{definition}[Weighted H\"older space in a compact domain]\label{WeightedM-X}
Let $l \geq 0$ and $t$ be integers, $\alpha \in(0,1)$, and $\delta\in \R$. For any $0\leq r_{1}<r_{2}$. The weighted H\"{o}lder space $\ti{C}^{l,\alpha}_{t,\delta}(\{r_{1}\leq d \leq r_{2} \}  )$ is defined for all $f\in C^{4,\alpha}(\{r_{1}\leq d \leq r_{2} \})$ with the finite weighted norm defined as follows.
\begin{equation}
\begin{split}
     \|f\|_{\ti{C}^{l,\alpha}_{t,\delta}(\{r_{1}\leq d\leq r_{2}\})}=
      \sup\limits_{p\in X}\sup\limits_{ \frac{2}{3}r_{1}\leq r\leq 2r_{2}} 
      \sup\limits_{0\leq |I|\leq l}
      &[\ti{\tau}(r)]^{-\delta}\Big\{ 
    [\tau(r)]^{t}\|\di^{I}f\|_{C^{0}((B_{2r}^{2}\setminus B_{r}^{2}\times B_{2r}^{m-2})_{p}\cap \{r_{1}\leq d \leq r_{2}  \})}   \\
    &+[\tau(r)]^{t+\alpha}\|\di^{I}f\|_{C^{\alpha}((B_{2r}^{2}\setminus B_{r}^{2}\times B_{2r}^{m-2})_{p}\cap \{r_{1}\leq d \leq r_{2}  \})} \Big\}.
\end{split}
\end{equation}
Here $I=(i_1, \cdots, i_m)$ with $\sum_{k=1}^m i_k=|I|$ is a multi-index. Note that $\{ r_{1} \leq d \leq r_{2}\} \subset \cup_{p\in X}\Big(B_{2r_{2}}^{2}\setminus B_{\frac{2}{3}r_{1}}^{2}\times B_{2r_{2}}^{m-2}\Big)$. Moreover, $\tau(r)\equiv \epsilon$ when $r\leq\epsilon$ and $r$ can be taken to be $0$ validly. 

Similarly, the definition of $\|f\|_{C^{l,\alpha}_{\delta}( \{ r_{1} \leq d \leq r_{2} \} )}$ follows by replacing the weight function from $\ti{\tau}$ to $\tau$.
\end{definition}

For any fixed $\epsilon>0$, the above weighted H\"older space $\ti{C}^{l,\alpha}_{t,\delta}(\{r_{1}\leq d\leq r_{2} \})$ is a Banach space since it is equivalent to the usual H\"{o}lder space $C^{l,\alpha}(\{r_{1}\leq d\leq r_{2} \})$.

\begin{lemma}[Based on {\cite[Lemma 8.18]{Sze2014}}]\label{Q2}
    Take $0<r_{1}<r_{2}$, and $\vphi,\psi\in C^{4,\alpha}_{2}(\{r_{1}\leq d \leq r_{2}\})$ such that
    \begin{equation}
        \|\vphi\|_{C^{4,\alpha}_{2}(\{r_{1}\leq d \leq r_{2}\})}, \ \|\psi\|_{C^{4,\alpha}_{2}(\{r_{1}\leq d \leq r_{2}\})}\leq c
    \end{equation}
    for some $c>0$ sufficiently small.
    Fix any $\delta\in\R$, view $Q_{\omega}:\ti{C}^{4,\alpha}_{0,\delta}(\{r_{1}\leq d \leq r_{2} \})\to \ti{C}^{0,\alpha}_{-4,\delta}(\{r_{1}\leq d \leq r_{2} \})$ as a map between Banach spaces, we have the following estimate:
    \begin{equation}
       \begin{split}
        &\|Q_{\omega}(\vphi)-Q_{\omega}(\psi)\|_{\ti{C}^{0,\alpha}_{-4,\delta}(\{r_{1}\leq d \leq r_{2}\})}\\
        \leq & C(\|\vphi\|_{C^{4,\alpha}_{2}(\{r_{1}\leq d \leq r_{2} \})}+\|\psi\|_{C^{4,\alpha}_{2}(\{r_{1}\leq d \leq r_{2} \})})\|\vphi-\psi\|_{\ti{C}^{4,\alpha}_{0,\delta}(\{r_{1}\leq d\leq r_{2}\})}.
      \end{split}
    \end{equation}
    The constant $C>0$ can be taken such that it does not depend on $r_{1},r_{2}$ and $\delta$.

    The result naturally extends to whole $Bl_{X}M$. Namely, when $\|\vphi\|_{C^{4,\alpha}_{2}(Bl_{X}M)}, \ \|\psi\|_{C^{4,\alpha}_{2}(Bl_{X}M)}\leq c$ and $c$ is sufficiently small, we have
    \begin{equation}
        \|Q_{\omega}(\vphi)-Q_{\omega}(\psi)\|_{\ti{C}^{0,\alpha}_{-4,\delta}(Bl_{X}M)}\leq C(\|\vphi\|_{C^{4,\alpha}_{2}(Bl_{X}M)}+\|\psi\|_{C^{4,\alpha}_{2}(Bl_{X}M)})\|\vphi-\psi\|_{\ti{C}^{4,\alpha}_{0,\delta}(Bl_{X}M)}.
    \end{equation}
\end{lemma}

\begin{proposition}\label{cut_Qlog}
Fix any $\beta\in (\frac{1}{2},\frac{2}{3})$. Assume that $\delta$ satisfies $\max(-1, \frac{6\beta-4}{\beta-\theta}) <\delta <0$. Then we have
\begin{equation}
    \Big\|-\gamma(\frac{2d}{\epsilon^{\theta}})\chi_{\epsilon^{\beta}}Q_{\omega}\Bigl(\epsilon^{2}\gamma\bigl(\frac{d}{\epsilon^{\theta}}\bigr)\log\bigl(\frac{d^{2}}{\epsilon^{2}}\bigr)\Bigr)\Big\|_{C^{0}(Bl_{X}M)}\to 0,\ \ \text{as}\ \epsilon \to 0.
\end{equation}
\end{proposition}

\begin{proof}[Proof of Proposition \ref{cut_Qlog}]
We only need to consider when $\epsilon^{\beta}\leq d\leq \epsilon^{\theta}$. By \eqref{Sec6_need1} in Appendix \ref{App_D}, 
\begin{equation}
    \| \epsilon^{2}\log\bigl(\frac{d^{2}}{\epsilon^{2}}\bigr) \|_{C^{4,\alpha}_{2}(\{ \epsilon^{\beta}\leq d\leq \epsilon^{\theta} \})}\leq C\epsilon^{2-2\beta}\log\frac{1}{\epsilon}\to 0, \ (\epsilon\to 0).
\end{equation}
We may choose $\epsilon_{0}>0$ sufficiently small so that $C_{3}\epsilon^{2-2\beta}\log\frac{1}{\epsilon}\leq c$ for every $\epsilon\in(0,\epsilon_{0}]$, where $c$ is the constant required in Lemma \ref{Q2}. Recall \eqref{Sec6_need1} in Appendix \ref{E_log}.
\begin{equation}
    \| \epsilon^{2}\log\bigl(\frac{d^{2}}{\epsilon^{2}}\bigr) \|_{\ti{C}^{4,\alpha}_{0,\delta}(\{ \epsilon^{\beta}\leq d\leq \epsilon^{\theta} \})}\leq C\epsilon^{2-\theta\delta}[\log\frac{1}{\epsilon}]^{1-\delta}.
\end{equation}
Then, according to Lemma \ref{Q2}, when $\delta$ is sufficiently close to $0$, we may derive
\begin{equation}
\begin{split}
    &\Big\|-\gamma(\frac{2d}{\epsilon^{\theta}})\chi_{\epsilon^{\beta}}Q_{\omega}\Bigl(\epsilon^{2}\gamma\bigl(\frac{d}{\epsilon^{\theta}}\bigr)\log\bigl(\frac{d^{2}}{\epsilon^{2}}\bigr)\Bigr)\Big\|_{C^{0}(Bl_{X}M)}\\
    \leq& C[\log\frac{1}{\epsilon}]^{\delta}(\epsilon^{\beta})^{\delta-4} \Big\|Q_{\omega}\Bigl(\epsilon^{2}\gamma\bigl(\frac{d}{\epsilon^{\theta}}\bigr)\log\bigl(\frac{d^{2}}{\epsilon^{2}}\bigr)\Bigr)\Big\|_{\ti{C}^{0,\alpha}_{-4,\delta}(\{\epsilon^{\beta}\log\frac{1}{\epsilon}\leq d\leq\epsilon^{\theta}\log\frac{1}{\epsilon}\})} \\
    \leq& C_{0}[\log\frac{1}{\epsilon}]^{\delta}(\epsilon^{\beta})^{\delta-4} \| \epsilon^{2}\log\bigl(\frac{d^{2}}{\epsilon^{2}}\bigr) \|_{C^{4,\alpha}_{2}(\{ \epsilon^{\beta}\leq d\leq \epsilon^{\theta} \})} \cdot \| \epsilon^{2}\log\bigl(\frac{d^{2}}{\epsilon^{2}}\bigr) \|_{\ti{C}^{4,\alpha}_{0,\delta}(\{ \epsilon^{\beta}\leq d\leq \epsilon^{\theta} \})}\\
    \leq& C_{1}\epsilon^{4-6\beta+(\beta-\theta)\delta}[\log\frac{1}{\epsilon}]^{2}\to 0, \ (\epsilon\to 0).
\end{split}
\end{equation}
\end{proof}

In view of the linearized operator $L_{\omega_{\epsilon}}$ defined in \eqref{def_L}, \eqref{Equation2} can be written as
\begin{equation}\label{Equation2_2}
    \gamma(\frac{2d}{\epsilon^{\theta}})\Big(S(\omega)-S(\omega_{\epsilon}) +\chi_{\epsilon^{\beta}}Q_{\omega}(\epsilon^{2}\gamma\bigl(\frac{d}{\epsilon^{\theta}}\bigr)\log\bigl(\frac{d^{2}}{\epsilon^{2}}\bigr))\Big)=\ti{L}_{\omega_{\epsilon}}\phi+Q_{\omega_{\epsilon}}\phi.
\end{equation}
We introduce
\begin{equation}
    F_{\epsilon}= \gamma(\frac{2d}{\epsilon^{\theta}})\Big(S(\omega)-S(\omega_{\epsilon}) +\chi_{\epsilon^{\beta}}Q_{\omega}(\epsilon^{2}\gamma\bigl(\frac{d}{\epsilon^{\theta}}\bigr)\log\bigl(\frac{d^{2}}{\epsilon^{2}}\bigr))\Big).
\end{equation}
To solve (\ref{Equation2_2}), we seek a fixed point of the map
\begin{equation}
\begin{split}
        \mathcal{N}_{\epsilon}:\ &\ti{C}^{4,\alpha}_{0,\delta}(Bl_{X}M) \to \ti{C}^{4,\alpha}_{0,\delta}(Bl_{X}M) \\
        &\phi\mapsto
        \ti{L}_{\omega_{\epsilon}}^{-1}\Big(F_{\epsilon}-Q_{\omega_{\epsilon}}(\phi)\Big),
\end{split}
\end{equation}
More precisely, we look for a fixed point $\phi$ satisfying $\|\phi \|_{\ti{C}^{4,\alpha}_{0,\delta}(Bl_{X}M,\omega_{\epsilon})}
\longrightarrow 0$ as $\epsilon\to 0$. Moreover, for all sufficiently small $\epsilon$,
$\|\phi \|_{C^{4,\alpha}_{2}(Bl_{X}M,\omega_{\epsilon})}\leq c$, where $c>0$ is a sufficiently small constant, independent of $\epsilon$, chosen so that the estimates in Proposition \ref{UniEst} and Lemma \ref{Q2} apply.

As a crucial step in the contraction mapping argument, we estimate $\|F_{\epsilon}\|_{\ti{C}^{0,\alpha}_{-4,\delta}}$ in the following proposition.

\begin{proposition}\label{F}
    There exist constants $C>0$ and $\kappa>0$ independent of sufficiently small $\epsilon$, such that 
    \begin{equation}
       \|F_{\epsilon}\|_{\ti{C}^{0,\alpha}_{-4,\delta}(Bl_{X}M)}\leq C\epsilon^{2-\delta+\kappa}. 
    \end{equation}
\end{proposition}

\begin{proof}[Proof of Proposition \ref{F}]
    When $2\epsilon^{\beta}<d<\epsilon^{\theta}$, we have
    \begin{equation}
    F_{\epsilon}=\gamma(\frac{2d}{\epsilon^{\theta}})\Big(S(\omega)-S(\omega_{\epsilon}) +\chi_{\epsilon^{\beta}}Q_{\omega}(\epsilon^{2}\gamma\bigl(\frac{d}{\epsilon^{\theta}}\bigr)\log\bigl(\frac{d^{2}}{\epsilon^{2}}\bigr))\Big)\\
        =-\gamma(\frac{2d}{\epsilon^{\theta}})L_{\omega}\big( \epsilon^{2}\log\frac{d^{2}}{\epsilon^{2}} \big).
   \end{equation}
    According to \eqref{F_epsi_est1} in Appendix \ref{EF}, we have
    \begin{equation}
        \|F_{\epsilon}\|_{\ti{C}^{0,\alpha}_{-4,\delta}(\{2\epsilon^{\beta}<d<\epsilon^{\theta}\})}\leq C\epsilon^{2-\delta}\epsilon^{\theta+(1-\theta)\delta}(\log\frac{1}{\epsilon})^{-\delta}.
    \end{equation}

    When $\epsilon^{\beta}\leq d\leq 2\epsilon^{\beta}$,
    \begin{equation}
        F_{\epsilon}=S(\omega)-S(\omega_{\epsilon})+\chi_{\epsilon^{\beta}}Q_{\omega}\Big(\epsilon^{2}\gamma\bigl(\frac{d}{\epsilon^{\theta}}\bigr)\log\bigl(\frac{d^{2}}{\epsilon^{2}}\bigr)\Big).
    \end{equation}
    According to \eqref{BSCompare} in Lemma \ref{MC2}, we have
    \begin{equation}
    \begin{split}
        \|S(\omega)-S(\omega_{\epsilon})\|_{\ti{C}^{0,\alpha}_{-4,\delta}(\{\epsilon^{\beta}\leq d\leq 2\epsilon^{\beta} \})}
        \leq& \|S(\omega)\|_{\ti{C}^{0,\alpha}_{-4,\delta}(\{\epsilon^{\beta}\leq d\leq2\epsilon^{\beta} \})}+\|S(\omega_{\epsilon})-S(\omega_{\epsilon,\eta})\|_{\ti{C}^{0,\alpha}_{-4,\delta}(\{\epsilon^{\beta}\leq d\leq2\epsilon^{\beta} \})}\\
        \leq& C[\log\frac{1}{\epsilon}]^{-\delta}(\epsilon^{\beta})^{(4-\delta)}=C\epsilon^{2-\delta}\epsilon^{(4\beta-2)+(1-\beta)\delta}[\log\frac{1}{\epsilon}]^{-\delta}.
    \end{split}
    \end{equation}
    According to Lemma \ref{Q2} and calculations in section \ref{E_log},
    \begin{equation}
    \begin{split}
        &\|Q_{\omega}\Big(\epsilon^{2}\gamma\bigl(\frac{d}{\epsilon^{\theta}}\bigr)\log\bigl(\frac{d^{2}}{\epsilon^{2}}\bigr)\Big)\|_{\ti{C}^{0,\alpha}_{-4,\delta}(\{\epsilon^{\beta}\leq d\leq 2\epsilon^{\beta}\})}\\
        \leq& C_{0}\|\epsilon^{2}\log\frac{d^{2}}{\epsilon^{2}}\|_{C^{4,\alpha}_{2}(\{\epsilon^{\beta}\leq d\leq 2\epsilon^{\beta}\})}\|\epsilon^{2}\log\frac{d^{2}}{\epsilon^{2}}\|_{\ti{C}^{4,\alpha}_{0,\delta}(\{\epsilon^{\beta}\leq d\leq 2\epsilon^{\beta}\})} \\
        \leq& C[\epsilon^{2-2\beta}\log\frac{1}{\epsilon}]\cdot[\epsilon^{2-\delta\beta}(\log\frac{1}{\epsilon})^{-\delta}] \\
        =& C\epsilon^{2-\delta}\epsilon^{(2-2\beta)+(1-\beta)\delta}(\log\frac{1}{\epsilon})^{1-\delta}
    \end{split}
    \end{equation}

When $d<\epsilon^{\beta}$,
\begin{equation}
    F_{\epsilon}=S(\omega)-S(\omega_{\epsilon}).
\end{equation}
Then, by Lemma \ref{MC2} and Proposition \ref{metric_Uni},
\begin{equation}
\begin{split}
    \|F_{\epsilon}\|_{\ti{C}^{0,\alpha}_{-4,\delta}(\{d\leq\epsilon^{\beta} \})}
    \leq& \|S(\omega)\|_{\ti{C}^{0,\alpha}_{-4,\delta}(\{d\leq\epsilon^{\beta} \})}+\|S(\omega_{\epsilon})-S(\omega_{\eta,\epsilon})\|_{\ti{C}^{0,\alpha}_{-4,\delta}(\{d\leq\epsilon^{\beta} \})} \\
    \leq& C_{1}[\log\frac{1}{\epsilon}]^{-\delta}(\epsilon^{\beta})^{4-\delta}+[\log\frac{1}{\epsilon}]^{-\delta}(\epsilon^{\beta})^{4-\delta}(C_{2}+C_{3}\epsilon^{2}) \\
    \leq& C\epsilon^{\beta(4-\delta)}[\log\frac{1}{\epsilon}]^{-\delta}
    =C\epsilon^{2-\delta}\epsilon^{(4\beta-2)+(1-\beta)\delta}[\log\frac{1}{\epsilon}]^{-\delta}.
\end{split}
\end{equation}
Recall that $\beta\in(\frac{1}{2},\frac{2}{3})$ and we may fix $\delta\in(-1,0)$ sufficiently close to $0$. we can take
\begin{equation}
    \kappa=\frac{1}{2}\min\{(4\beta-2)+(1-\beta)\delta, \ (2-2\beta)+(1-\beta)\delta, \ \theta+(1-\theta)\delta\},
\end{equation}
which yields the conclusion.
\end{proof}

The estimates in Lemma \ref{Q2} ensure that we can follow the proof of \cite[Lemma 8.18]{Sze2014} to obtain the following contraction property. 
\begin{corollary}\label{Contrac2}
    Suppose $\delta\in(-1,0)$, there exists a constant $c>0$ independent of $\epsilon$ such that for any $\|\phi_{i}\|_{C^{4,\alpha}_{2}(Bl_{X}M)}\leq c$, the following holds.
    \begin{equation}
        \|\mathcal{N}_{\epsilon}(\phi_{1})-\mathcal{N}_{\epsilon}(\phi_{2})\|_{\ti{C}^{4,\alpha}_{0,\delta}(Bl_{X}M)}\leq \frac{1}{2}\|\phi_{1}-\phi_{2}\|_{\ti{C}^{4,\alpha}_{0,\delta}(Bl_{X}M)}.
    \end{equation}
\end{corollary}

Combining these estimates, we apply the contraction mapping principle to solve \eqref{Equation2}.

\begin{proposition}\label{Solv2}
Suppose that $\delta\in(-1,0)$. Then, for every sufficiently small $\epsilon$, \eqref{Equation2} admits a solution $\phi$ satisfying
\(
\|\phi\|_{\ti{C}^{4,\alpha}_{0,\delta}(Bl_{X}M)}
\leq \frac{1}{3}c\epsilon^{2-\delta}[\log\frac{1}{\epsilon}]^{-\delta},
\)
where the constant $c$ is sufficiently small such that the estimates in Proposition \ref{UniEst}, Lemma \ref{Q2} and Corollary \ref{Contrac2} are valid.
\end{proposition}

\begin{proof}[Proof of Proposition \ref{Solv2}]
    Fix $p\in X$ and a coordinate neighborhood $B_{2r_{0}}^{2}\times B_{2r_{0}}^{m-2}$ of Lemma \ref{holosubcoor} type. Since $\|\phi\|_{\ti{C}^{4,\alpha}_{0,\delta}(Bl_{X}M)}\leq \frac{1}{3}c\epsilon^{2-\delta}[\log\frac{1}{\epsilon}]^{-\delta}$, we have for any $\frac{2}{3}\epsilon\leq r\leq r_{0}$,
    \begin{align}
        & \epsilon^{-2}\|\phi(\epsilon Z_{(j)},\epsilon W)\|_{C^{4,\alpha}(\ti{B}_{(j),2}^{2}\times B_{2}^{m-2})}
        \leq \epsilon^{-2}\|\phi\|_{\ti{C}^{4,\alpha}_{0,\delta}(Bl_{X}M)}(\epsilon\log\frac{1}{\epsilon})^{\delta}\leq \frac{1}{3}c; \\
        & r^{-2}\|\phi(rZ,rW)\|_{C^{4,\alpha}(B_{2}^{2}\setminus B_{1}^{2}\times B_{2}^{m-2})}
        \leq r^{-2}\|\phi\|_{\ti{C}^{4,\alpha}_{0,\delta}(Bl_{X}M)}(r\log\frac{1}{\epsilon})^{\delta}\leq c(\frac{\epsilon}{r})^{2-\delta}\leq \frac{1}{3}(\frac{3}{2})^{2-\delta}c.
    \end{align}
    We may derive that \(
     \|\phi\|_{C^{4,\alpha}_{2}(Bl_{X}M)}\leq c.
    \) Then, according to Corollary \ref{Contrac2}, we already know that $\mathcal{N}_{\epsilon}$ is a contraction when restricted to the $\ti{C}^{4,\alpha}_{0,\delta}$ norm ball $B_{\frac{1}{3}c\epsilon^{2-\delta}[\log\frac{1}{\epsilon}]^{-\delta}}$. Now we show that $\mathcal{N}_{\epsilon}$ maps $B_{\frac{1}{3}c\epsilon^{2-\delta}[\log\frac{1}{\epsilon}]^{-\delta}}$ to itself. Note that $\mathcal{N}_{\epsilon}(0)=(L_{\omega_{\epsilon}})^{-1}(F_{\epsilon})$.
    Then, according to Proposition \ref{F} and Proposition \ref{Inverse}, we have
    \begin{equation}
        \|\mathcal{N}_{\epsilon}(0)\|_{\ti{C}^{4,\alpha}_{0,\delta}(Bl_{X}M)}\leq K\epsilon^{2-\delta+\kappa}.
    \end{equation}
    Then, according to Corollary \ref{Contrac2},
    \begin{equation}
        \begin{split}
            \|\mathcal{N}_{\epsilon}(\phi)\|_{\ti{C}^{4,\alpha}_{0,\delta}(Bl_{X}M)}
            \leq& \|\mathcal{N}_{\epsilon}(\phi)-\mathcal{N}_{\epsilon}(0)\|_{\ti{C}^{4,\alpha}_{0,\delta}(Bl_{X}M)}+\|\mathcal{N}_{\epsilon}(0)\|_{\ti{C}^{4,\alpha}_{0,\delta}(Bl_{X}M)}\\
            \leq& \frac{1}{2}\|\phi\|_{\ti{C}^{4,\alpha}_{0,\delta}(Bl_{X}M)}+K\epsilon^{\kappa}\epsilon^{2-\delta}.
        \end{split}
    \end{equation}
Choose $\epsilon>0$ sufficiently small such that
\(
K\epsilon^{\kappa}\left(\log\frac{1}{\epsilon}\right)^{\delta}\leq \frac{1}{6}c.
\)
Then, for any $\phi$ satisfying
\(
\|\phi\|_{\ti{C}^{4,\alpha}_{0,\delta}(Bl_XM)}
\leq \frac{1}{3}c\epsilon^{2-\delta}[\log\frac{1}{\epsilon}]^{-\delta},
\)
we have
\[
\|\mathcal{N}_{\epsilon}(\phi)\|_{\ti{C}^{4,\alpha}_{0,\delta}(Bl_XM)}
\leq \frac{1}{3}c\epsilon^{2-\delta}[\log\frac{1}{\epsilon}]^{-\delta}.
\]
Therefore, by the contraction mapping theorem, $\mathcal{N}_{\epsilon}$ has a fixed point in the norm ball
\(
B_{\frac{1}{3}c\epsilon^{2-\delta}[\log\frac{1}{\epsilon}]^{-\delta}}.
\)
This fixed point solves \eqref{Equation2}.
\end{proof}

Finally, we conclude the proof of Theorem \ref{MainThm_intro} stated in the introduction.

\begin{proof}[Proof of Theorem \ref{MainThm_intro}]
    When $k\geq 3$, according to Proposition \ref{Solv}, we have a K\"{a}hler metric $\omega_{\epsilon}+\sqrt{-1}\di\dbar\phi$ on $Bl_{X}M$ with $\|\phi\|_{C^{4,\alpha}_{\delta}(Bl_{X}M)}\leq c\epsilon^{2-\delta}$ that satisfies
    \begin{equation}
        S(\omega)-S(\omega_{\epsilon}+\sqrt{-1}\di\dbar\phi)=-\sum_{i=1}^{d}\phi(q_{i})f_{i}.
    \end{equation}
    We take $r_0$ in \eqref{WeightFunc3} suitably small, such that $\{q_{i}\}_{i=1}^{d}\subset M\setminus\{d>2r_0\}$. As $\|\phi\|_{C^{4,\alpha}_{\delta}(Bl_{X}M)}\leq c\epsilon^{2-\delta}$, we get
    \(
    \|\phi\|_{C^{0}(M\setminus\{d>r\})}\leq c\epsilon^{2-\delta}.
    \)
    As a result, we obtain
    \begin{equation}\label{s_close}
    \begin{split}
        \| S(\omega)-S(\omega_{\epsilon}+\sqrt{-1}\di\dbar\phi)\|_{C^{0}(Bl_{X}M)}
        \leq& \sup_{i}\|f_{i}\|_{C^{0}(M)}\|\phi\|_{C^{0}(M\setminus\{d>r\})}\leq C\epsilon^{2-\delta}.
    \end{split}
    \end{equation}
    Note that \eqref{s_close} holds for some constant $C>0$ independent of $\epsilon$. After taking $\epsilon$ sufficiently small,  $S(\omega_{\epsilon}+\sqrt{-1}\di\dbar\phi)$ has the same sign with that of $S(\omega)$. Moreover, standard elliptic regularity for second-order linear equations implies that $\phi \in C^{\infty}(Bl_{X}M)$; see \cite[Section 6.3]{Brown} for the details. Therefore, $S(\omega_{\epsilon}+\sqrt{-1}\di\dbar\phi)$ is a smooth K\"ahler metric on $Bl_{X} M$. When $k=2$, the corresponding conclusion follows by a similar argument, using Propositions \ref{cut_S-S} and \ref{cut_Qlog}.    
\end{proof}

\section{Further remarks}\label{section7}

Let $f: X \rightarrow B$ be a holomorphic submersion between two compact complex manifolds. Assume that the base $B$ is K\"ahler and each fiber $X_b = f^{-1}(b)$ of a point $b \in B$ is a Fano manifold of complex dimension $p$. The goal of this section to construct a K\"ahler metric with $S>0$ on $X$, thereby proving Proposition \ref{Fanofib}. 

The relative canonical bundle is $K_{X/B} \coloneqq K_X\otimes f^*(K_B^{-1})$. Its restriction to a fiber is the canonical bundle of the fiber $K_{X/B}|_{X_b}\cong K_{X_b}$. Since \(X_b\) is Fano, $-K_{X/Y}$ is positive along the fibers. We will find a relative K\"ahler form $\omega_{\mathrm{rel}}$ on \(X\) such that $\omega_{\mathrm{rel}}|_{X_b}$ is a Kähler form on each fiber \(X_b\). Let \(\omega_Y\) be a K\"ahler form on \(Y\). Then, for \(A>0\) sufficiently large,
\begin{equation}
\omega_X=\omega_{\mathrm{rel}}+A f^{\ast}\omega_Y       \label{Kahler_X}
\end{equation}
is a K\"ahler form on \(X\). 

The construction of $\omega_{\mathrm{rel}}$ is standard and proceeds as follows. We choose a smooth Hermitian metric \(h_0\) on \(L=-K_{X/Y}\) with its curvature form $\chi=\sqrt{-1}\,\Theta(h_0)$. As $-K_{X/B}$ is positive along the fibers, we may assume that $\chi|_{X_b}>0$ for every \(b \in B\). Let $\chi_b\coloneqq\chi|_{X_b}$. Note that
$[\chi_b]=c_1(X_b)$. Next we look for a K\"ahler metric
\(
\omega_b=\chi_b+\sqrt{-1}\partial \bar\partial \varphi_b
\) on \(X_b\) satisfying $\operatorname{Ric}(\omega_b)=\chi_b$. It amounts to solve the complex Monge-Ampère equation
\[
\bigl( \chi_b+\sqrt{-1}\partial \overline{\partial}\varphi_b \bigr)^k =e^{F_b}\chi_b^k,
\]
where \(F_b\) is determined by
\[
\operatorname{Ric}(\chi_b)-\chi_b =\sqrt{-1}\partial \overline\partial F_b\ \ \ \text{and}\ \ 
\int_{X_b}e^{F_b}\chi_b^k =\int_{X_b}\chi_b^k.
\]
By the Calabi-Yau theorem \cite{Yau1978}, it has a unique smooth solution after imposing $\int_{X_b}\varphi_b\chi_b^k=0$. Thus we obtain a K\"ahler metric $\omega_b=\chi_b+\sqrt{-1}\partial \overline\partial \varphi_b$ on $X_b$ which satisfies $\operatorname{Ric}(\omega_b)=\chi_b>0$. Since \(f:X\to Y\) is a smooth fibration, the solutions \(\varphi_b\) vary smoothly with \(b\). Hence there exists a smooth function $\varphi\in C^\infty(X)$ such that $\varphi|_{X_b}=\varphi_b$. Now we define the relative K\"ahler form in (\ref{Kahler_X}) as
\begin{equation}
\omega_{\mathrm{rel}}=\chi+\sqrt{-1}\partial\bar\partial\varphi.   \label{Kahler_rel}    
\end{equation}

\begin{proof}[Proof of Proposition \ref{Fanofib}]

    We show that the K\"ahler metric $\omega_X$ defined by (\ref{Kahler_X}) satisfies $S(\omega_X)>0$ if $A$ is chosen sufficiently large. It is equivalent to work with $\omega_{X}=t\omega_{\mathrm{rel}}+f^{*}\omega_{Y}$ when $t>0$ is sufficiently small. 
    
    Let $\dim X=n>m=\dim B$ with $n-m=k$. Recall that the scalar curvature satisfies
    \(S(\omega_{X})\omega_{X}^{n}=nRic(\omega_{X})\wedge \omega_{X}^{n-1}\). For a fixed point $q\in X_{b_0}$, choose local holomorphic coordinates 
    \[
    (U_q,(z,w))=(U_q,(z^1,\ldots,z^{n-m},w^1,\ldots,w^m))
    \] on $X$ such that
\(
X_{b_0}\cap U_q=f^{-1}(0)=\{w^1=\cdots=w^m=0\},
\)
where, in these coordinates,
\[
f(z^1,\ldots,z^{n-m},w^1,\ldots,w^m)=(w^1,\ldots,w^m).
\]
Then we further assume
  \begin{align}
        &(\omega_{\mathrm{rel}}|_{X_{b}})^{n-m}(z,w)=(\sqrt{-1})^{n-m}F(z,w)dz^{1}\wedge d\ol{z^{1}}\wedge\cdots\wedge dz^{n-m}\wedge d\ol{z^{n-m}},\\
        &(f^{*}\omega_{Y})^{m}(z,w)=(\sqrt{-1})^{m} G(w)dw^{1}\wedge d\ol{w^{1}}\wedge\cdots\wedge dw^{m}\wedge d\ol{w^{m}}.
       \end{align}
    Note that
    \[
     (\omega_{X})^n=\sum_{p=0}^m \binom{n}{p} t^{n-p} (\omega_{\mathrm{rel}})^{n-p} \wedge (f^{\ast}\omega_Y)^p.
    \]
    When $t\to 0$, $\omega_{X}^{n}|_{U_q}$ can be estimated
    \begin{equation}
        t^{n-m}\omega_{rel}^{n-m}\wedge (f^{*}\omega_{Y})^{m}=(\sqrt{-1})^{n} t^{n-m}F(z,w)G(w)dz^{1}\wedge d\ol{z^1}\cdots d w^{m} \wedge d\ol{w^{m}}+O(t^{n-m+1}).
    \end{equation}
    Similarly, at the fixed point $q \in X_{b_0}$, we may estimate $nRic(\omega_{X})\wedge \omega_{X}^{n-1}$.
    \begin{equation*}
    \begin{split}
        &-n\sqrt{-1}\di\dbar\log\Big(\sum_{p=0}^m \binom{n}{p} t^{n-p} (\omega_{\mathrm{rel}})^{n-p} \wedge (f^{\ast}\omega_Y)^p\Big)\wedge t^{n-m-1}\omega_{\mathrm{rel}}^{n-m-1}\wedge (f^{*}\omega_{Y})^{m}|_{w=0}\\
        =&-n\sqrt{-1}Ric(\omega_{\mathrm{rel}}|_{X_{b_0}})\wedge t^{n-m-1}\omega_{\mathrm{rel}}^{n-m-1}\wedge (f^{*}\omega_{Y})^{m}|_{w=0}\\
        =&S(\omega_{\mathrm{rel}}|_{X_{b_0}})\frac{n}{n-m}t^{n-m-1}\omega_{\mathrm{rel}}^{n-m}\wedge (f^{*}\omega_{Y})^{m}+O(t^{n-m})
    \end{split}
    \end{equation*}
    Note that $S(\omega_{\mathrm{rel}}|_{X_{b_0}})>0$ as $Ric(\omega_{\mathrm{rel}}|_{X_{b_0}})=\chi_{b_0}>0$, we conclude that
    \begin{align}
      S(\omega_{X})(q)=\frac{n}{n-m} S(\omega_{\mathrm{rel}}|_{X_{b_0}}) \frac{1}{t}+O(1). \label{S_posi_pt}  
    \end{align}
    In fact, (\ref{S_posi_pt}) holds uniformly in $U_q$. By the compactness of $X$, we may choose a finite cover by relatively compact open sets. It then follows that for all sufficiently small $t>0$ we have $S(\omega_{X})>0$ on $X$.

\end{proof}

\begin{remark}\label{Rem_OnYau}
    After obtaining the relative K\"ahler form in \eqref{Kahler_rel}, the proof of Proposition \ref{Fanofib} proceeds similarly to the proof of \cite[Proposition 1]{Yau1974}. Compare the following remark in \cite[p.~218]{Yau1974}. Let $X$ be the total space of holomorphic fiber bundle over a compact K\"ahler manifold $Y$ with compact fibers. Assume that $X$ admits a K\"ahler metric such that each fiber has positive Ricci curvature with respect to the induced metric. Then $X$ admits a K\"ahler metric with $S>0$.    
\end{remark}

\appendix

\section{Holomorphic ``normal" coordinates near a submanifold}\label{holo_nor_app}

\begin{lemma}[\cite{SS2020}]\label{holosubcoor}
    Let $(M^m, \omega)$ be a K\"ahler manifold and $X$ be a complex submanifold of codimension $k$. Let $d(q)=\inf_{p\in X}d(p,q)$. We follow the convention that $d$ is $\frac{1}{\sqrt{2}}$ of the Riemannian distance with respect to $\omega$. Then for any $p \in X$, there exist holomorphic chart
    \[
    (U_p; (z^{1}, \dots, z^{k},w^{1}, \dots, w^{m-k}))
    \]
    such that $z^{i}(U_{p}\cap X)=0$ for any $1 \leq  i \leq k$. Moreover, there exist two smooth functions $\rho(z, w)$ and $\phi(z, w)$ with
    \begin{align}
        &\omega=\sqrt{-1}\di\dbar(|z|^{2}+|w|^{2}+\phi(z, w)),\label{LOC.metric} \ \  \phi(z,w)=O(|z|^{3}+|w|^{3});\\
        & d^{2}(z, w)=|z|^{2}(1+\rho(z, w)), \ \ \rho(z, w)=O(|z|+|w|).\label{LOC.distance}
    \end{align}
    In addition, if $X$ is compact, there exists some constant $r_0>0$ independent to $p \in X$ such that $\{(z, w) \mid |z|<r_0, |w|<r_0\} \subset U_p$ for each $p$. And the estimates for $\phi$ and $\rho$ in \eqref{LOC.metric} and \eqref{LOC.distance} can be chosen uniformly with respect to the point $p\in X$.
\end{lemma}

\begin{proof}[Proof of Lemma \ref{holosubcoor}]

It is stated without proof in {\cite[Lemma 5, p.~175]{SS2020}}. We include a proof for the reader's convenience. It conists of three steps.

Step 1: Construction of holomorphic coordinates.

Fix any $p\in X$ and choose any holomorphic coordinate neighborhood chart $(U,(Z,W))$ such that
\begin{equation}\label{adapted_chart}
    (Z(p),W(p))=0, \quad Z^{i}(U \cap X)=0, \ i=1,...,k.
\end{equation}
The corresponding metric matrix in these coordinates will be denoted, for example, by
\begin{equation*}
    g_{Z^{i}\ol{W^{j}}}=\omega(\frac{\di}{\di Z^{i}},\frac{\di}{\di \ol{W^{j}}}).
\end{equation*} 
If we choose another holomorphic coordinate $(V,(\widetilde{Z},\widetilde{W}))$ which satisfies the above condition, then the transition function satisfies 
\[
\frac{\di Z^{\beta}}{\di \widetilde{W}^{j}}(p')=0, \ \forall p'\in V \cap X.
\]
The corresponding metric matrix satisfies
\begin{equation}
g_{\widetilde{Z}^{i}\ol{\widetilde{W}^{j}}}=\ol{\frac{\di W^{\beta}}{\di \widetilde{W}^{j}}}(\frac{\di Z^{\alpha}}{\di \widetilde{Z}^{i}}g_{Z^{\alpha}\ol{W^{\beta}}}+\frac{\di W^{\alpha}}{\di \widetilde{Z^{i}}}g_{W^{\alpha}\ol{W^{\beta}}}).  \label{relation}
\end{equation}
Let $h^{W^{\alpha}\ol{W}^{\beta}}(p)$ denote the transpose of the inverse matrix of $(g_{W^{i}\ol{W}^{j}}(p))_{(m-k)\times (m-k)}$ in the sense that 
\(
g_{W^{i}\ol{W^{\beta}}}(p)h^{W^{j}\ol{W^{\beta}}}(p)=\delta_{i}^{j}.
\)
Now we choose a holomorphic coordinate $(\tilde{Z},\tilde{W})$ in an open neighborhood of $p$, denoted by $\widetilde{U_p} \subset U \cap V$.
\begin{equation}\label{adapted_chart2}
  \left\{
    \begin{aligned}
        &\tilde{Z}^{\alpha}=Z^{\alpha},\\
        &\tilde{W}^{\beta}=W^{\beta}+g_{Z^{i}\ol{W}^{k}}(p)h^{W^{\beta}\ol{W^{k}}}(p)Z^{i}.
    \end{aligned}
  \right.
\end{equation}
Then $g_{\widetilde{Z}^{i}\ol{\widetilde{W}^{j}}}(p)=g_{\tilde{W}^{i}\ol{\tilde{Z}^{j}}}(p)=0$ by (\ref{relation}). Therefore, the metric components of $\omega$ with respect to $(\widetilde{Z},\widetilde{W})$ are block diagonal. If we further choose constant unitary matrices $A_{k\times k}, \ B_{(m-k) \times (m-k)}$ such that
\begin{equation*}
\begin{pmatrix}
    A^{T} & 0 \\
    0 & B^{T}
\end{pmatrix}
\begin{pmatrix}
    g_{\widetilde{Z}^{i}\ol{\widetilde{Z}^{j}}}(p) & 0 \\
    0 & g_{\widetilde{W}^{i}\ol{\widetilde{W}^{j}}}(p)
\end{pmatrix}
\begin{pmatrix}
    \ol{A} & 0 \\
    0 & \ol{B}
\end{pmatrix}
=
\begin{pmatrix}
    I_{k} & 0\\
    0 & I_{m-k}
\end{pmatrix},
\end{equation*}
Now we choose the holomorphic coordinate in $U_p$ as
\begin{equation}\label{adapted_chart3}
    z=A^{-1}\widetilde{Z},\ \ \ 
    w=B^{-1}\widetilde{W}.
\end{equation}
It is straightforward to check
\begin{equation}\label{normal_est_1}
\omega(z,w)=\sqrt{-1}\di\dbar(|z|^{2}+|w|^{2}+\phi(z,w)), \ \phi(z,w)=O(|z|^{3}+|w|^{3}).  
\end{equation}

Step 2: Estimates on the distance function.

Shrinking $U_p$ constructed in Step 1 if necessary, we may assume that, for any $q \in U_{p} \setminus X$, there exists a unique $p'(q)\in U_{p}\cap X$ and a geodesic $\gamma$ from $p'(q)$ to $q$ with $d(q)=L(\gamma)$. Here, we point out that, by an abuse of notation, $d(q)$ denotes the \emph{Riemannian distance} from $q$ to the submanifold $X$ with respect to $\omega$. Obviously $\gamma^{\prime}$ is orthogonal to $T X$ at $p'(q)$. We consider the holomorphic coordinate $(U_{p}, z, w)$ we constructed in (\ref{normal_est_1}). After choosing the arc length parametrization of $\gamma: [0,d(q)]\rightarrow{M}$, we may assume that the following holds at $\gamma(0)=p'(q) \in X \cap U_p$.
\begin{align}
    (z(\gamma(0)),w(\gamma(0)))=(0,w'),\ \ \text{and}\  
    \gamma'(0)=\Gamma^{i}\frac{\di}{\di z^{i}}|_{\gamma(0)}+\ol{\Gamma^{i}}\frac{\di}{\di \ol{z^{i}}}|_{\gamma(0)}+\Lambda^{j}\frac{\di}{\di w^{j}}|_{\gamma(0)}+\ol{\Lambda^{j}}\frac{\di}{\di \ol{w^{j}}}|_{\gamma(0)}.
\end{align}
Then we consider Taylor expansions of components of $\gamma$ as the following.
\begin{align}
       & z^{i}(q)=z^{i}(\gamma(d(q)))=\Gamma^{i}d(q)+O(d^{2}(q)),\\
       & w^{\alpha}(q)=w^{\alpha}(\gamma(d(q)))=w'+\Lambda^{\alpha}d(q)+O(d^{2}(q)).
\end{align}
Therefore
\begin{align} 
        |z(q)|^{2}=(\sum_{i=1}^{k}|\Gamma^{i}|^{2})d^{2}(q)+O(d^{3}(q)),\ \text{and} \ 
        |w(q)-w'|=O(d(q))   \label{CurveExpan.}
\end{align} hold for $q \in U_p$.

On the other hand, we have expansions of metric coefficients according to (\ref{normal_est_1}).
\begin{align}\label{MetricExpan.}
        &g_{z^{i}\ol{z^{j}}}(0,w')=\delta_{i\ol{j}}+\di_{w^{\alpha}}g_{z^{i}\ol{z^{j}}}(p)(w')^{\alpha}+\di_{\ol{w}^{\alpha}}g_{z^{i}\ol{z^{j}}}(p)(\ol{w'})^{\alpha}+O(|w'|^{2}),\\
        &g_{w^{i}\ol{w}^{j}}(0,w')=\delta_{i\ol{j}}+\di_{w^{\alpha}}g_{w^{i}\ol{w^{j}}}(p)(w')^{\alpha}+\di_{\ol{w}^{\alpha}}g_{w^{i}\ol{w^{j}}}(p)(\ol{w'})^{\alpha}+O(|w'|^{2}),\\
        &g_{z^{i}\ol{w}^{j}}(0,w')=\di_{w^{\alpha}}g_{z^{i}\ol{w^{j}}}(p)(w')^{\alpha}+\di_{\ol{w}^{\alpha}}g_{z^{i}\ol{w^{j}}}(p)(\ol{w'})^{\alpha}+O(|w'|^{2}).
\end{align}
Recall that $\|\gamma'(0)\|=1$ and $\gamma'(0)\perp T_{\gamma(0)}X=\text{span}\{\frac{\di}{\di w^{j}}\,\frac{\di}{\di \ol{w}^{k}}\}$, It follows from (\ref{MetricExpan.}) that
\begin{equation} \label{esti_tangent}
    \begin{split}
        1=\|\gamma'(0)\|^{2}=&g(\gamma'(0),\Gamma^{i}\frac{\di}{\di z^{i}}|_{\gamma(0)}+\ol{\Gamma^{i}}\frac{\di}{\di \ol{z}^{i}}|_{\gamma(0)})\\
        =&\Gamma^{i}\ol{\Gamma^{j}}g_{z^{i}\ol{z^{j}}}(\gamma(0))+\ol{\Gamma^{i}\ol{\Gamma^{j}}g_{z^{i}\ol{z^{j}}}(\gamma(0))}+\Gamma^{i}\ol{\Lambda^{j}}g_{z^{i}\ol{w^{j}}}(\gamma(0))+\ol{\Gamma^{i}\ol{\Lambda^{j}}g_{z^{i}\ol{w^{j}}}(\gamma(0))}\\
        =&2|\Gamma|^{2}+O(|w'|)\\
        =&2|\Gamma|^{2}+O(|w(q)|+d(q)).
    \end{split}
\end{equation}
After plugging (\ref{esti_tangent}) into (\ref{CurveExpan.}), we have
\begin{equation}
    \begin{split}
        2|z(q)|^{2}=&d^{2}(q)+(2|\Gamma|^{2}-\|\gamma'(0)\|^{2})d^{2}(q)+O(d^{3}(q))\\
        =& d^{2}(q)+O(|w(q)|)d^{2}(q)+O(d^{3}(q))\\
        =& d^{2}(q)(1+O(|w(q)|)+O(d(q))).
    \end{split}
\end{equation}
After taking the neighborhood $U_{p}$ sufficiently small, we have
\begin{equation}\label{normal_est_2}
d^{2}(q)=2|z(q)|^{2}+O(|w(q)|)|z(q)|^{2}+O(|z(q)|^{3})
        =2|z(q)|^{2}[1+O(|z(q)|+|w(q)|)].
\end{equation}
Then the desired estimate (\ref{LOC.distance}) follows in view of our convention on $d^2$.

Step 3: The uniformity of sizes of holomorphic charts and related estimates when $X$ is compact.

For each point $p \in X$, we have a holomorphic chart $(U_{\alpha}; \Phi_{\alpha})$ with $\Phi_{\alpha}: U_{\alpha} \rightarrow \mathbb{C}^n$ be the coordinate map given by \eqref{adapted_chart}. Choose a open subset $V_{\alpha}$ which is relatively compact in $U_{\alpha}$. As $X$ is compact, we may assume $V_{1}, \cdots, V_{N}$ be a finite cover of $X$, with center of each $p_i \in X$. We choose 
\begin{equation}\label{def_delta_co}
\delta=\frac{1}{3}\min_{1 \leq i \leq N} d_{E}(\Phi_{i}(\overline{V_i}), \partial \Phi_{i}(U_i))>0,    
\end{equation} where $d_E$ is the Euclidean distance on $\mathbb{C}^n$.

For each $q \in X$, we must have $q \in V_{k}$ with $1 \leq k \leq N$, then we define a specific holomorphic chart based on a translation of $(U_k, \Phi_{k})$. Indeed, $(U_k, \widetilde{\Phi}_k \coloneqq \Phi_{k}-\Phi_k(q))$ defines a holomorphic chart centered at $q$ which satisfies \eqref{adapted_chart}. By the definition of $\delta$ in \eqref{def_delta_co}, we conclude that
\(
\{(z, w) \mid |z|<r_0, |w|<r_0\} \subset \widetilde{\Phi}_k(U_k).
\)
From now on, we only considered holomorphic charts in the form of $(U_k, \widetilde{\Phi}_k)$ obtained by translation of $(V_i, \Phi_i)$ with $1 \leq i \leq N$. Then the remaining argument in Step 1, \eqref{adapted_chart2} and \eqref{adapted_chart2} in particular, involves a finite number of continuous functions on $p$. Moreover, there exists a constant $\widetilde{\delta}>0$ such that the normal exponential map $\exp_{\perp}$ is a diffeomorphism onto the metric tubular neighborhood $\{q \in M\ |\ d(q) \leq \widetilde{\delta}\}$. The estimate on the distance function $d$ in Step 2 can be carried out with corresponding coefficients continuous on $p$. To sum up, the size of the resulting holomorphic charts $U_p$, as well as relevant estimates in \eqref{normal_est_1} and \eqref{normal_est_2} can be made uniform on $p$.

\end{proof}

\section{Estimates on background metrics}\label{GE}

\subsection{The proof of Lemma \ref{MC3}}

We show some detailed calculations in the proof of Lemma \ref{MC3}.

\medskip

\noindent\textbf{Case 1.} $R_{0}\epsilon<d<4r_{\epsilon}$ where $R_0$ is defined in (\ref{BSpotential}).
   
    In this case, $\gamma(\frac{d}{\epsilon})=0$ since $R_{0}\epsilon>>2\epsilon$, the following holds.
    \begin{equation}
        \omega_{\epsilon}-\omega=\epsilon^{2}\sqrt{-1}\di\dbar\Big(\gamma(\frac{d}{r_{\epsilon}})\psi(\frac{d^{2}}{\epsilon^{2}})\Big).
    \end{equation}
    As $\epsilon$ is small enough, it follows from (\ref{diff_d_z}) that
    \begin{equation}\label{Dz_cut_re}
        \di_{z}^{I}\gamma(\frac{d}{r_{\epsilon}})=
            \begin{cases} \gamma'(\frac{d}{r_{\epsilon}})\dfrac{\di_{z}^{I}d}{r_{\epsilon}}=O(\dfrac{1}{r_{\epsilon}}), \ \text{if}\ |I|=1,\\
            O(\dfrac{1}{r^{|I|}_{\epsilon}}), \ \text{if}\ \ |I|\geq 2.
        \end{cases}
    \end{equation}
    On the other hand,
    \begin{equation}
        \di_{w}^{J}\gamma(\frac{d}{r_{\epsilon}})=O(1),  \ \forall J. \label{Dw_cut_re}
    \end{equation}
    Moreover, we may derive the following from by (\ref{BSpotential})
   \begin{equation}
    \begin{split}
    \di_{z^{i}}\Big(\psi(\frac{d^{2}}{\epsilon^{2}}) \Big)
    =& \psi'(\frac{d^{2}}{\epsilon^{2}})\frac{\di_{z^{i}}d^{2}}{\epsilon^{2}}\\
    =&\Big((2-k)(\frac{d^{2}}{\epsilon^{2}})^{1-k}+O((\frac{d^{2}}{\epsilon^{2}})^{-k}) \Big)\frac{1}{\epsilon^{2}}\Big( \ol{z}^{i}(1+\rho)+|z|^{2}\di_{z^{i}}\rho \Big)\\
    =& O(1)\epsilon^{2k-4}d^{3-2k}.
    \end{split}
    \end{equation}
    On the other hand
    \begin{equation}
        \di_{w^{j}}\Big(\psi(\frac{d^{2}}{\epsilon^{2}}) \Big)=\psi'(\frac{d^{2}}{\epsilon^{2}})\frac{|z|^{2}}{\epsilon^{2}}\di_{w^{j}}(1+\rho)
        =O\Big((\frac{d}{\epsilon})^{2-2k} \Big)
    \end{equation}
    By induction, we obtain the following.
    \begin{align}
        &\di_{z}^{I}\Big(\psi(\frac{d^{2}}{\epsilon^{2}}) \Big)=O(1)\epsilon^{2k-4}d^{4-2k-|I|}, \ \forall |I|\geq 1,\label{Dz_BSp_re} \\
        &\di_{w}^{J}\Big(\psi(\frac{d^{2}}{\epsilon^{2}}) \Big)=O\Big((\frac{d}{\epsilon})^{2-2k} \Big), \ \forall J. \label{Dw_BSp_re}
    \end{align}
    Therefore, we may estimate the coefficients of $\omega_{\epsilon}-\omega$ as the following.
    \begin{align}
            &g_{\omega_\epsilon}-g_{\omega}=O(\epsilon^{2k-2}d^{2-2k}),\\
            &\di_{z}^{I}\di_{w}^{J}(g_{\omega_\epsilon}-g_{\omega})=O(\epsilon^{2k-2}d^{2-2k-|I|}), \ \text{if}\ |I|\geq 1.
    \end{align}

    In this case, a comparison with $\omega_{\eta,\epsilon}$ is also available and we will use it later. For simplicity, we suppose $R_{0}\epsilon<d<\frac{1}{2}r_{\epsilon}$ and $\frac{R_{0}}{2}\epsilon<|z|<r_{\epsilon}$, then
    \begin{equation}
        (\omega_{\epsilon}-\omega_{\eta,\epsilon})(z,w)=\sqrt{-1}\di\dbar\Big(\phi(z,w)+\epsilon^{2}[\psi(\frac{d^{2}}{\epsilon^{2}})-\psi(\frac{|z|^{2}}{\epsilon^{2}})]\Big).
    \end{equation}
    We calculate, for example,
    \begin{equation}
        \begin{split}
            \di_{z^{i}}\Big(\psi(\frac{d^{2}}{\epsilon^{2}})-\psi(\frac{|z|^{2}}{\epsilon^{2}})\Big)
            =& \Big(\psi'(\frac{d^{2}}{\epsilon^{2}})-\psi'(\frac{|z|^{2}}{\epsilon^{2}})\Big)\epsilon^{-2}\di_{z^{i}}d^{2}+\psi'(\frac{|z|^{2}}{\epsilon^{2}})\epsilon^{-2}\di_{z^{i}}(d^{2}-|z|^{2})\\
            =& O\Big(\epsilon^{-2}d[(\frac{d^{2}}{\epsilon^{2}})^{1-k}-(\frac{|z|^{2}}{\epsilon^{2}})^{1-k}] \Big)+O\Big((\frac{|z|^{2}}{\epsilon^{2}})\epsilon^{-2}\di_{z^{i}}(|z|^{2}\rho) \Big)\\
            =&O\Big(\epsilon^{2k-4}|z|^{3-2k}(|z|+|w|)\Big),
        \end{split}
    \end{equation}
    and
    \begin{equation}
        \di_{w^{j}}\Big(\psi(\frac{d^{2}}{\epsilon^{2}})-\psi(\frac{|z|^{2}}{\epsilon^{2}})\Big)=O\Big(\epsilon^{4-2k}|z|^{4-2k}(|z|+|w|) \Big).
    \end{equation}
    By induction, we have
    \begin{equation}
        \di_{z}^{I}\di_{w}^{J}\Big(\psi(\frac{d^{2}}{\epsilon^{2}})-\psi(\frac{|z|^{2}}{\epsilon^{2}})\Big)=O\Big(\epsilon^{4-2k}|z|^{4-2k-|I|}(|z|+|w|) \Big).
    \end{equation}
    Then,
    \begin{align}
        &g_{\omega_{\epsilon}}-g_{\omega_{\eta,\epsilon}}=O(|z|+|w|)+O\Big(\epsilon^{2-2k}|z|^{2-2k}(|z|+|w|) \Big)  \\
        &\di_{z}^{I}\di_{w}^{J}(g_{\omega_{\epsilon}}-g_{\omega_{\eta,\epsilon}})=O\Big(\epsilon^{2-2k}|z|^{2-2k-|I|}(|z|+|w|) \Big).
    \end{align}

\medskip

\noindent\textbf{Case 2.} $\epsilon<d<R_{0}\epsilon$.

    In this case, $\gamma(\frac{d}{r_{\epsilon}})\equiv 1$ since we may assume that $r_{\epsilon}>R_{0}\epsilon$. By (\ref{LOC.metric}) and (\ref{ScaledPMetric}), the following hold.
    \begin{equation}
        (\omega_{\epsilon}-\omega_{\eta,\epsilon})(z,w)=\sqrt{-1}\di\dbar\Big(\phi(z,w)+\epsilon^{2}[\gamma(\frac{d}{\epsilon})\log (\frac{d}{\epsilon})^{2}-\gamma(\frac{|z|}{\epsilon})\log (\frac{|z|}{\epsilon})^{2}]+\epsilon^{2}[\psi(\frac{d^{2}}{\epsilon^{2}})-\psi(\frac{|z|^{2}}{\epsilon^{2}})]\Big).
    \end{equation}
  First, the $\log$ terms are considered by:
    \begin{equation}
    \begin{split}
        \epsilon^{2}\Big(\gamma(\frac{d}{\epsilon})\log (\frac{d}{\epsilon})^{2}-\gamma(\frac{|z|}{\epsilon})\log (\frac{|z|}{\epsilon})^{2}\Big)
        =&\epsilon^{2}\gamma(\frac{d}{\epsilon})\log\frac{d^{2}}{|z|^{2}}+\epsilon^{2}(\gamma(\frac{d}{\epsilon})-\gamma(\frac{|z|}{\epsilon}))\log\frac{|z|^{2}}{\epsilon^{2}}\\
        =&\epsilon^{2}\gamma(\frac{d}{\epsilon})\log(1+\rho(z,w))+\epsilon^{2}\Big(\gamma(\frac{d}{\epsilon})-\gamma(\frac{|z|}{\epsilon})\Big)\log\frac{|z|^{2}}{\epsilon^{2}}.
    \end{split}  \label{term1case2}
    \end{equation}
    In order to estimate the derivatives of (\ref{term1case2}), we may check that the following hold.
    \begin{align}
     &\di_{z}^{I}\gamma(\frac{d}{\epsilon})=\sum_{i=1}^{|I|}O(\frac{1}{\epsilon^{i}d^{|I|-i}})= O(\frac{1}{d^{|I|}}), \ |I|\geq 0,   \label{dzcutoff}\\
     &\di_{w}^{J}\gamma(\frac{d}{\epsilon})=O(1), \forall J, \\
     &\di_{z}^{I}\log(\frac{d}{\epsilon})=O(\frac{1}{d^{|I|}}), \ |I|\geq 0,\\
     &\di_{w}^{J}\log(\frac{d}{\epsilon})=O(1), \ \forall J,\label{dwlog}
    \end{align}
    Note that in (\ref{dzcutoff}), the derivatives of $\gamma(\frac{d}{\epsilon})$ is nonzero only when $\epsilon<d <2\epsilon$. Next, we apply the mean value inequality for derivatives of $\gamma(\frac{d}{\epsilon})-\gamma(\frac{|z|}{\epsilon})$, for example, when $|I|=1$:
    \begin{equation}
        \begin{split}
            \Big|\di_{z^{i}}\Big(\gamma(\frac{d}{\epsilon})-\gamma(\frac{|z|}{\epsilon})\Big)\Big|
            =&\Big|\gamma'(\frac{d}{\epsilon})\frac{\di_{z^{i}}(d-|z|)}{\epsilon}+(\gamma'(\frac{d}{\epsilon})-\gamma'(\frac{|z|}{\epsilon}))\frac{\di_{z^{i}}|z|}{\epsilon})\Big|\\
            \leq & \sup_{[1,2]}|\gamma'|\Big|\frac{\di_{z^{i}}[|z|(\sqrt{1+\rho}-1)]}{\epsilon}\Big|+\sup_{[1,2]}|\gamma''|\Big|\frac{|z|(\sqrt{1+\rho}-1)\di_{z^{i}}|z|}{\epsilon^{2}}\Big| \\
            \leq& C_{1} |\gamma|_{C^{2}([1,2])} \frac{1}{\epsilon} \Big|\rho\di_{z^{i}}|z|+|z|\frac{\di_{z_i} \rho}{2\sqrt{1+\rho}}+ \frac{|z|\rho \, \di_{z_i} |z|}{\epsilon}\Big| \\
            \leq & C_{2}\epsilon^{-1}|\rho|+C_{3}\epsilon^{-1}|z|+C_{4}\epsilon^{-1}|\rho| \\
            \leq & Cd^{-1}(|z|+|w|).
        \end{split}  \label{case2base}
    \end{equation}
    Then by induction,
    \begin{equation}
        \di_{z}^{I}\Big(\gamma(\frac{d}{\epsilon})-\gamma(\frac{|z|}{\epsilon})\Big)=d^{-|I|}O\Big((|z|+|w|) \Big), \ |I|\geq 2.
    \end{equation}
    As a consequence, we obtain the following.
    \begin{equation}
        \begin{split}
            &\di_{z}^{I}\Big[\epsilon^{2}(\gamma(\frac{d}{\epsilon})\log (\frac{d}{\epsilon})^{2}-\gamma(\frac{|z|}{\epsilon})\log (\frac{|z|}{\epsilon})^{2})(z,w)\Big]\\
            =&\sum_{|J|+|K|=|I|}\epsilon^{2}\di_{z}^{J}[\gamma(\frac{d}{\epsilon})]\di_{z}^{K}\log(1+\rho)
            +\sum_{|J|+|K|=|I|}\epsilon^{2}\di_{z}^{J}\Big(\gamma(\frac{d}{\epsilon})-\gamma(\frac{|z|}{\epsilon})\Big)\di_{z}^{K}\log\frac{|z|^{2}}{\epsilon^{2}}\\
            =&\epsilon^{2}\Big( \di_{z}^{I}[\gamma(\frac{d}{\epsilon})]\log(1+\rho)+\sum_{\substack{|J|+|K|=|I|\\ |K|\neq 0}}\di_{z}^{J}[\gamma(\frac{d}{\epsilon})]\frac{\di_{z}^{K}\rho}{1+\rho} \Big)
            +\epsilon^{2}O\Big(\frac{1}{\epsilon^{|I|}}(|z|+|w|)\Big)\\
            =& O(d^{2-|I|}\rho)+O(d^{2-|I|+1}) +O(d^{2-|I|}(|z|+|w|))\\
            =&O\Big(d^{2-|I|}(|z|+|w|)\Big), \ |I|\geq 2.
        \end{split}
    \end{equation} 
    On the other hand
    \begin{equation}\label{Dw_log_e}
        \di_{w}^{J}\Big[\epsilon^{2}(\gamma(\frac{d}{\epsilon})\log (\frac{d}{\epsilon})^{2}-\gamma(\frac{|z|}{\epsilon})\log (\frac{|z|}{\epsilon})^{2})(z,w)\Big]
        =\di_{w}^{J}\epsilon^{2}(\gamma(\frac{d}{\epsilon})\log (\frac{d}{\epsilon})^{2})
        =O(1).
    \end{equation}
    We proceed to estimate the second term in the right hand side of (\ref{case2base}) in the following.
    \begin{equation}\label{Dz_log_e}
        \begin{split}
            \epsilon^{2}\Big|\di_{z^{i}}\Big(\psi(\frac{d^{2}}{\epsilon^{2}})-\psi(\frac{|z|^{2}}{\epsilon^{2}})\Big)\Big|
            =&\epsilon^{2}\Big|\psi'(\frac{d^{2}}{\epsilon^{2}})\frac{\di_{z^{i}}(d^{2}-|z|^{2})}{\epsilon^{2}}+\Big[\Big(\psi'(\frac{d^{2}}{\epsilon^{2}})-\psi'(\frac{|z|^{2}}{\epsilon^{2}})\Big)\frac{\di_{z^{i}}|z|^{2}}{\epsilon^{2}}\Big]\Big|\\
            \leq & \epsilon^{2}\sup_{[0,R_{0}]}|\psi'|\frac{|\di_{z^{i}}(|z|^{2}\rho)|}{\epsilon^{2}}+\epsilon^{2}\sup_{[0,R_{0}]}|\psi''|\frac{|z|^{2}|\rho\ol{z}^{i}|}{\epsilon^{4}}\\
            \leq &|\psi|_{C^{2}([0,R_{0}])} \Big|\rho \ol{z}^{i}+|z|^{2}\di_{z}^{i}\rho+\frac{|z|^{2}\rho\ol{z}^{i}}{\epsilon^2}\Big| \\
            \leq & Cd(|z|+|w|).
        \end{split}
    \end{equation}
    By induction, we have:
    \begin{align}
        &\di_{z}^{I}\epsilon^{2}\Big(\psi(\frac{d^{2}}{\epsilon^{2}})-\psi(\frac{|z|^{2}}{\epsilon^{2}})\Big)=O\Big(d^{2-|I|}(|z|+|w|)\Big), \ |I|\geq 1, \label{Dz_BSp_e} \\
        &\di_{w}^{J}\epsilon^{2}\Big(\psi(\frac{d^{2}}{\epsilon^{2}})-\psi(\frac{|z|^{2}}{\epsilon^{2}})\Big)=O(d^{2}),\ |J|\geq 1.\label{Dw_BSp_e}
    \end{align}
    Therefore, from the above analysis:
    \begin{align}
       &(g_{\omega_{\epsilon}}-g_{\omega_{\eta,\epsilon}})(z,w)=O\Big( (|z|+|w|) \Big),  \\
       &\di_{z}^{I}\di_{w}^{J}(g_{\omega_\epsilon}-g_{\omega_{\eta,\epsilon}})(z,w)=O(1)+O\Big( d^{-|I|}(|z|+|w|) \Big), \ \text{if \ } |I|\geq 1. 
    \end{align} 
    Note that the $O(1)$ term in the right hand side of (\ref{Compare3.5}) follows from (\ref{LOC.metric}). And we point out that these estimates actually hold in $\epsilon<d<R\epsilon$ for any sufficiently large $R>0$ independent of $\epsilon$.

\medskip

\noindent\textbf{Case 3.} $d<\epsilon$.

    In this case, we use blow up coordinates (\ref{BLcoord}) in direction $z^{i}$ with respect to (\ref{LOC.metric}), then the difference is:
    \begin{equation}
    \begin{split}
        (\omega_{\epsilon}-\omega_{\eta,\epsilon})(\zi,w)=&\sqrt{-1}\di\dbar\Big(\sigma^{*}\phi(\zi,w)+\epsilon^{2}[\log (\frac{\sigma^{*}d}{\epsilon})^{2}-\gamma(\frac{\sigma^{*}|\zi|}{\epsilon})\log (\frac{\sigma^{*}|\zi|}{\epsilon})^{2}]\\
        &+\epsilon^{2}[\psi(\frac{\sigma^{*}d^{2}}{\epsilon^{2}})-\psi(\frac{\sigma^{*}|\zi|^{2}}{\epsilon^{2}})]\Big).
    \end{split}
    \end{equation}
    Notice that after pulling back
    \begin{align}
    &\sigma^{*}|\zi|^{2}=|\zi^{i}|^{2}(1+\sum_{i\neq j}|\zi^{j}|^{2}).\\
    &\sigma^{*}d^{2}(\zi,w)=|\zi^{i}|^{2}(1+\sum_{j\neq i}|\zi^{j}|^{2})(1+\sigma^{*}\rho).\\
    &\di_{\zi}^{I}\epsilon^{2}\log(1+\sigma^{*}\rho(\zi,w))=O(\epsilon^{2}), \ |I|\geq 0,
    \end{align}
    In blow up coordinates (\ref{BLcoord}), when $|\zi^{i}|>0$ and $|I|\geq1$, the estimate (\ref{Dz_BSp_e}) and (\ref{Dw_BSp_e}) become
    \begin{equation}\label{Dzw_BSp_E}
     \begin{split}   \di_{\zi}^{I}\di_{w}^{J}\epsilon^{2}\Big(\psi(\frac{\sigma^{*}d^{2}}{\epsilon^{2}})-\psi(\frac{|\zi^{i}|^{2}(1+\sum_{j\neq i}|\zi^{j}|^{2})}{\epsilon^{2}})\Big)
     =O\Big(\epsilon^{2-|I|}(|\zi^{i}|(1+\sum_{j\neq i}|\zi^{j}|^{2})^{\frac{1}{2}}+|w|)\Big), 
     \end{split}
    \end{equation}
    and the estimate naturally extend to the exceptional divisor when $|\zi^{i}|\to 0$.

    Then we calculate, for example:
    \begin{equation}
        \begin{cases}
            \di_{\zi^{i}}\gamma(\frac{\sigma^{*}|\zi|}{\epsilon})=\gamma'(\frac{\sigma^{*}|\zi|}{\epsilon})\frac{(1+\sum_{j\neq i}|\zi^{j}|^{2})^{\frac{1}{2}}}{\epsilon}\frac{\ol{\zi^{i}}}{2|\zi^{i}|}=O(\epsilon^{-1});\\
            \di_{\zi^{j}}\gamma(\frac{\sigma^{*}|\zi|}{\epsilon})=\gamma'(\frac{\sigma^{*}|\zi|}{\epsilon})\frac{|\zi^{i}|}{\epsilon}\frac{\ol{\zi^{j}}}{2(1+\sum_{j\neq i}|\zi^{j}|^{2})^{\frac{1}{2}}}=O(1), \ j\neq i.
        \end{cases}
    \end{equation}
    By induction
    \begin{equation}
        \di_{\zi^{i}}^{s}\di_{\zi^{j}}^{t}\Big(\gamma(\frac{\sigma^{*}|\zi|}{\epsilon})\Big)=-\di_{\zi^{i}}^{s}\di_{\zi^{j}}^{t}\Big(1-\gamma(\frac{\sigma^{*}|\zi|}{\epsilon})\Big)
        = O(\epsilon^{-s}), \ \forall j\neq i
    \end{equation}
    and the derivatives are non-zero only if $\epsilon<|\zi^{i}|(1+\sum_{j\neq i}|\zi^{j}|)^{\frac{1}{2}}$. In this case, more specifically, if
    \begin{equation}
        \epsilon<|z|=|\zi^{i}|(1+\sum_{j\neq i}|\zi^{j}|)^{\frac{1}{2}}=\frac{\sigma^{*}d}{(1+\sigma^{*}\rho)^{\frac{1}{2}}}<(1+\sigma^{*}\rho)^{-\frac{1}{2}}\epsilon,
    \end{equation}
    is valid, we have
    \begin{equation}
            |1-\gamma(\frac{\sigma^{*}|\zi|}{\epsilon})|\leq |\gamma|_{C^{1}(\mathbb{R})}\frac{|\epsilon-\sigma^{*}|\zi||}{\epsilon}
            \leq  |\gamma|_{C^{1}(\mathbb{R})}\Big|1-(1+\sigma^{*}\rho)^{-\frac{1}{2}} \Big|
            \leq  \ti{C}\frac{1}{2}|\gamma|_{C^{1}(\mathbb{R})}|\sigma^{*}\rho|.
    \end{equation}
    Then, we may derive the following estimate
    \begin{equation}
        \begin{split}
            &\di_{\zi}^{I}\epsilon^{2}[\log (\frac{\sigma^{*}d}{\epsilon})^{2}-\gamma(\frac{\sigma^{*}|\zi|}{\epsilon})\log (\frac{\sigma^{*}|\zi|}{\epsilon})^{2}]\\
            =&\epsilon^{2}\di_{\zi}^{I}\log(1+\sigma^{*}\rho)
            -\epsilon^{2}\log(\frac{\sigma^{*}|\zi|}{\epsilon})^{2}\di_{\zi}^{I}\gamma(\frac{\sigma^{*}|\zi|}{\epsilon}) \\
            &+\epsilon^{2}\sum_{\substack{J+K=I \\ 0=|J|<|I|}}\Big(\di_{\zi}^{J}[1-\gamma(\frac{\sigma^{*}|\zi|}{\epsilon})]\Big)\Big(\di_{\zi}^{K}\log (\frac{\sigma^{*}|\zi|}{\epsilon})^{2} \Big)\\
            =&O(\epsilon^{2})+\epsilon^{2}O(\log(1+\sigma^{*}\rho)\epsilon^{-|I|})+\epsilon^{2}O(|\sigma^{*}\rho|\epsilon^{-|I|})\\
            =&O(\epsilon^{2})+O\Big(\epsilon^{2-|I|}(|\zi^{i}|(1+\sum_{j\neq i}|\zi^{j}|^{2})^{\frac{1}{2}}+|w|)\Big).
        \end{split}
    \end{equation}
    When $|\zi^{i}|(1+\sum_{j\neq i}|\zi^{j}|)^{\frac{1}{2}}\leq \epsilon$, this term is simply
    \begin{equation}
        \di_{\zi}^{I}\epsilon^{2}[\log (\frac{\sigma^{*}d}{\epsilon})^{2}-\log (\frac{\sigma^{*}|\zi|}{\epsilon})^{2}]=\di_{\zi}^{I}\epsilon^{2}\log(1+\sigma^{*}\rho)=O(\epsilon^{2})
    \end{equation}
    On the other hand, the following always hold.
    \begin{align}
        &\di_{w}^{J}\epsilon^{2}\log(\frac{\sigma^{*}d}{\epsilon})^{2}=\di_{w}^{J}\epsilon^{2}\log(1+\sigma^{*}\rho)=O(\epsilon^{2}), \\
        &\di_{w}^{J}\epsilon^{2}\psi(\frac{\sigma^{*}d^{2}}{\epsilon^{2}})=O(\epsilon^{2}).
    \end{align}
    For the mixed terms, we calculate, for example:
    \begin{equation}
        \begin{split}
          \omega_{\epsilon,\zi^{i}\ol{w}^{j}}
            =&\di_{\zi^{i}}\di_{\ol{w}^{j}}\sigma^{*}\phi+\epsilon^{2}\di_{\zi^{i}}[\di_{\ol{w}^{j}}\log(1+\sigma^{*}\rho)+\psi'(\frac{\sigma^{*}d^{2}}{\epsilon^{2}})
            \frac{\di_{\ol{w}^{j}}\sigma^{*}d^{2}}{\epsilon^{2}}]\\
            =& O(|\zi^{i}|^{2}(1+\sum_{j\neq i}|\zi^{j}|^{2})+|w|^{2})+O(\epsilon^{2})\\
            &+\ol{z}_{(i)}^{i}(\di_{\ol{w}^{j}}\sigma^{*}\rho)\Big(\psi''(\frac{\sigma^{*}d^{2}}{\epsilon^{2}})\frac{\sigma^{*}d^{2}}{\epsilon^{2}}+\psi'(\frac{\sigma^{*}d^{2}}{\epsilon^{2}})(1+\sum_{j\neq i}|\zi^{j}|^{2}) \Big)+\epsilon^{2}O(1)\\
            =&O(|\zi^{i}|(1+\sum_{j\neq i}|\zi^{j}|^{2})^{\frac{1}{2}}+|w|).
        \end{split}
    \end{equation}
    Then, we derive the following
    \begin{align}
       &(g_{\omega_\epsilon}-g_{\omega_{\eta,\epsilon}})(\zi,w)=O(|\zi^{i}|(1+\sum_{j\neq i}|\zi^{j}|^{2})^{\frac{1}{2}}+|w|), \\
       &\di_{\zi}^{I}\di_{w}^{J}(g_{\omega_\epsilon}-g_{\omega_{\eta,\epsilon}})(\zi,w)=O(1)+O\Big(\epsilon^{-|I|}(|\zi^{i}|(1+\sum_{j\neq i}|\zi^{j}|^{2})^{\frac{1}{2}}+|w|)\Big), \ |I|\geq 1 .
    \end{align}

\medskip

\noindent\textbf{Case 4.} $R_{0}\epsilon<d<r_{0}$.

We also need to compare the background metric $\omega_{\epsilon}$ with the Euclidean metric.
In this case,
\begin{equation}
    \omega_{\epsilon}(z,w)=\sqrt{-1}\di\dbar\Big(|z|^{2}+|w|^{2}+\phi(z,w)+\epsilon^{2}\gamma(\frac{d}{r_{\epsilon}})[\gamma(\frac{d}{\epsilon})\log(\frac{d^{2}}{\epsilon^{2}})+\psi(\frac{d^{2}}{\epsilon^{2}})]\Big).
\end{equation}
Then, denote the Euclidean metric in this coordinates chart by $\omega_{euc}$, we have
\begin{equation}
    (\omega_{\epsilon}-\omega_{euc})(z,w)=\sqrt{-1}\di\dbar\Big(\phi(z,w)+\epsilon^{2}\gamma(\frac{d}{r_{\epsilon}})[\gamma(\frac{d}{\epsilon})\log(\frac{d^{2}}{\epsilon^{2}})+\psi(\frac{d^{2}}{\epsilon^{2}})]\Big).
\end{equation}
Then, for the same calculation of (\ref{Dz_cut_re})-(\ref{Dw_BSp_re}) and (\ref{dzcutoff})-(\ref{dwlog}) we have
\begin{equation}
    \di_{z}^{I}\di_{w}^{J}\epsilon^{2}\gamma(\frac{d}{r_{\epsilon}})\psi(\frac{d^{2}}{\epsilon^{2}})=O(\epsilon^{2k-2}d^{4-|I|-2k})
\end{equation}
Hence, when $2\epsilon<R_{0}\epsilon<d<r_{0}$, we have
\begin{align}
        &(g_{\omega_{\epsilon}}-g_{euc})(z,w)=O(|z|+|w|)+O(\epsilon^{2k-2}d^{2-2k});  \\
    &\di_{z}^{I}\di_{w}^{J}(g_{\omega_{\epsilon}}-g_{euc})(z,w)=O(1)+O(\epsilon^{2k-2}d^{2-|I|-2k}).
\end{align}

\subsection{Uniform upper bounds on weighted norms of background metric}\label{UUB}
\begin{proof}[Proof of Proposition \ref{metric_Uni}]
    We prove when $k\geq3$. The proof of $k=2$ is similar.

    In $\{d\geq r_{0}\}$,
    \begin{equation}
        \|g_{\omega_{\epsilon}}\|_{C^{2,\alpha}(M\setminus\{d<r_{0}\})}=\|g_{\omega}\|_{C^{2,\alpha}(M\setminus\{d<r_{0}\})}\leq C,
    \end{equation}
    and the same argument holds for $g_{\omega_{\epsilon}}^{-1}$.

    We may assume that $\frac{r}{2}\leq d\leq \frac{3r}{2}$. For any $p\in X$, take a coordinates chart $\big(B_{4r_{0}}^{k}\times B_{4r_{0}}^{m-k},(z,w)\big)$ centered at $p=(0,0)$. We consider a change of coordinates $(rZ,rW)=(z,w)$. For simplicity, we set 
    \begin{align}\label{r_change_nota}
       [g_{\omega_{\epsilon},i\ol{j}}]_{r}(Z,W)=g_{\omega_{\epsilon},i\ol{j}}(rZ,rW), \ \ \ \
       (Z,W)\in\mathcal{B} \coloneqq B_{2}^{k}\setminus B_{1}^{k}\times B_{2}^{m-k},
    \end{align}
    where $g_{\omega_{\epsilon},i\ol{j}}$ are local coefficients with respect to $(z,w)$. Recalling the definition in \eqref{WeightedTensor}, we have the following estimates.
    
    When $4r_{\epsilon}\leq r\leq 2r_{0}$, we have $d\geq 2r_{\epsilon}$, then
    \begin{equation}
        \|[g_{\omega_{\epsilon}}]_{r}\|_{C^{2,\alpha}(\mathcal{B})}=\|[g_{\omega}]_{r}\|_{C^{2,\alpha}(\mathcal{B})}\leq \|g_{\omega}\|_{C^{2,\alpha}(M)}.
    \end{equation}
    
    When $2R_{0}\epsilon\leq r\leq 4r_{\epsilon}$, we have $R_{0}\epsilon\leq d \leq 6r_{\epsilon}$. Then, by \eqref{Compare5},
     \begin{equation}
         \begin{split}
             \|[g_{\omega_{\epsilon}}]_{r}\|_{C^{2,\alpha}(\mathcal{B})}
             &\leq  \|[g_{\omega_{\epsilon}}-g_{\omega_{\eta,\epsilon}}]_{r}\|_{C^{2,\alpha}(\mathcal{B})} +\|[g_{\omega_{\eta,\epsilon}}]_{r}\|_{C^{2,\alpha}(\mathcal{B})}\\
            & \leq C_{1}r\sum_{i=0}^{2}\sum_{|I|=i}(1+\frac{\epsilon^{2k-2}}{d^{2k-2}}(\frac{r}{d})^{|I|})
            +\|g_{\omega_{\eta}}\|_{C^{2,\alpha}(Bl_{0}\C^{k}\times \C^{m-k})}
            \leq C.
         \end{split}
     \end{equation}

     When $2\epsilon\leq r\leq 2R_{0}\epsilon$, we have $\epsilon\leq d\leq 3R_{0}\epsilon$. Then, by \eqref{Compare3.5},
     \begin{equation}
         \begin{split}
             \|g_{\omega_{\epsilon}}\|_{C^{2,\alpha}(\mathcal{B})}
             &\leq  \|g_{\omega_{\epsilon}}-g_{\omega_{\eta,\epsilon}}\|_{C^{2,\alpha}(\mathcal{B})}
             +\|g_{\omega_{\eta,\epsilon}}\|_{C^{2,\alpha}(\mathcal{B})}\\
            & \leq C_{2}r\sum_{i=0}^{2}\sum_{|I|=i}(1+(\frac{r}{\epsilon})^{|I|})
            +\|g_{\omega_{\eta}}\|_{C^{2,\alpha}(Bl_{0}\C^{k}\times \C^{m-k})} 
            \leq C.
         \end{split}
     \end{equation}
     Since we may fix $R_{0}$ sufficiently large, $r\leq 2\epsilon$ implies that $d\leq 2R_{0}\epsilon$. The estimate when $r\leq 2\epsilon$ is basically the same as the inequality above by using \eqref{Compare3}. Thus we may conclude that $\|g_{\omega_{\epsilon}}\|_{C^{2,\alpha}_{0}(Bl_{X}M)}\leq C$. 

     As for the inverse matrix, observe that $g_{\omega_{\epsilon}}^{-1}-g_{\omega}^{-1}=g_{\omega_{\epsilon}}^{-1}(g_{\omega}-g_{\omega_{\epsilon}})g_{\omega}^{-1}$. We estimate, for example when $2R_{0}\epsilon\leq r\leq 4r_{\epsilon}$,
     \begin{equation}
         \begin{split}
             \|[g_{\omega_{\epsilon}}^{-1}]_{r} \|_{C^{2,\alpha}(\mathcal{B})}
             &\leq  \|[g_{\omega_{\epsilon}}^{-1}-g_{\omega_{\eta,\epsilon}}^{-1}]_{r}\|_{C^{2,\alpha}(\mathcal{B})} +\|[g_{\omega_{\eta,\epsilon}}^{-1}]_{r}\|_{C^{2,\alpha}(\mathcal{B})}\\
            & \leq C_{0} \Big(\|[g_{\omega_{\epsilon}}^{-1}]_{r}\|_{C^{2,\alpha}(\mathcal{B})}\|[g_{\omega_{\epsilon}}-g_{\omega_{\eta,\epsilon}}]_{r}\|_{C^{2,\alpha}(\mathcal{B})}+1\Big)\|[g_{\omega_{\eta,\epsilon}}]_{r}\|_{C^{2,\alpha}(\mathcal{B})} \\
            & \overset{\eqref{Compare5}}{\leq} C_{3}r\sum_{i=0}^{2}\sum_{|I|=i}(1+\frac{\epsilon^{2k-2}}{d^{2k-2}}(\frac{r}{d})^{|I|})+C_{4}.
         \end{split}
     \end{equation}
     Taking $\epsilon_{0}$ sufficiently small, we may assume that 
     \begin{equation}
         C_{3}r\sum_{i=0}^{2}\sum_{|I|=i}(1+\frac{\epsilon^{2k-2}}{d^{2k-2}}(\frac{r}{d})^{|I|})\leq \ti{C}_{3}r_{\epsilon}\leq \frac{1}{2}.
     \end{equation}
     Then, we have for any $\epsilon\leq r\leq r_{0}$,
     \begin{equation}
         \|g_{\omega_{\epsilon}}^{-1}(rZ_{p},rW_{p}) \|_{C^{2,\alpha}(B_{2}^{k}\setminus B_{1}^{k}\times B_{2}^{m-k})}\leq 2C_{4}.
     \end{equation}
     Since other cases can be estimated similarly, we may conclude that $\|g_{\omega_{\epsilon}}^{-1}\|_{C^{2,\alpha}_{0}(Bl_{X}M)}\leq C$.

     In the same coordinate system $(z,w)$ , the curvature can be written as
     \begin{equation}
         R_{\omega_{\epsilon},i\ol{j}k\ol{l}}=
         -\di_{k}\di_{\ol{l}}g_{\omega_{\epsilon},i\ol{j}}
         +g_{\omega_{\epsilon}}^{p\ol{q}}(\di_{k}g_{\omega_{\epsilon},i\ol{q}})(\di_{\ol{l}}g_{\omega_{\epsilon},p\ol{j}}).
     \end{equation}
    We estimate these terms respectively. Fix $\epsilon\leq r\leq r_{0}$,
    \begin{align}
        &\|-[\di_{k}\di_{\ol{l}}g_{\omega_{\epsilon},i\ol{j}}]_{r}\|_{C^{0,\alpha}(\mathcal{B})}
        \leq \|g_{\omega_{\epsilon}}\|_{C^{2,\alpha}_{0}(Bl_{X}M)}r^{-2},
        \\
        &\|[\di_{k}g_{\omega_{\epsilon},i\ol{q}}]_{r}\|_{C^{0,\alpha}(\mathcal{B})}
        \leq \|g_{\omega_{\epsilon}}\|_{C^{2,\alpha}_{0}(Bl_{X}M)}r^{-1},
        \\
        &\|[g_{\omega_{\epsilon}}^{p\ol{q}}]_{r}\|
        \leq \|g_{\omega_{\epsilon}}\|_{C^{2,\alpha}_{0}}.
    \end{align}
    Hence, we have
    \begin{equation}
        r^{-2}\|[R_{\omega_{\epsilon}}]_{r}\|_{C^{0,\alpha}(\mathcal{B})}\leq C_{5}\|g_{\omega_{\epsilon}}\|_{C^{2,\alpha}_{0}(Bl_{X}M)}\leq C_{6}
    \end{equation}
    The estimates in $\ti{B}_{(j),\epsilon}^{2}\times B_{\epsilon}^{m-2}$ follows completely similarly as above. When $d\geq r_{0}$ we already have $R_{\omega_{\epsilon}}=R_{\omega}$, the estimates follows obviously. Estimates on $Ric_{\omega_{\epsilon}}$ and $S(\omega_{\epsilon})$ are valid due to the uniform upper bound of $\|g_{\omega_{\epsilon}}^{-1}\|_{C^{2,\alpha}_{0}(Bl_{X}M)}$.
\end{proof}

\subsection{Uniform comparison estimates for $\omega_{\epsilon,\vphi}$ and $\omega_{\epsilon}$}
\label{UEC}
\begin{proof}[Proof of Proposition \ref{UniEst}]
      By Definition \ref{WeightedBLxM3}, we have  
      \begin{equation}
          \|g_{\omega_{\vphi,\epsilon},i\ol{j}}-g_{\omega_{\epsilon},i\ol{j}}\|_{C^{2,\alpha}(M\setminus\{d<r_{0} \})}=\|\di_{i}\di_{\ol{j}}\vphi\|_{C^{2,\alpha}(M\setminus\{d<r_{0} \})}\leq\|\vphi\|_{C^{4,\alpha}_{2}(Bl_{X}M)}.
      \end{equation}
      
      Fix $p\in X$ and a coordinates chart $\big( B_{4r_{0}}^{k}\times B_{4r_{0}}^{m-k},(z,w) \big)$ of Lemma \ref{holosubcoor} type.  For any $\epsilon \leq r \leq r_{0}$, we follow the notations in \eqref{r_change_nota} and estimate
      \begin{equation}
              \|[g_{\omega_{\vphi,\epsilon},i\ol{j}}-g_{\omega_{\epsilon},i\ol{j}}]_{r}\|_{C^{2,\alpha}(\mathcal{B})}
              =\|[\di_{i}\di_{\ol{j}}\vphi]_{r}\|_{C^{2,\alpha}(\mathcal{B})}
              \leq \|\vphi\|_{C^{4,\alpha}_{2}(Bl_{X}M)}.
      \end{equation}
      The estimates in $\ti{B}_{(j),\epsilon}^{2}\times B_{\epsilon}^{m-2}$ are exactly the same as above. Hence, we have proved the estimate on $\|g_{\omega_{\vphi,\epsilon}}-g_{\omega_{\epsilon}}\|_{C^{2,\alpha}_{\delta-2}(Bl_{X}M)}$. In the meantime, we have
      \begin{equation}
          \begin{split}
          \|g_{\omega_{\vphi,\epsilon}}\|_{C^{2,\alpha}_{0}(Bl_{X}M)}
          \leq& \|g_{\omega_{\vphi,\epsilon}}-g_{\omega_{\epsilon}}\|_{C^{2,\alpha}_{0}(Bl_{X}M)}+\|g_{\omega_{\epsilon}}\|_{C^{2,\alpha}_{0}(Bl_{X}M)}\\
          \leq & \|\vphi\|_{C^{4,\alpha}_{2}(Bl_{X}M)}+C_{1}\leq c+C_{1}
          \end{split}
      \end{equation}
     For the inverse metric, by $g_{\omega_{\vphi,\epsilon}}^{-1}-g^{-1}_{\omega_{\epsilon}}=g_{\omega_{\vphi,\epsilon}}^{-1}(g_{\omega_{\epsilon}}-g_{\omega_{\vphi,\epsilon}})g_{\omega_{\epsilon}}^{-1}$ and the uniform upper bound in Proposition \ref{metric_Uni}, we have
     \begin{equation}
         \begin{split}
          \|g_{\omega_{\vphi,\epsilon}}^{-1}\|_{C^{2,\alpha}_{0}(Bl_{X}M)}
          \leq& \|g_{\omega_{\vphi,\epsilon}}^{-1}-g_{\omega_{\epsilon}}^{-1}\|_{C^{2,\alpha}_{0}(Bl_{X}M)}+\|g_{\omega_{\epsilon}}^{-1}\|_{C^{2,\alpha}_{0}(Bl_{X}M)}\\
          \leq& C_{2} \|g_{\omega_{\vphi,\epsilon}}-g_{\omega_{\epsilon}}\|_{C^{2,\alpha}_{0}(Bl_{X}M)}\|g_{\omega_{\vphi,\epsilon}^{-1}}\|_{C^{2,\alpha}_{0}(Bl_{X}M)}+C_{3}\\
          \leq& C_{2}c\|g_{\omega_{\vphi,\epsilon}}^{-1}\|_{C^{2,\alpha}_{0}(Bl_{X}M)}+C_{3}.
         \end{split}
     \end{equation}
     Hence, when $c$ is taken sufficiently small such that $C_{2}c\leq\frac{1}{2}$, we have $\|g_{\omega_{\vphi,\epsilon}}^{-1}\|_{C^{2,\alpha}_{0}(Bl_{X}M)}\leq 2C_{3}$. It follows again by $g_{\omega_{\vphi,\epsilon}}^{-1}-g^{-1}_{\omega_{\epsilon}}=g_{\omega_{\vphi,\epsilon}}^{-1}(g_{\omega_{\epsilon}}-g_{\omega_{\vphi,\epsilon}})g_{\omega_{\epsilon}}^{-1}$, we have
     \begin{equation}\label{C_inverse}
     \begin{split}
         \|g_{\omega_{\vphi,\epsilon}}^{-1}-g_{\omega_{\epsilon}}^{-1}\|_{C^{2,\alpha}_{0}(Bl_{X}M)}
         \leq &C_{4}\|g_{\omega_{\vphi,\epsilon}}^{-1}\|_{C^{2,\alpha}_{0}(Bl_{X}M)}\|g_{\omega_{\epsilon}}^{-1}\|_{C^{2,\alpha}_{0}(Bl_{X}M)}\|g_{\omega_{\vphi,\epsilon}}-g_{\omega_{\epsilon}}\|_{C^{2,\alpha}_{0}(Bl_{X}M)}\\
         \leq& C_{5}\|\vphi \|_{C^{4,\alpha}_{2}(Bl_{X}M)}.
    \end{split} 
    \end{equation}
     
    Regarding the curvature tensor and the linearized operator, we estimate in the case $k=2$ as follows.
    \begin{equation}
        \begin{split}
           R_{\omega_{\vphi,\epsilon},i\ol{j}k\ol{l}}-R_{\omega_{\epsilon},i\ol{j}k\ol{l}}
           =&-\di_{k}\di_{\ol{l}}(g_{\omega_{\vphi,\epsilon},i\ol{j}}-g_{\omega_{\epsilon},i\ol{j}})
           +g_{\omega_{\vphi,\epsilon}}^{p\ol{q}}(\di_{k}g_{\omega_{\vphi,\epsilon},i\ol{q}})\di_{\ol{l}}(g_{\omega_{\vphi,\epsilon},p\ol{j}}-g_{\omega_{\epsilon},p\ol{j}})\\
           &+g_{\omega_{\vphi,\epsilon}}^{p\ol{q}}\di_{k}(g_{\omega_{\vphi,\epsilon},i\ol{q}}-g_{\omega_{\epsilon},i\ol{q}})(\di_{\ol{l}}g_{\omega_{\epsilon},p\ol{j}})
           +(g_{\omega_{\vphi,\epsilon}}^{p\ol{q}}-g_{\omega_{\epsilon}}^{p\ol{q}})(\di_{k}g_{\omega_{\epsilon},i\ol{q}})(\di_{\ol{l}}g_{\omega_{\epsilon},p\ol{j}}).
        \end{split}
    \end{equation}
    For any $\epsilon\leq r\leq r_{0}$, we have
    \begin{equation}\label{R1}
        \begin{split}
            [\ti{\tau}(r)]^{-\delta}r^{4}
            \|[\di_{k}\di_{\ol{l}}(g_{\omega_{\vphi,\epsilon},i\ol{j}}-g_{\omega_{\epsilon},i\ol{j}})]_{r}\|_{C^{0,\alpha}(\mathcal{B})}
            = [\ti{\tau}(r)]^{-\delta}r^{4}
            \|[\di_{k}\di_{\ol{l}}\di_{i}\di_{\ol{j}}\vphi]_{r}\|_{C^{0,\alpha}(\mathcal{B})}
            \leq \|\vphi\|_{\ti{C}^{4,\alpha}_{0,\delta}}.
        \end{split}
    \end{equation}
    For those two terms that involve first order derivatives the estimates are similar, we show one of the terms in the following.
    \begin{equation}
        \begin{split}
           &[\ti{\tau}(r)]^{-\delta}r^{4}
           \|[g_{\omega_{\vphi,\epsilon}}^{p\ol{q}}(\di_{k}g_{\omega_{\vphi,\epsilon},i\ol{q}})\di_{\ol{l}}(g_{\omega_{\vphi,\epsilon},p\ol{j}}-g_{\omega_{\epsilon},p\ol{j}})]_{r}\|_{C^{0,\alpha}(\mathcal{B})}
           \\
           \leq& \|[g_{\omega_{\vphi,\epsilon}}^{-1}]_{r}\|_{C^{0,\alpha}(\mathcal{B})}r\|(\di_{k}g_{\omega_{\vphi,\epsilon},i\ol{q}})_{r}\|_{C^{0,\alpha}(\mathcal{B})}[\ti{\tau}(r)]^{-\delta}r^{3}\|[\di_{\ol{l}}\di_{p}\di_{\ol{j}}\vphi]_{r}\|_{C^{0,\alpha}(\mathcal{B})}
           \\
           \leq& \|g_{\omega_{\vphi,\epsilon}}^{-1}\|_{C^{2,\alpha}_{0}(Bl_{X}M)}\|g_{\omega_{\vphi,\epsilon}}\|_{C^{2,\alpha}_{0}(Bl_{X}M)}\|\vphi\|_{\ti{C}^{4,\alpha}_{0,\delta}(Bl_{X}M)}
           \\
           \leq& C_{6}\|\vphi\|_{\ti{C}^{4,\alpha}_{0,\delta}(Bl_{X}M)}.
        \end{split}
    \end{equation}
    For the last term, in the same spirit of \eqref{C_inverse}, we get the following.
    \begin{equation}\label{R3}
        \begin{split}
            &[\ti{\tau}(r)]^{-\delta}r^{4}
            \|[(g_{\omega_{\vphi,\epsilon}}^{p\ol{q}}-g_{\omega_{\epsilon}}^{p\ol{q}})(\di_{k}g_{\omega_{\epsilon},i\ol{q}})(\di_{\ol{l}}g_{\omega_{\epsilon},p\ol{j}})]_{r}\|_{C^{0,\alpha}(\mathcal{B})}
            \\
            \leq& [\ti{\tau}(r)]^{-\delta}r^{2}\|[g_{\omega_{\vphi,\epsilon}}^{p\ol{q}}-g_{\omega_{\epsilon}}^{p\ol{q}}]_{r}\|_{C^{0,\alpha}(\mathcal{B})}
            r\|[\di_{k}g_{\omega_{\epsilon},i\ol{q}}]_{r}\|_{C^{0,\alpha}(\mathcal{B})}
            r\|[\di_{\ol{l}}g_{\omega_{\epsilon},p\ol{j}}]_{r}\|_{C^{0,\alpha}(\mathcal{B})}
            \\
            \leq& 
            C_{7}[\ti{\tau}(r)]^{-\delta}r^{2}\sup_{|I|=2}\|[\di^{I}\vphi]_{r}\|_{C^{0,\alpha}(\mathcal{B})}(\|g_{\omega_{\epsilon}}\|_{C^{2,\alpha}_{0}(Bl_{X}M)})^{2}
            \\
            \leq& C_{8}\|\vphi\|_{\ti{C}^{4,\alpha}_{0,\delta}(Bl_{X}M)}.
        \end{split}
    \end{equation}
    Applying the similar method as above in $\ti{B}_{(j),\epsilon}^{2}\times B_{\epsilon}^{2}$ and $\{d\geq r_{0}\}$, we may derive
    \begin{equation}
        \|R_{\omega_{\vphi,\epsilon}}-R_{\omega_{\epsilon}}\|_{\ti{C}^{0,\alpha}_{-4,\delta}(Bl_{X}M)}\leq C\|\vphi\|_{\ti{C}^{4,\alpha}_{0,\delta}(Bl_{X}M)}.
    \end{equation}
    Replacing the weight function $\ti{\tau}$ by $\tau$, we have the similar estimate when $k\geq 3$.

    Regarding the linearized operator, we write down the local expression as follows.
    \begin{equation}
        \begin{split}
            (L_{\omega_{\vphi,\epsilon}}-L_{\omega_{\epsilon}})f
            =& -(\Delta^{2}_{\omega_{\vphi,\epsilon}}-\Delta^{2}_{\omega_{\epsilon}})f
            \\
            &-\Big( Ric_{\omega_{\vphi,\epsilon},i\ol{j}}g_{\omega_{\vphi,\epsilon}}^{i\ol{l}}g_{\omega_{\vphi,\epsilon}}^{k\ol{j}}-Ric_{\omega_{\epsilon},i\ol{j}}g_{\omega_{\epsilon}}^{i\ol{l}}g_{\omega_{\epsilon}}^{k\ol{j}} \Big)\di_{k}\di_{\ol{l}}f.
        \end{split}
    \end{equation}
    Note that
    \begin{equation}\label{biLap_4terms}
        (\Delta^{2}_{\omega_{\vphi,\epsilon}}-\Delta^{2}_{\omega_{\epsilon}})f= (g_{\omega_{\vphi,\epsilon}}^{i\ol{j}}-g_{\omega_{\epsilon}}^{i\ol{j}})g_{\omega_{\epsilon}}^{k\ol{l}}\di_{i}\di_{\ol{j}}\di_{k}\di_{\ol{j}}f+g_{\omega_{\epsilon}}^{i\ol{j}}(\di_{i}\di_{\ol{j}}(g_{\omega_{\vphi,\epsilon}}^{k\ol{l}}-g_{\omega_{\epsilon}}^{k\ol{l}}))(\di_{k}\di_{\ol{l}}f)+\text{two extra terms}. 
    \end{equation}
    We estimate the first two terms on the right-hand side of \eqref{biLap_4terms}, as two extra terms can be estimated similarly. We consider the region $\epsilon\leq r\leq r_{0}$, firstly we have the following.
    \begin{equation}\label{L-L_1}
        \begin{split}
            &[\ti{\tau}(r)]^{-\delta}r^{4}\|[(g_{\omega_{\vphi,\epsilon}}^{i\ol{j}}-g_{\omega_{\epsilon}}^{i\ol{j}})g_{\omega_{\epsilon}}^{k\ol{l}}\di_{i}\di_{\ol{j}}\di_{k}\di_{\ol{j}}f]_{r}\|_{C^{0,\alpha}(\mathcal{B})}
            \\
            \leq& 
            \|[g_{\omega_{\vphi,\epsilon}}^{i\ol{j}}-g_{\omega_{\epsilon}}^{i\ol{j}}]_{r}\|_{C^{0,\alpha}(\mathcal{B})}
            \|[g_{\omega_{\epsilon}}^{k\ol{l}}]_{r}\|_{C^{0,\alpha}(\mathcal{B})}
            [\ti{\tau}(r)]^{-\delta}r^{4}\|[\di_{i}\di_{\ol{j}}\di_{k}\di_{\ol{j}}f]_{r}\|_{C^{0,\alpha}(\mathcal{B})}
            \\
            \leq& 
            C_{9}\sup_{|I|=2}\|[\di^{I}\vphi]_{r}\|_{C^{0,\alpha}(\mathcal{B})}
            \|g_{\omega_{\epsilon}}^{-1}\|_{C^{2,\alpha}_{0}(Bl_{X}M)}
            [\ti{\tau}(r)]^{-\delta}\|[f]_{r}\|_{C^{4,\alpha}(\mathcal{B})}
            \\
            \leq& 
            C_{9}r^{-2}\|[\vphi]_{r}\|_{C^{4,\alpha}(\mathcal{B})}
            \|g_{\omega_{\epsilon}}^{-1}\|_{C^{2,\alpha}_{0}(Bl_{X}M)}
            [\ti{\tau}(r)]^{-\delta}\|[f]_{r}\|_{C^{4,\alpha}(\mathcal{B})}
            \\
            \leq& C_{10}\|\vphi\|_{C^{4,\alpha}_{2}(Bl_{X}M)}\|f\|_{\ti{C}^{4,\alpha}_{0,\delta}(Bl_{X}M)}.
        \end{split}
    \end{equation}
    Secondly, we have
    \begin{equation}
        \begin{split}
            &[\ti{\tau}(r)]^{-\delta}r^{4}
            \|[g_{\omega_{\epsilon}}^{i\ol{j}}(\di_{i}\di_{\ol{j}}(g_{\omega_{\vphi,\epsilon}}^{k\ol{l}}-g_{\omega_{\epsilon}}^{k\ol{l}}))(\di_{k}\di_{\ol{l}}f)]_{r}\|_{C^{0,\alpha}(\mathcal{B})}
            \\
            \leq& 
            \|g_{\omega_{\epsilon}}^{-1}\|_{C^{2,\alpha}_{0}(Bl_{X}M)}
            r^{2}\sup_{|I|=2}\|[\di_{i}\di_{\ol{j}}\di^{I}\vphi]_{r}\|_{C^{0,\alpha}(\mathcal{B})}[\ti{\tau}(r)]^{-\delta}r^{2}\|[\di_{k}\di_{\ol{l}}f]_{r}\|_{C^{0,\alpha}(\mathcal{B})}
            \\
            \leq&
            C_{11}
            \|g_{\omega_{\epsilon}}^{-1}\|_{C^{2,\alpha}_{0}(Bl_{X}M)}
            r^{-2}\|[\vphi]_{r}\|_{C^{4,\alpha}(\mathcal{B})}
            [\ti{\tau}(r)]^{-\delta}\|[f]_{r}\|_{C^{4,\alpha}(\mathcal{B})}
            \\
            \leq& 
            C_{12}\|\vphi\|_{C^{4,\alpha}_{2}(Bl_{X}M)}\|f\|_{\ti{C}^{4,\alpha}_{0,\delta}(Bl_{X}M)}.
        \end{split}
    \end{equation}
    For the Ricci term, notice that
    \begin{equation}\label{Ricci_diff_term}
        \begin{split}
            &\Big( Ric_{\omega_{\vphi,\epsilon},i\ol{j}}g_{\omega_{\vphi,\epsilon}}^{i\ol{l}}g_{\omega_{\vphi,\epsilon}}^{k\ol{j}}-Ric_{\omega_{\epsilon},i\ol{j}}g_{\omega_{\epsilon}}^{i\ol{l}}g_{\omega_{\epsilon}}^{k\ol{j}} \Big)\di_{k}\di_{\ol{l}}f
            \\
            =&\Big( Ric_{\omega_{\vphi,\epsilon},i\ol{j}}-Ric_{\omega_{\epsilon},i\ol{j}} \Big)g_{\omega_{\vphi,\epsilon}}^{i\ol{l}}g_{\omega_{\vphi,\epsilon}}^{k\ol{j}}\di_{k}\di_{\ol{l}}f
            +Ric_{\omega_{\epsilon},i\ol{j}}\Big( (g_{\omega_{\vphi,\epsilon}}^{i\ol{l}}-g_{\omega_{\epsilon}}^{i\ol{l}})g_{\omega_{\vphi,\epsilon}}^{k\ol{j}}
            +g_{\omega_{\vphi,\epsilon}}^{i\ol{l}}(g_{\omega_{\vphi,\epsilon}}^{k\ol{j}}-g_{\omega_{\vphi,\epsilon}}^{k\ol{j}}) \Big)\di_{k}\di_{\ol{l}}f.
        \end{split}
    \end{equation}
    For the first term on the right-hand side of \eqref{Ricci_diff_term}, we have
    \begin{equation}
        \begin{split}
            &[\ti{\tau}(r)]^{-\delta}r^{4}
            \|[\Big( Ric_{\omega_{\vphi,\epsilon},i\ol{j}}-Ric_{\omega_{\epsilon},i\ol{j}} \Big)g_{\omega_{\vphi,\epsilon}}^{i\ol{l}}g_{\omega_{\vphi,\epsilon}}^{k\ol{j}}\di_{k}\di_{\ol{l}}f]_{r}\|_{C^{0,\alpha}(\mathcal{B})}
            \\
            \leq& C_{13}
            r^{-2}\|[\vphi]_{r}\|_{C^{4,\alpha}(\mathcal{B})}\|g_{\omega_{\vphi,\epsilon}}^{-1}\|_{C^{2,\alpha}_{0}(Bl_{X}M)}
            \|g_{\omega_{\epsilon}}^{-1}\|_{C^{2,\alpha}_{0}(Bl_{X}M)}
            [\ti{\tau}(r)]^{-\delta}r^{2}\|[\di_{k}\di_{\ol{l}}f]_{r}\|_{C^{0,\alpha}(\mathcal{B})}
            \\
            \leq&
            C_{14}\|\vphi\|_{C^{4,\alpha}_{2}(Bl_{X}M)}\|f\|_{\ti{C}^{4,\alpha}_{0,\delta}(Bl_{X}M)}.
        \end{split}
    \end{equation}
    Here, we replace the quantity $[\ti{\tau}(r)]^{-\delta}r^{4}$ by $r^{2}$ in \eqref{R1}--\eqref{R3} to obtain
    \begin{equation}
        r^{2}
            \|[\Big( Ric_{\omega_{\vphi,\epsilon},i\ol{j}}-Ric_{\omega_{\epsilon},i\ol{j}} \Big)]_{r}\|_{C^{0,\alpha}(\mathcal{B})}
            \leq C 
            r^{-2}\|[\vphi]_{r}\|_{C^{4,\alpha}(\mathcal{B})}.
    \end{equation}
    For the second term on the right-hand side of \eqref{Ricci_diff_term}, by \eqref{C_inverse}, we have
    \begin{equation}\label{L-L_2}
        \begin{split}
            &[\ti{\tau}(r)]^{-\delta}r^{4}
            \|[Ric_{\omega_{\epsilon},i\ol{j}}\Big( (g_{\omega_{\vphi,\epsilon}}^{i\ol{l}}-g_{\omega_{\epsilon}}^{i\ol{l}})g_{\omega_{\vphi,\epsilon}}^{k\ol{j}} \Big)\di_{k}\di_{\ol{l}}f]_{r}\|_{C^{0,\alpha}(\mathcal{B})}
            \\
            \leq& 
            C_{15}r_{0}^{2}\|g_{\omega_{\epsilon}}\|_{C^{2,\alpha}_{0}(Bl_{X}M)}
            \|(g_{\omega_{\vphi,\epsilon}}^{-1}-g_{\omega_{\epsilon}}^{-1})\|_{C^{2,\alpha}_{0}(Bl_{X}M)}
            \|g_{\omega_{\vphi,\epsilon}}^{-1}\|_{C^{2,\alpha}_{0}(Bl_{X}M)}[\ti{\tau}(r)]^{-\delta}r^{2}
            \|[\di_{k}\di_{\ol{l}}f]_{r}\|_{C^{0,\alpha}(\mathcal{B})}
            \\
            \leq& 
            C_{16}\|\vphi\|_{C^{4,\alpha}_{2}(Bl_{X}M)}\|f\|_{\ti{C}^{4,\alpha}_{0,\delta}(Bl_{X}M)}.
        \end{split}
    \end{equation}
    Hence, we have 
    \begin{equation}
        [\ti{\tau}(r)]^{-\delta}r^{4}\|[(L_{\omega_{\vphi,\epsilon}}-L_{\omega_{\epsilon}})f]_{r}\|_{C^{0,\alpha}(\mathcal{B})}\leq C\|\vphi\|_{C^{4,\alpha}_{2}(Bl_{X}M)}\|f\|_{\ti{C}^{4,\alpha}_{0,\delta}(Bl_{X}M)}.
    \end{equation}
    The corresponding estimates in the regions $\ti{B}_{(j),\epsilon}^{2}\times B_{\epsilon}^{m-2}$ and $\{d\geq r_{0}\}$ can be done similarly. By replacing the weight function $\ti{\tau}$ to $\tau$, we may derive estimates when $k\geq 3$.
    \end{proof}

\section{Estimates on linearized operators}

\subsection{Proof of Schauder estimates}\label{PS}
\begin{proof}[Proof of Proposition \ref{Schauder}]
   We consider the case of codimension $k=2$.

   We introduce the following notations.
    \begin{align}
        &[h]_{(j),\epsilon}(Z_{j},W)=h(\epsilon Z_{(j)},\epsilon W),\ \ \ \ti{\mathcal{B}}_{(j),t}=B_{(j),t}^{2}\times B_{t}^{m-2},\\
        &[h]_{r}(Z,W)=h(rZ,rW),\ \ \ \ \ 
        \mathcal{B}_{t}=B_{2t}^{2}\setminus B_{\frac{1}{t}}^{2}\times B_{2t}^{m-2}.
    \end{align}
    for any function $h$ defined in the corresponding domain.

    Fix $p\in X$ and a coordinate neighborhood $\big( B_{2r_{0}}^{2}\times B_{2r_{0}}^{m-2},(z,w) \big)$ of Lemma \ref{holosubcoor} type centered at $p$. For any $j=1,2$, we define
    \begin{equation}
    \begin{split}
        f_{(j),\epsilon}(Z_{(j)},W)
        =[\ti{\tau}(\epsilon)]^{-\delta}f|_{\ti{B}_{(j),2\epsilon}^{2}\times B_{2\epsilon}^{m-2}}(\epsilon Z_{(j)},\epsilon W),
        \ \ (Z_{(j)},W)\in \ti{\mathcal{B}}_{(j),2}:=\ti{B}_{(j),2}^{2}\times B_{2}^{m-2}.
    \end{split}
    \end{equation}

    Then we apply the usual Schauder interior estimates to each $f_{(j),\epsilon}$ with respect to $L_{\omega_{\eta}}$.
    \begin{equation}\label{Schauder1}
        \|f_{(j),\epsilon}\|_{C^{4,\alpha}(\ti{\mathcal{B}}_{(j),1})}\leq \ti{C}\Big( \|f_{(j),\epsilon}\|_{C^{0}(\ti{\mathcal{B}}_{(j),2})}+\|L_{\omega_{\eta}}f_{(j),\epsilon}\|_{C^{0,\alpha}(\ti{\mathcal{B}}_{(j),2})} \Big),
    \end{equation}
    where, the constant $\ti{C}$ depends on the constants of ellipticity and the $C^{4,\alpha}$ norms of coefficients of $L_{\omega_{\eta}}$, which are independent of $\epsilon$. Hence, $\ti{C}$ does not depends on $\epsilon$ and the choice of $p\in X$.

    Note that
    \begin{equation}
        L_{\omega_{\eta}}f_{(j),\epsilon}(Z_{(j)},W)=[\ti{\tau}(\epsilon)]^{-\delta}\epsilon^{4}(L_{\omega_{\eta,\epsilon}}f)(\epsilon Z_{(j)},\epsilon W).
    \end{equation}
    Then, we have 
    \begin{equation}
        \begin{split}
            \|L_{\omega_{\eta}}f_{(j),\epsilon}\|_{C^{0,\alpha}(\ti{\mathcal{B}}_{(j),2})}
            =&[\ti{\tau}(\epsilon)]^{-\delta}\epsilon^{4}\|[L_{\omega_{\eta,\epsilon}}f]_{(j),\epsilon}\|_{C^{0,\alpha}(\ti{\mathcal{B}}_{(j),2})}
            \\
            \leq& 
            [\ti{\tau}(\epsilon)]^{-\delta}\epsilon^{4}\|[L_{\omega_{\epsilon}}f]_{(j),\epsilon}\|_{C^{0,\alpha}(\ti{\mathcal{B}}_{(j),2})}
            \\
            &+
            [\ti{\tau}(\epsilon)]^{-\delta}\epsilon^{4}\|[(L_{\omega_{\eta,\epsilon}}-L_{\omega_{\epsilon}})f]_{(j),\epsilon}\|_{C^{0,\alpha}(\ti{\mathcal{B}}_{(j),2})}.
        \end{split}
    \end{equation}
    According to \eqref{L-L_1}--\eqref{L-L_2} and the comparison results of the background metric \eqref{BSCompare}, we have
    \begin{equation}
        \begin{split}
            &[\ti{\tau}(\epsilon)]^{-\delta}\epsilon^{4}\|[(L_{\omega_{\eta,\epsilon}}-L_{\omega_{\epsilon}})f]_{(j),\epsilon}\|_{C^{0,\alpha}(\ti{\mathcal{B}}_{(j),2})}
            \\
            \leq&
            C_{0}\|[g_{\omega_{\eta,\epsilon}}-g_{\omega_{\epsilon}}]_{(j),\epsilon}\|_{C^{2,\alpha}(\ti{\mathcal{B}}_{(j),2})}
            \|f_{(j),\epsilon}\|_{C^{4,\alpha}(\ti{\mathcal{B}}_{(j),2})}
            \\
            \leq&
            C_{1}\|(|\epsilon Z_{(j)}^{j}|+|\epsilon W|)+\epsilon^{2}\|_{C^{2,\alpha}(\ti{\mathcal{B}}_{(j),2})}
            \|f_{(j),\epsilon}\|_{C^{4,\alpha}(\ti{\mathcal{B}}_{(j),2})}
            \\
            \leq& 
            \frac{1}{2\ti{C}}  \|f\|_{C^{4,\alpha}_{0,\delta}(Bl_{X}M)},
        \end{split}
    \end{equation}
    when $\epsilon$ is taken sufficiently small. Substitute this estimate in (\ref{Schauder1}), we have
    \begin{equation}\label{Schauder1.5}
        \begin{split}
            &[\ti{\tau}(\epsilon)]^{-\delta}\|f(\epsilon Z_{(j)},\epsilon W)\|_{C^{4,\alpha}(\ti{\mathcal{B}}_{(j),1})}
            \\
            \leq& \ti{C}\Big( [\ti{\tau}(\epsilon)]^{-\delta}\|f(\epsilon Z_{(j)},\epsilon W)\|_{C^{0}(\ti{\mathcal{B}}_{(j),2})}
            \\
            &+
            [\ti{\tau}(\epsilon)]^{-\delta}\epsilon^{4}\|L_{\omega_{\epsilon}}f(\epsilon Z_{(j)},\epsilon W)\|_{C^{0,\alpha}(\ti{\mathcal{B}}_{(j),2})}
            +
            \frac{1}{8\ti{C}}\|f\|_{C^{4,\alpha}_{0,\delta}(Bl_{X}M)} \Big)
            \\
            \leq&
            \ti{C}_{1}\Big( \|f\|_{\ti{C}^{0}_{\delta}(Bl_{X}M)}
            +\|L_{\omega_{\epsilon}}f\|_{\ti{C}^{0,\alpha}_{-4,\delta}(Bl_{X}M)} \Big)
            +\frac{1}{8}\|f\|_{C^{4,\alpha}_{0,\delta}(Bl_{X}M)}.
        \end{split}
    \end{equation}

    Similarly, we consider the region $\epsilon \leq r \leq r_{0}$ and define
    \begin{equation}
        f_{r}(Z,W)=[\ti{\tau}(r)]^{-\delta}f|_{B_{4r}^{k}\setminus B_{\frac{r}{2}}^{k}\times B_{4r}^{m-k}}(rZ,rW), \ (Z,W)\in B_{4}^{2}\setminus B_{\frac{1}{2}}^{2}\times B_{4}^{m-2}.
    \end{equation}
    Then, we apply the Schauder interior estimate to the Euclidean bi-Laplacian $\Delta_{\mathrm{euc}}^{2}$ and obtain
    \begin{equation}\label{Schauder2}
        \|f_{r}\|_{C^{4,\alpha}(\mathcal{B}_{1})}\leq \ti{C}_{2}\Big( \|f_{r}\|_{C^{0}(\mathcal{B}_{2})}+\|\Delta_{euc}^{2}f_{r}\|_{C^{0,\alpha}(\mathcal{B}_{2})} \Big).
    \end{equation}
    According to the comparison estimates (\ref{EucCompare}), we have
    \begin{equation}
        \begin{split}
            \|\Delta_{\mathrm{euc}}^{2}f_{r}\|_{C^{0,\alpha}(\mathcal{B}_{2})}
            =& 
            [\ti{\tau}(r)]^{-\delta}r^{4}\Big(\|[(\Delta_{euc}^{2}-L_{\omega_{\epsilon}})f]_{r}\|_{C^{0,\alpha}(\mathcal{B}_{2})}
            +\|[L_{\omega_{\epsilon}}f]_{r}\|_{C^{0,\alpha}(\mathcal{B}_{2})} \Big)
            \\
            \leq& 
            [\ti{\tau}(r)]^{-\delta}r^{4}\|[L_{\omega_{\epsilon}}f]_{r}\|_{C^{0,\alpha}(\mathcal{B}_{2})}
            +\|[g_{euc}-g_{\omega_{\epsilon}}]_{r}\|_{C^{2,\alpha}(\mathcal{B}_{2})}[\ti{\tau}(r)]^{-\delta}\|f(rZ,rW)\|_{C^{4,\alpha}(\mathcal{B}_{2})}
            \\
            \leq& 
            [\ti{\tau}(r)]^{-\delta}r^{4}\|[L_{\omega_{\epsilon}}f]_{r}\|_{C^{0,\alpha}(\mathcal{B}_{2})}
            \\
            &+C_{2}\|(|rZ|+|rW|)+\epsilon^{2-2\theta}\log\frac{1}{\epsilon}\|_{C^{2,\alpha}(\mathcal{B}_{2})}
            \|f\|_{\ti{C}^{4,\alpha}_{0,\delta}(Bl_{X}M)}
            \\
            \leq& [\ti{\tau}(r)]^{-\delta}r^{4}\|[L_{\omega_{\epsilon}}f]_{r}\|_{C^{0,\alpha}(\mathcal{B}_{2})}
            +C_{3}r_{0}\|f\|_{\ti{C}^{4,\alpha}_{0,\delta}(Bl_{X}M)}
         \end{split}
    \end{equation}
    when $\epsilon$ is taken sufficiently small. Therefore, when $r_{0}$ is also sufficiently small, we have
    
    \begin{equation}\label{Schauder2.5}
        \begin{split}
            [\ti{\tau}(r)]^{-\delta}\|f(rZ,rW)\|_{C^{4,\alpha}(\mathcal{B}_{1})}
            \leq& \ti{C}_{2}\Big( [\ti{\tau}(r)]^{-\delta}\|f(rZ,rW)\|_{C^{0}(\mathcal{B}_{2})}
            +[\ti{\tau}(r)]^{-\delta}r^{4}\|L_{\omega_{\epsilon}}f(rZ,rW)\|_{C^{0,\alpha}(\mathcal{B}_{1})}
            \\
            &+2C_{2}r_{0}\|f\|_{\ti{C}^{4,\alpha}_{0,\delta}(Bl_{X}M)} \Big)
            \\
            \leq& 
            \ti{C}_{3}\Big(
            \|f\|_{\ti{C}^{0}_{\delta}(Bl_{X}M)}
            +\|L_{\omega_{\epsilon}}f\|_{\ti{C}^{0,\alpha}_{0,\delta}(Bl_{X}M)}
            \Big)
            +\frac{1}{4} \|f\|_{\ti{C}^{4,\alpha}_{0,\delta}(Bl_{X}M)}.
        \end{split}
    \end{equation}
    
    Note that the constants $\ti{C}_{2},\ti{C}_{3},C_{2}$ and $C_{3}$ all independent of $r\in[\epsilon,r_{0}]$, $\epsilon$ and $p\in X$. Then, apply (\ref{Schauder2.5}) to every $r\in[\epsilon,r_{0}]$, add \eqref{Schauder1.5}, and let $p$ ranges over $X$, we have the following
    \begin{equation}\label{Schauder3}
    \begin{split}
        \sup_{p\in X}&\Big\{ \sup_{\epsilon<r<r_{0}}[\ti{\tau}(r)]^{-\delta}\|f(rZ_{p},rW_{p})\|_{C^{4,\alpha}(B_{2}^{k}\setminus B_{1}^{k}\times B_{2}^{m-k})}
        \\
        &+\sum_{j=1}^{2}[\ti{\tau}(\epsilon)]^{-\delta}\|f(\epsilon \ti{Z}_{(j),p},\epsilon\ti{W}_{p})\|_{C^{4,\alpha}(\ti{B}_{(j),1}^{k}\times B_{1}^{m-k})} \Big\}
        \\
       \leq& \ti{C}_{4}(\|f\|_{\ti{C}^{0}_{\delta}(Bl_{X}M)}+\|L_{\omega_{\epsilon}}f\|_{\ti{C}^{0,\alpha}_{-4,\delta}(Bl_{X}M)})+\frac{1}{2}\|f\|_{\ti{C}^{4,\alpha}_{0,\delta}(Bl_{X}M)}.
    \end{split}
    \end{equation}

    On the region $d\geq \frac{r_{0}}{2}$, as $\omega_{\epsilon}=\omega$ and the choice of $r_{0}$ is independent of $\epsilon$, we may apply the Schauder estimates directly.
    \begin{equation}\label{Schauder3.5}
    \begin{split}
        \|f\|_{C^{4,\alpha}(M\setminus\{d<r_{0}\})}
        \leq& \ti{C}_{5}\Big(\|f\|_{C^{0}(M\setminus\{d<\frac{r_{0}}{2}\})}+\|L_{\omega_{\epsilon}}f\|_{C^{0,\alpha}(M\setminus\{d<\frac{r_{0}}{2}\})}\Big)\\
        \leq& \ti{C}_{5}\Big(\|f\|_{C^{0}_{\delta}(Bl_{X}M)}+\|L_{\omega_{\epsilon}}f\|_{C^{0,\alpha}_{\delta-4}(Bl_{X}M)}\Big).
    \end{split}
    \end{equation}
    Take $\ti{C}_{6}>2\max\{\ti{C}_{4},\ti{C}_{5}\}$ , then summation of (\ref{Schauder3}) and (\ref{Schauder3.5}) implies
    \begin{equation}\label{Schauder4}
        \|f\|_{\ti{C}^{4,\alpha}_{0,\delta}(Bl_{X}M,\omega_{\epsilon})}\leq \ti{C}_{6}\Big( \|f\|_{\ti{C}^{0}_{\delta}(Bl_{X}M,\omega_{\epsilon})}+\|L_{\omega_{\epsilon}}f\|_{\ti{C}^{0,\alpha}_{-4,\delta}(Bl_{X}M,\omega_{\epsilon})} \Big).
    \end{equation}

    Finally, note that
    \begin{equation}
    \begin{split}
        \|f_{j}\|_{\ti{C}^{0,\alpha}_{-4,\delta}(Bl_{X}M)}
        \leq&
        [\ti{\tau}(\epsilon)]^{-\delta}\epsilon^{4}\|f_{j}\|_{C^{0,\alpha}(\{d\leq \frac{3}{2}\epsilon\})}
        +
        [\ti{\tau}(r_{0})]^{-\delta}r_{0}^{4}\|f_{j}\|_{C^{0,\alpha}(\{\frac{1}{2}\epsilon \leq d\leq 3r_{0}\})}
        +
        \|f_{j}\|_{C^{0,\alpha}(\{d\geq r_{0}\})}
        \\
        \leq& C_{4}\|f_{j}\|_{C^{0,\alpha}(Bl_{X}M)}.
    \end{split}
    \end{equation}
    Then,
    \begin{equation}
        \begin{split}
           \|L_{\omega_{\epsilon}}f\|_{\ti{C}^{0,\alpha}_{-4,\delta}(Bl_{X}M,\omega_{\epsilon})}
           \leq& \|\ti{L}_{\omega_{\epsilon}}f\|_{\ti{C}^{0,\alpha}_{-4,\delta}(Bl_{X}M,\omega_{\epsilon})}
           +\|\sum_{j=1}^{d}f(q_{j})f_{j}\|_{\ti{C}^{0,\alpha}_{-4,\delta}(Bl_{x}M,\omega_{\epsilon})}\\
           \leq& \|\ti{L}_{\omega_{\epsilon}}f\|_{\ti{C}^{0,\alpha}_{-4,\delta}(Bl_{X}M,\omega_{\epsilon})}+\sum_{j=1}^{d}C_{4}\|f_{j}\|_{C^{0,\alpha}(Bl_{X}M)}\|f\|_{\ti{C}^{0}_{\delta}(Bl_{X}M, \omega_{\epsilon})}.
        \end{split}
    \end{equation}
    Then, (\ref{Schauder4}) implies
    \begin{equation}
        \|f\|_{\ti{C}^{4,\alpha}_{0,\delta}(Bl_{X}M,\omega_{\epsilon})}\leq \ti{C}_{6}\Big(C_{5} \|f\|_{\ti{C}^{0}_{\delta}(Bl_{X}M,\omega_{\epsilon})}+\|\ti{L}_{\omega_{\epsilon}}f\|_{\ti{C}^{0,\alpha}_{-4,\delta}(Bl_{X}M,\omega_{\epsilon})} \Big).
    \end{equation}
    Then, we may conclude the desired result. Replacing the weight function from $\ti{\tau}$ to $\tau$, and use estimates in Lemma \ref{MC3}, the results when $k\geq 3$ follows directly.
\end{proof}

\subsection{Some auxiliary estimates in the proof of Proposition \ref{Inverse}}\label{EL}

In this subsection, we provide several detailed calculations used in the proof of Proposition \ref{Inverse}.

\medskip
\noindent\textbf{Proof of \eqref{S2E0}.}

    We show that there exists a constant $C$ independent of sufficiently large $i$ such that
    \begin{equation}
        \|\di_{Z_{(1)}}^{I}\di_{W}^{J}\Psi_{i}\|_{C^{0,\alpha}(B)}\leq C\epsilon_{i}^{|J|}\|\psi_{i}\|_{C^{4,\alpha}_{\delta}(Bl_{X}M)}.
    \end{equation}
    Now, we will use the following scaling notation
    \begin{align}
        &(z,w)=(\epsilon_{i}Z,\epsilon_{i}^{2}W)=(\epsilon_{i}r\ti{Z},\epsilon_{i}^{2}r\ti{W}),\\
        &(z_{(j)},w)=(\epsilon_{i}Z_{(j)},\epsilon_{i}^{2}W), \ j=1,...,k.
    \end{align}
    
    Fix positive constants $R_{1}$, $S_{1}$ and consider $\ti{B}_{R_{1}}^{k}\times B_{S_{1}}^{m-k}$. Note that when $i$ is large enough
    \begin{equation}
        \Psi_{i}=\epsilon_{i}^{-\delta}\Lambda_{i}^{*}\psi_{i}|_{U_{i}}, \ \text{on } \ti{B}_{R_{1}}^{k}\times B_{S_{1}}^{m-k}.
    \end{equation}
    For domains close to the exceptional divisor, without loss of generality, we first consider local coordinates $(Z_{(j)},W)\in\ti{B}_{(j),1}^{k}\times B_{S_{1}}^{m-k}$. Then, when $i$ is sufficiently large, we have the following $C^{0}$ estimate
    \begin{equation}
        \begin{split}
            \|\di_{Z_{(1)}}^{I}\di_{W}^{J}\Psi_{i} \|_{C^{0}(\ti{B}_{(j),1}^{k}\times B_{S_{1}}^{m-k})}
            =& \epsilon_{i}^{-\delta}\|\di_{Z_{(1)}}^{I}\di_{W}^{J}\psi_{i}(\epsilon_{i}Z_{(j)},\epsilon_{i}^{2}W) \|_{C^{0}(\ti{B}_{(1),1}^{k}\times B_{S_{1}}^{m-k})} \\
            \leq& \epsilon_{i}^{|I|+2|J|-\delta}\|\di_{z_{(1)}}^{I}\di_{w}^{J}\psi_{i} \|_{C^{0}(\ti{B}_{(1),\epsilon_{i}}^{k}\times B_{S_{1}\epsilon_{i}^{2}}^{m-k})}\\
            \leq & \epsilon_{i}^{|J|}\|\psi_{i} \|_{C^{4,\alpha}_{\delta}(Bl_{X}M)}, \ \forall |I|+|J|\leq 4.
        \end{split}
    \end{equation}
    For H\"{o}lder continuity, 
    \begin{equation}\label{Ho_psi_e}
        \begin{split}
            \frac{|\Psi_{i}(Z_{(j),1},W_{1})-\Psi_{i}(Z_{(j),2},W_{2})|}{|(Z_{(j),1},W_{1})-(Z_{(j),2},W_{2})|^{\alpha}}
            =&\epsilon_{i}^{-\delta}\frac{|\psi_{i}(\epsilon_{i}Z_{(j),1},\epsilon_{i}^{2}W_{1})-\psi_{i}(\epsilon_{i}Z_{(j),2},\epsilon_{i}^{2}W_{2})|}{|(Z_{(j),1},W_{1})-(Z_{(j),2},W_{2})|^{\alpha}} \\
            \leq& \|\psi_{i}\|_{C^{4,\alpha}_{\delta}(Bl_{X}M)}\frac{|(Z_{(j),1},\epsilon_{i}W_{1})-(Z_{(j),2},\epsilon_{i}W_{2})|^{\alpha}}{|(Z_{(j),1},W_{1})-(Z_{(j),2},W_{2})|^{\alpha}}\\
            \leq & \|\psi_{i}\|_{C^{4,\alpha}_{\delta}(Bl_{X}M)}(1+\epsilon_{i}^{2})^{\frac{\alpha}{2}}.
        \end{split}
    \end{equation}
    Then, we have
    \begin{equation}\label{H_continue}
        \begin{split}
            \|\Psi_{i}\|_{C^{\alpha}(\ti{B}_{(j),1}^{k}\times B_{S_{1}}^{m-k})}
            \leq (1+\epsilon_{i}^{2})^{\frac{\alpha}{2}}\|\psi_{i} \|_{C^{4,\alpha}_{\delta}(Bl_{X}M)}.
        \end{split}
    \end{equation}
    The same estimates hold for any $j=1,..k$. For higher order derivatives, the estimates are basically the same. By induction, we may derive the following
    \begin{equation}\label{ScaleEstimate0}
        \|\di_{Z_{(1)}}^{I}\di_{W}^{J}\Psi_{i}\|_{C^{0,\alpha}(\ti{B}_{1}^{k}\times B_{S_{1}}^{m-k})}\leq 2\epsilon_{i}^{|J|}\|\psi_{i}\|_{C^{4,\alpha}_{\delta}(Bl_{X}M)}.
    \end{equation}

    For domains away from the exceptional divisor, let $(Z,W)\in B_{2r}^{k}\setminus B_{r}^{k}\times B_{S_{1}}^{m-k}$. Note that when $r$ is fixed, for any sufficiently large $i$ we still have
    \begin{equation}
       \Psi_{i}(Z,W)=\epsilon_{i}^{-\delta}\psi_{i}(\epsilon_{i}Z,\epsilon_{i}^{2}W), \ \text{on } B_{2r}^{k}\setminus B_{r}^{k}\times B_{S_{1}}^{m-k}.
    \end{equation}
    The $C^{0}$ estimate is basically the same,
    \begin{equation}
        \begin{split}
        \|\di_{Z}^{I}\di_{W}^{J}\Psi_{i}(Z,W)\|_{C^{0}(B_{2r}^{k}\setminus B_{r}^{k}\times B_{S_{1}}^{m-k})}
        =&\epsilon_{i}^{|I|+2|J|-\delta}\|\di_{z}^{I}\di_{w}^{J}\psi_{i}\|_{C^{0}(B_{2r\epsilon_{i}}^{k}\setminus B_{r\epsilon_{i}}^{k}\times B_{S_{1}\epsilon_{i}^{2}}^{m-k})}\\
        \leq&\epsilon_{i}^{|I|+2|J|-\delta}\|\psi_{i}\|_{C^{4,\alpha}_{\delta}(Bl_{X}M)}(r\epsilon_{i})^{\delta-|I|-|J|}\\
        =& r^{\delta-|I|-|J|}\epsilon_{i}^{|J|}\|\psi_{i}\|_{C^{4,\alpha}_{\delta}(Bl_{X}M)}.
        \end{split}
    \end{equation}
    The estimate of H\"{o}lder continuity is 
    \begin{equation}\label{Ho_psi_r}
        \begin{split}
            &\frac{|\Psi_{i}(Z_{1},W_{1})-\Psi_{i}(Z_{2},W_{2})|}{|(Z_{1},W_{1})-(Z_{2},W_{2})|^{\alpha}}\\
            =& \epsilon_{i}^{-\delta}\frac{|\psi_{i}(\epsilon_{i}Z_{1},\epsilon_{i}^{2}W_{1})-\psi_{i}(\epsilon_{i}Z_{2},\epsilon_{i}^{2}W_{2})|}{|(Z_{1},\epsilon_{i}W_{1})-(Z_{2},\epsilon_{i}W_{2})|^{\alpha}}\frac{|(Z_{1},\epsilon_{i}W_{1})-(Z_{2},\epsilon_{i}W_{2})|^{\alpha}}{|(Z_{1},W_{1})-(Z_{2},W_{2})|^{\alpha}}\\
            \leq& \epsilon_{i}^{-\delta}\|\psi_{i} \|_{C^{4,\alpha}_{\delta}(Bl_{X}M)}(r\epsilon_{i})^{\delta}(1+(\epsilon_{i})^{2})^{\frac{\alpha}{2}}\\
            \leq& 2r^{\delta}\|\psi_{i}\|_{C^{4,\alpha}_{\delta}(Bl_{X}M)}.
        \end{split}
    \end{equation}
   
    Then, by induction, for any fixed $1<r$ and $0\leq|I|+|J|\leq 4$, when $i$ is taken sufficiently large, the following holds
    \begin{equation}\label{ScaleEstimate1}
            \|(\di_{Z}^{I}\di_{W}^{J}\Psi_{i})(r\ti{Z},r\ti{W})\|_{C^{0,\alpha}(B_{2}^{k}\setminus B_{1}^{k}\times B_{S_{1}}^{m-k})}
            \leq 2r^{\delta-|I|-|J|}\epsilon_{i}^{|J|}\|\psi_{i} \|_{C^{4,\alpha}_{\delta}(Bl_{X}M)},
    \end{equation}
    
    Hence, by (\ref{ScaleEstimate0}) and (\ref{ScaleEstimate1}), we have proved (\ref{S2E0}).

\medskip
\noindent\textbf{Proof of \eqref{Scale_L_eta}}.

    We show that 
     \begin{equation}
        L_{\eta}\Psi
        =L_{\epsilon_{i}^{-2}\Lambda_{i}^{*}\eta_{\epsilon_{i}}}\Lambda_{i}^{*}\ti{\Lambda}_{i}^{*}\Psi
        =\epsilon_{i}^{4}\Lambda_{i}^{*}(L_{\eta_{\epsilon_{i}}}\ti{\Lambda}_{i}^{*}\Psi)
        =\epsilon_{i}^{4}\Lambda_{i}^{*}(L_{\omega_{\eta,\epsilon_{i}}}\ti{\Lambda}_{i}^{*}\Psi).
    \end{equation}
    The third equality in the above can be seen first by
    \begin{equation}
        \begin{split}
            L_{\epsilon_{i}^{-2}\Lambda_{i}^{*}\eta_{\epsilon_{i}}}
            =&-\Delta_{\epsilon_{i}^{-2}\Lambda_{i}^{*}\eta_{\epsilon_{i}}}^{2}-Ric(\epsilon_{i}^{-2}\Lambda_{i}^{*}\eta_{\epsilon_{i}})\cdot \di\dbar\\
            =&-(\epsilon_{i}^{-2}\Lambda_{i}^{*}\eta_{\epsilon_{i}})^{p\ol{q}}\di_{p}\di_{\ol{q}}(\epsilon_{i}^{-2}\Lambda_{i}^{*}\eta_{\epsilon_{i}})^{k\ol{l}}\di_{k}\di_{\ol{l}}-Ric(\epsilon_{i}^{-2}\Lambda_{i}^{*}\eta_{\epsilon_{i}})_{k\ol{l}}(\epsilon_{i}^{-2}\Lambda_{i}^{*}\eta_{\epsilon_{i}})^{k\ol{q}}(\epsilon_{i}^{-2}\Lambda_{i}^{*}\eta_{\epsilon_{i}})^{p\ol{l}}\di_{p}\di_{\ol{q}}\\
            =&-\epsilon_{i}^{2}(\Lambda_{i}^{*}\eta_{\epsilon_{i}})^{p\ol{q}}\di_{p}\di_{\ol{q}}\epsilon_{i}^{2}(\Lambda_{i}^{*}\eta_{\epsilon_{i}})^{k\ol{l}}\di_{k}\di_{\ol{l}}-Ric(\Lambda_{i}^{*}\eta_{\epsilon_{i}})_{k\ol{l}}\epsilon_{i}^{2}(\Lambda_{i}^{*}\eta_{\epsilon_{i}})^{k\ol{q}}\epsilon_{i}^{2}(\Lambda_{i}^{*}\eta_{\epsilon_{i}})^{p\ol{l}}\di_{p}\di_{\ol{q}}\\
            =&\epsilon_{i}^{4}L_{\Lambda_{i}^{*}\eta_{\epsilon_{i}}}.
        \end{split}
    \end{equation}
    Secondly, by change of coordinates $z=\epsilon_{i}Z$,
    \begin{equation}
    Ric(\epsilon_{i}^{2}\eta_{\epsilon_{i}}(\epsilon_{i}Z))_{Z^{k}\ol{Z}^{l}}=Ric(\eta_{\epsilon_{i}}(\epsilon_{i}Z))_{Z^{k}\ol{Z}^{l}}=\epsilon_{i}^{2}Ric(\eta_{\epsilon_{i}})_{z^{k}\ol{z}^{l}}(\epsilon_{i}Z), 
    \end{equation}
    then, we have
    \begin{equation}
        \begin{split}
            L_{\Lambda_{i}^{*}\eta_{\epsilon_{i}}}\Lambda_{i}^{*}f=&-(\Lambda_{i}^{*}\eta_{\epsilon_{i}})^{Z^{p}\ol{Z}^{q}}\di_{Z^{p}}\di_{\ol{Z}^{q}}(\Lambda_{i}^{*}\eta_{\epsilon_{i}})^{Z^{k}\ol{Z}^{l}}\di_{Z^{k}}\di_{\ol{Z}^{l}}(f(\epsilon_{i}Z))\\
            &-Ric(\Lambda_{i}^{*}\eta_{\epsilon_{i}})_{Z^{k}\ol{Z}^{l}}(\Lambda_{i}^{*}\eta_{\epsilon_{i}})^{Z^{k}\ol{Z}^{q}}(\Lambda_{i}^{*}\eta_{\epsilon_{i}})^{Z^{p}\ol{Z}^{l}}\di_{Z^{p}}\di_{\ol{Z}^{q}}(f(\epsilon_{i}Z))\\
            =&-(\epsilon_{i}^{2}\eta_{\epsilon_{i}}(\epsilon_{i}Z))^{z^{p}\ol{z}^{q}}\di_{Z^{p}}\di_{\ol{Z}^{q}}(\epsilon_{i}^{2}\eta_{\epsilon_{i}}(\epsilon_{i}Z))^{z^{k}\ol{z}^{l}}\di_{Z^{k}}\di_{\ol{Z}^{l}}(f(\epsilon_{i}Z))\\
            &-Ric(\epsilon_{i}^{2}\eta_{\epsilon_{i}}(\epsilon_{i}Z))_{Z^{k}\ol{Z}^{l}}(\epsilon_{i}^{2}\eta_{\epsilon_{i}}(\epsilon_{i}Z))^{z^{k}\ol{z}^{q}}(\epsilon_{i}^{2}\eta_{\epsilon_{i}}((\epsilon_{i}Z))^{z^{p}\ol{z}^{l}}\di_{Z^{p}}\di_{\ol{Z}^{q}}(f(\epsilon_{i}Z))\\
            =&(L_{\eta_{\epsilon_{i}}}f)(\epsilon_{i}Z),
        \end{split}
    \end{equation}
    which is the desired relation.

\medskip
\noindent\textbf{Proof of \eqref{S2E1}.} 

    We show that for any fixed $R_{1}>0$,
     \begin{equation}
        \|L_{\eta}\Psi(\cdot,0)\|_{C^{0}(\ti{B}_{R_{1}}^{k})}
        =\|L_{\omega_{\eta,\epsilon_{i}}}\epsilon_{i}^{4}\ti{\Lambda}_{i}^{*}\Psi(\cdot,0)\|_{C^{0}(\ti{B}_{\epsilon_{i}R_{1}}^{k})}\to 0, \ (i\to\infty).
    \end{equation}
    Consider the following
    \begin{equation}\label{Target2}
        \begin{split}
        \|L_{\eta}\Psi(\cdot,0)\|_{C^{0}(\ti{B}_{R_{1}}^{k})}
        =&\|L_{\omega_{\eta,\epsilon_{i}}}\epsilon_{i}^{4}\ti{\Lambda}_{i}^{*}\Psi(\cdot,0)\|_{C^{0}(\ti{B}_{\epsilon_{i}R_{1}}^{k})}\\
        \leq& \|(L_{\omega_{\eta,\epsilon_{i}}}-L_{\omega_{\epsilon_{i}}|_{U_{i}}})\epsilon_{i}^{4}\ti{\Lambda}_{i}^{*}\Psi\|_{C^{0}(\ti{B}_{\epsilon_{i}R_{1}}^{k}\times B_{\epsilon_{i}^{2}S_{1}}^{m-k})}\\
        &+\epsilon_{i}^{4}\|L_{\omega_{\epsilon_{i}}|_{U_{i}}}(\ti{\Lambda}_{i}^{*}\Psi-\epsilon_{i}^{-\delta}\psi_{i}|_{U_{i}})\|_{C^{0}(\ti{B}_{\epsilon_{i}R_{1}}^{k}\times B_{\epsilon_{i}^{2}S_{1}}^{m-k})}\\
        &+\epsilon_{i}^{4-\delta}\|L_{\omega_{\epsilon_{i}}}\psi_{i}|_{U_{i}}\|_{C^{0}(\ti{B}_{\epsilon_{i}R_{1}}^{k}\times B_{\epsilon_{i}^{2}S_{1}}^{m-k})},
        \end{split}
    \end{equation}
    for any fixed constants $R_{1}$ and $S_{1}$. Let $i$ be large enough so that $\ti{B}_{\epsilon_{i}R_{1}}^{k}\times B_{\epsilon_{i}^{2}S_{1}}^{m-k}\subset \ti{B}_{r_{0}}^{k}\times B_{r_{0}}^{m-k}\subset U_{i}\subset Bl_{X}M$.
    We first note that according to (\ref{LocConverge1}), we have $\forall |J|\neq 0, |I|+|J|\leq 4$,
    \begin{align}
        \epsilon_{i}^{4}\|\di_{z}^{I}(\ti{\Lambda}_{i}^{*}\Psi-\epsilon_{i}^{-\delta}\psi_{i}|_{U_{i}})\|_{C^{0}(\ti{B}_{\epsilon_{i}R_{1}}^{k}\times B_{\epsilon_{i}^{2}S_{1}}^{m-k})} =&\epsilon_{i}^{4-|I|}\|\di_{Z}^{I}(\Psi-\Psi_{i})\|_{C^{0}(\ti{B}_{R_{1}}^{k}\times B_{S_{1}}^{m-k})} \label{Dz_P-p}\\
        \leq&\epsilon_{i}^{4-|I|}\|\Psi-\Psi_{i}\|_{C^{4}(\ti{B}_{R_{1}}^{k}\times B_{S_{1}}^{m-k})}, \nonumber \\
        \epsilon_{i}^{4}\|\di_{z}^{I}\di_{w}^{J}(\ti{\Lambda}_{i}^{*}\Psi-\epsilon_{i}^{-\delta}\psi_{i})\|_{C^{0}(\ti{B}_{\epsilon_{i}R_{1}}^{k}\times B_{\epsilon_{i}^{2}S_{1}}^{m-k})}
        =&\epsilon_{i}^{4-\delta}\|\di_{z}^{I}\di_{w}^{J}\psi_{i}\|_{C^{0}(\ti{B}_{\epsilon_{i}R_{1}}^{k}\times B_{\epsilon_{i}^{2}S_{1}}^{m-k})}. 
    \end{align}
    As a consequence, we have the following.
    \begin{equation}\label{LC2}
    \begin{split}
        \epsilon_{i}^{4}\|L_{\omega_{\epsilon_{i}}}(\ti{\Lambda}_{i}^{*}\Psi-\epsilon_{i}^{-\delta}\psi_{i})\|&_{C^{0}(\ti{B}_{\epsilon_{i}R_{1}}^{k}\times B_{s_{i}S_{1}}^{m-k})}\\
        \leq& \|g_{\omega_{\epsilon_{i}}}\|_{C^{4}(\ti{B}_{2r_{0}}^{k}\times B_{2r_{0}}^{m-k})}\|\Psi-\Psi_{i}\|_{C^{4}(\ti{B}_{R_{1}}^{k}\times B_{S_{1}}^{m-k})}+C\|L_{\omega_{\epsilon_{i}}}\psi_{i}\|_{C^{0,\alpha}_{\delta-4}(Bl_{X}M)}.
    \end{split}
    \end{equation}
     Since $\psi_{i}\xrightarrow{C^{4,\alpha}_{loc}}0$ on $M\setminus X$ and recalling that $\{q_{j}\}\subset M\setminus X$ are fixed points. Then we have
    \begin{equation}\label{Term_Fixing}
        \|\sum_{j=1}^{d}\psi_{i}(q_{j})f_{j}\|_{C^{0,\alpha}_{\delta-4}(Bl_{X}M)}\leq \sum_{j=1}^{d}|\psi_{i}(q_{j})|\|f_{j}\|_{C^{0,\alpha}(M)}\leq C\sum_{j=1}^{d}|\psi_{i}(q_{j})|\to 0, \ (i\to\infty).
    \end{equation}
    Combining our assumption $\|\ti{L}_{\omega_{\epsilon_{i}}}\psi_{i}\|_{C^{0,\alpha}_{\delta-4}(Bl_{X}M)}\to 0, \ (i\to\infty)$, we have
    \begin{equation}
        \|L_{\omega_{\epsilon_{i}}}\psi_{i}\|_{C^{0,\alpha}_{\delta-4}(Bl_{X}M)}\\
        \leq \Big(\|\ti{L}_{\omega_{\epsilon_{i}}}\psi_{i}\|_{C^{0,\alpha}_{\delta-4}(Bl_{X}M)}+\|\sum_{j=1}^{l}\psi_{i}(q_{j})f_{j}\|_{C^{0,\alpha}_{\delta-4}(Bl_{X}M)}\Big)
    \end{equation}
    tends to $0$. Therefore,
    \begin{equation}
        \epsilon_{i}^{4}\|L_{\omega_{\epsilon_{i}}}(\ti{\Lambda}_{i}^{*}\Psi-\epsilon_{i}^{-\delta}\psi_{i})\|_{C^{0}(\ti{B}_{\epsilon_{i}R_{1}}^{k}\times B_{\epsilon_{i}^{2}S_{1}}^{m-k})}\to 0, \ (i\to\infty).
    \end{equation}
    
    Then, for the first term in the R.H.S. of (\ref{Target2}), we calculate, for example, 
    \begin{equation}
        \begin{split}
            &(\Delta_{\omega_{\eta,\epsilon_{i}}}^{2}-\Delta_{\omega_{\epsilon_{i}}}^{2})\epsilon_{i}^{4}\ti{\Lambda}_{i}^{*}\Psi\\
            =&g_{\omega_{\eta,\epsilon_{i}}}^{z^{p}\ol{z}^{q}}\di_{z^{p}}\di_{\ol{z}^{q}}g_{\omega_{\eta,\epsilon_{i}}}^{z^{k}\ol{z}^{l}}\di_{z^{k}}\di_{\ol{z}^{l}}\epsilon_{i}^{4}\ti{\Lambda}_{i}^{*}\Psi-g_{\omega_{\epsilon_{i}}}^{z^{p}\ol{z}^{q}}\di_{z^{p}}\di_{\ol{z}^{q}}g_{\omega_{\epsilon_{i}}}^{z^{k}\ol{z}^{l}}\di_{z^{k}}\di_{\ol{z}^{l}}\epsilon_{i}^{4}\ti{\Lambda}_{i}^{*}\Psi\\
            =&(g_{\omega_{\eta,\epsilon_{i}}}^{z^{p}\ol{z}^{q}}-g_{\omega_{\epsilon_{i}}}^{z^{p}\ol{z}^{q}})\di_{z^{p}}\di_{\ol{z}^{q}}g_{\omega_{\eta,\epsilon_{i}}}^{z^{k}\ol{z}^{l}}\di_{z^{k}}\di_{\ol{z}^{l}}\epsilon_{i}^{4}\ti{\Lambda}_{i}^{*}\Psi
            +g_{\omega_{\epsilon_{i}}}^{z^{p}\ol{z}^{q}}\di_{z^{p}}\di_{\ol{z}^{q}}(g_{\omega_{\eta,\epsilon_{i}}}^{z^{k}\ol{z}^{l}}-g_{\omega_{\epsilon_{i}}}^{z^{k}\ol{z}^{l}})\di_{z^{k}}\di_{\ol{z}^{l}}\epsilon_{i}^{4}\ti{\Lambda}_{i}^{*}\Psi
        \end{split}
    \end{equation}
    We calculate the highest order term only. Note that by (\ref{Dz_P-p}), we have uniform upper bounds:
    \begin{equation}
        \|\di_{z}^{I}\epsilon_{i}^{4}\ti{\Lambda}^{*}_{i}\Psi\|_{C^{0}(\ti{B}_{\epsilon_{i}R_{1}}^{k}\times B_{\epsilon_{i}^{2}S_{1}}^{m-k})}\leq 2\epsilon_{i}^{4-|I|} \|\psi_{i}\|_{C^{4,\alpha}_{\delta}(Bl_{X}M)}\leq C\epsilon_{i}^{4-|I|}, \ \forall |I|\leq 4,
    \end{equation}
    for sufficiently large $i$. Then, let $(z_{(1)},w)\in\ti{B}_{(1),\epsilon_{i}R_{1}}^{k}\times B_{\epsilon_{i}^{2}S_{1}}^{m-k}$, by (\ref{Compare3.5}), we have the following estimate
    \begin{equation}
        \begin{split}
           |g_{\omega_{\epsilon_{i}}}^{z^{p}\ol{z}^{q}}(g_{\omega_{\eta,\epsilon_{i}}}^{z^{k}\ol{z}^{l}}-g_{\omega_{\epsilon_{i}}}^{z^{k}\ol{z}^{l}})\di_{z^{p}}\di_{\ol{z}^{q}}\di_{z^{k}}\di_{\ol{z}^{l}}\epsilon_{i}^{4}\ti{\Lambda}_{i}^{*}\Psi(z_{(1)},w)|\leq& C_{1}(|z_{(1)}^{1}|(1+\sum_{j\neq 1}|z_{(1)}^{j}|^{2})^{\frac{1}{2}}+|w|)C_{2}\\
           \leq& C\epsilon_{i}(R_{1}+\epsilon_{i}S_{1})\to 0, \ (i\to\infty).
        \end{split}
    \end{equation}
    For the Ricci term, we have
    \begin{equation}
        \begin{split}
            &Ric(\omega_{\eta,\epsilon_{i}})\cdot\di\dbar\epsilon_{i}^{4}\ti{\Lambda}_{i}^{*}\Psi(z_{(1)},w)-Ric(\omega_{\epsilon_{i}})\cdot\di\dbar\epsilon_{i}^{4}\ti{\Lambda}_{i}^{*}\Psi(z_{(1)},w)\\
            =&\Big(Ric(\omega_{\eta,\epsilon_{i}})_{z^{p}\ol{z}^{q}}-Ric(\omega_{\epsilon_{i}})_{z^{p}\ol{z}^{q}} \Big)g_{\omega_{\eta,\epsilon_{i}}}^{z^{p}\ol{z}^{l}}g_{\omega_{\eta,\epsilon_{i}}}^{z^{k}\ol{z}^{q}}\di_{z^{k}}\di_{\ol{z}^{l}}\epsilon_{i}^{4}\ti{\Lambda}_{i}^{*}\Psi(z_{(1)},w)\\
            &+Ric(\omega_{\epsilon_{i}})_{z^{p}\ol{z}^{q}}(g_{\omega_{\eta,\epsilon_{i}}}^{z^{p}\ol{z}^{l}}-g_{\omega_{\epsilon_{i}}}^{z^{p}\ol{z}^{l}})g_{\omega_{\eta,\epsilon_{i}}}^{z^{k}\ol{z}^{q}}\di_{z^{k}}\di_{\ol{z}^{l}}\epsilon_{i}^{4}\ti{\Lambda}_{i}^{*}\Psi(z_{(1)},w)\\
            &+Ric(\omega_{\epsilon_{i}})_{z^{p}\ol{z}^{q}}g_{\omega_{\epsilon_{i}}}^{z^{p}\ol{z}^{l}}(g_{\omega_{\eta,\epsilon_{i}}}^{z^{k}\ol{z}^{q}}-g_{\omega_{\epsilon_{i}}}^{z^{k}\ol{z}^{q}})\di_{z^{k}}\di_{\ol{z}^{l}}\epsilon_{i}^{4}\ti{\Lambda}_{i}^{*}\Psi(z_{(1)},w).
        \end{split}
    \end{equation}
    Then, according to (\ref{Compare3.5}) again,
    \begin{equation}\label{Compare_Ric}
        \begin{split}
            &|(Ric(\omega_{\eta,\epsilon_{i}})\cdot\di\dbar\epsilon_{i}^{4}\ti{\Lambda}_{i}^{*}\Psi-Ric(\omega_{\epsilon_{i}})\cdot\di\dbar\epsilon_{i}^{4}\ti{\Lambda}_{i}^{*}\Psi)(z_{(1)},w)|\\
            \leq& C_{1}\|g_{\omega_{\eta,\epsilon_{i}}}-g_{\omega_{\epsilon_{i}}}\|_{C^{2}(\ti{B}_{\epsilon_{i}R_{1}}^{k}\times B_{\epsilon_{i}^{2}S_{1}}^{m-k})}\epsilon_{i}^{2}+C_{2}\|g_{\omega_{\eta,\epsilon_{i}}}-g_{\omega_{\epsilon_{i}}}\|_{C^{0}(\ti{B}_{\epsilon_{i}R_{1}}^{k}\times B_{\epsilon_{i}^{2}S_{1}}^{m-k})}\epsilon_{i}^{2}\\
            \leq&C_{1}(\epsilon_{i}R_{1}+\epsilon_{i}^{2}S_{1})+C_{2}\epsilon_{i}^{2}(\epsilon_{i}R_{1}+\epsilon_{i}^{2}S_{1})\to 0, \ (i\to\infty).
        \end{split}
    \end{equation}
    Hence, we may derive,
    \begin{equation}
        \|(L_{\omega_{\eta,\epsilon_{i}}}-L_{\omega_{\epsilon_{i}}})\epsilon_{i}^{4}\ti{\Lambda}_{i}^{*}\Psi\|_{C^{0}(\ti{B}_{\epsilon_{i}R_{1}}^{k}\times B_{\epsilon_{i}^{2}S_{1}}^{m-k})}\to 0, \ (i\to\infty).
    \end{equation}
    
    Finally, without loss of generality, we may suppose $R_{1}>1$. Then, by definition of the weighted norm, we have
    \begin{equation}
    \begin{split}
        \epsilon_{i}^{4-\delta}\|L_{\omega_{\epsilon_{i}}}\psi_{i}\|_{C^{0}(\ti{B}_{\epsilon_{i}R_{1}}^{k}\times B_{\epsilon_{i}^{2}S_{1}}^{m-k})}
        \leq&\epsilon_{i}^{4-\delta}\|L_{\omega_{\epsilon_{i}}}\psi_{i}\|_{C^{0}(\ti{B}_{\epsilon_{i}}^{k}\times B_{\epsilon_{i}}^{m-k})}+\sup_{\epsilon_{i}<r<R_{1}\epsilon_{i}}\epsilon_{i}^{4-\delta}\|L_{\omega_{\epsilon_{i}}}\psi_{i}\|_{C^{0}(\ti{B}_{2r}^{k}\setminus\ti{B}_{r}^{k}\times B_{2r}^{m-k})}\\
        \leq& \|L_{\omega_{\epsilon_{i}}}\psi_{i}\|_{C^{0,\alpha}_{\delta-4}(Bl_{X}M)}+\sup_{\epsilon_{i}<r<R_{1}\epsilon_{i}}\epsilon_{i}^{4-\delta}r^{\delta-4}\|L_{\omega_{\epsilon_{i}}}\psi_{i}\|_{C^{0,\alpha}_{\delta-4}(Bl_{X}M)}\\
        \leq & C\|L_{\omega_{\epsilon_{i}}}\psi_{i}\|_{C^{0,\alpha}_{\delta-4}(Bl_{X}M)}\to 0, \ (i\to\infty).
    \end{split}
    \end{equation}
   
    According to (\ref{Target2}), we have proved (\ref{S2E1}).

\medskip
\noindent\textbf{Proof of \eqref{S3E1}.}

We show that $\Delta^{2}_{\mathrm{euc}}\Phi_{\infty}=0$.

    Consider the following
    \begin{equation}\label{Target4}
        \begin{split}
            \|\Delta_{Z}^{2}\Phi_{\infty}(Z,0)\|_{C^{0}(B_{2r}^{k}\setminus B_{r}^{k})}
            =& \|(\Delta_{\mathrm{euc}}^{2}|z_{i}|^{4}\ti{\Theta}_{i}^{*}\Phi_{\infty})(z,0) \|_{C^{0}(B_{2r|z_{i}|}^{k}\setminus B_{r|z_{i}|}^{k})}\\
            \leq&|z_{i}|^{4}\|(\Delta_{euc}^{2}-L_{\omega_{\epsilon_{i}}|_{U_{i}}})\ti{\Theta}_{i}^{*}\Phi_{\infty}\|_{C^{0}(B_{2r|z_{i}|}^{k}\setminus B_{r|z_{i}|}^{k}\times B_{2r|z_{i}|^{2}}^{m-k})}\\
            &+ |z_{i}|^{4}\|L_{\omega_{\epsilon_{i}}|_{U_{i}}}(\ti{\Theta}_{i}^{*}\Phi_{\infty}-|z_{i}|^{-\delta}\psi_{i}|_{U_{i}})(z,w) \|_{C^{0}(B_{2r|z_{i}|}^{k}\setminus B_{r|z_{i}|}^{k}\times B_{2r|z_{i}|^{2}}^{m-k})}\\
            &+|z_{i}|^{4-\delta}\|L_{\omega_{\epsilon_{i}}}\psi_{i}|_{U_{i}}\|_{C^{0}(B_{2r|z_{i}|}^{k}\setminus B_{r|z_{i}|}^{k}\times B_{2r|z_{i}|^{2}}^{m-k})}.
        \end{split}
    \end{equation}
    For the second term, we have an estimate similar to (\ref{LC2}),
    \begin{equation}
        \begin{split}
            &|z_{i}|^{4}\|L_{\omega_{\epsilon_{i}}|_{U_{i}}}(\ti{\Theta}_{i}^{*}\Phi_{\infty}-|z_{i}|^{-\delta}\psi_{i}|_{U_{i}})(z,w) \|_{C^{0}(B_{2r|z_{i}|}^{k}\setminus B_{r|z_{i}|}^{k}\times B_{2r|z_{i}|^{2}}^{m-k})}\\
            \leq&C_{1}\|\Phi_{\infty}-\Phi_{i}\|_{C^{4}(B_{2r}^{k}\setminus B_{r}^{k}\times B_{2r}^{m-k})}+C_{2}\|L_{\omega_{\epsilon_{i}}}\psi_{i}\|_{C^{4,\alpha}_{\delta}(Bl_{X}M)}\to 0, (i\to\infty),
        \end{split}
    \end{equation}
    For the first term, note that by (\ref{D_psi_euc}),
    \begin{equation}
    \begin{split}
        &\|\di_{z}^{I}|z_{i}|^{4}\ti{\Theta}_{i}^{*}\Phi_{\infty}\|_{C^{0}(B_{2r|z_{i}|}^{k}\setminus B_{r|z_{i}|}^{k}\times B_{2r|z_{i}|^{2}}^{m-k})}
        =|z_{i}|^{4-|I|}\|\di_{Z}^{I}\Phi_{\infty}\|_{C^{0}(B_{2r}^{k}\setminus B_{r}^{k}\times B_{2r}^{m-k})}\\
        \leq& |z_{i}|^{4-|I|}\|\Phi_{\infty}-\Phi_{i}\|_{C^{4}(B_{2r}^{k}\setminus B_{r}^{k}\times B_{2r}^{m-k})}+2\|\psi_{i}\|_{C^{4,\alpha}_{\delta}(Bl_{X}M)}|z_{i}|^{4-|I|}r^{\delta-|I|}\\
        \leq& 2C |z_{i}|^{4-|I|}r^{\delta-|I|}, \ \forall |I|\leq 4.
    \end{split}
    \end{equation}
    Then by (\ref{Compare4}) the main term of $(\Delta_{\mathrm{euc}}^{2}-L_{\omega_{\epsilon_{i}}|_{U_{i}}})|z_{i}|^{4}\ti{\Theta}_{i}^{*}\Phi_{\infty}$ satisfies
    \begin{equation}
        \begin{split}
            &\big|(\delta^{z^{k}\ol{z}^{l}}\delta^{z^{p}\ol{z}^{q}}-g_{\omega_{\epsilon_{i}}}^{z^{k}\ol{z}^{l}}g_{\omega_{\epsilon_{i}}}^{z^{p}\ol{z}^{q}})\di_{z^{k}}\di_{\ol{z}^{l}}\di_{z^{p}}\di_{\ol{z}^{q}}|z_{i}|^{4}\ti{\Theta}_{i}^{*}\Phi_{\infty}(z,w)\Big|\\
            =&\Big|[\delta^{z^{k}\ol{z}^{l}}(\delta^{z^{p}\ol{z}^{q}}-g_{\omega_{\epsilon_{i}}}^{z^{p}\ol{z}^{q}})+(\delta^{z^{k}\ol{z}^{l}}-g_{\omega_{\epsilon_{i}}}^{z^{k}\ol{z}^{l}})g_{\omega_{\epsilon_{i}}}^{z^{p}\ol{z}^{q}}]\di_{z^{k}}\di_{\ol{z}^{l}}\di_{z^{p}}\di_{\ol{z}^{q}}|z_{i}|^{4}\ti{\Theta}_{i}^{*}\Phi_{\infty}(z,w)\Big|\\
            \leq& (O(|z|+|w|)+O(\epsilon_{i}^{2k-2}|z|^{2-2k}))2Cr^{\delta-4}\\
            \leq& \ti{C}r^{\delta-4}\Big(r|z_{i}|+r|z_{i}|^{2}+(\frac{\epsilon_{i}}{r|z_{i}|})^{2k-2}\Big)\to 0, \ (i\to\infty),
        \end{split}
    \end{equation}
    for any $(z,w)\in B_{2r|z_{i}|}^{k}\setminus B_{r|z_{i}|}^{k}\times B_{2r|z_{i}|^{2}}^{m-k}$.
    Using (\ref{Compare4}), the estimates for other terms are totally similar to (\ref{Compare_Ric}). Hence, we may derive that
    \begin{equation}
        \|(\Delta_{z}^{2}-L_{\omega_{\epsilon_{i}}})|z_{i}|^{4}\ti{\Theta}_{i}^{*}\Phi_{\infty}\|_{C^{0}(B_{2r|z_{i}|}^{k}\setminus B_{r|z_{i}|}^{k}\times B_{2r|z_{i}|^{2}}^{m-k})}\to 0, \ (i\to\infty).
    \end{equation}
    For the last term, according to (\ref{Term_Fixing}), we have
    \begin{equation}
        \begin{split}
            &|z_{i}|^{4-\delta}\|L_{\omega_{\epsilon_{i}}}\psi_{i}|_{U_{i}}\|_{C^{0}(B_{2r|z_{i}|}^{k}\setminus B_{r|z_{i}|}^{k}\times B_{2r|z_{i}|^{2}}^{m-k})}\\
            \leq& r^{\delta-4}(r|z_{i}|)^{4-\delta}\|\ti{L}_{\omega_{\epsilon_{i}}}\psi_{i}|_{U_{i}}\|_{C^{0}(B_{2r|z_{i}|}^{k}\setminus B_{r|z_{i}|}^{k}\times B_{2r|z_{i}|}^{m-k})}+\sum_{j=1}^{d}|\psi_{i}(q_{j})|\|f_{j}\|_{C^{0}(M)}\\
            \leq & r^{\delta-4}\|\ti{L}_{\omega_{\epsilon_{i}}}\psi_{i}\|_{C^{0,\alpha}_{\delta-4}(Bl_{X}M)}+C\sum_{j=1}^{d}|\psi_{i}(q_{j})|\to 0, \ (i\to\infty).
        \end{split}
    \end{equation}
    Then by letting $i\to\infty$ in the right-hand side of (\ref{Target4}), we have $\Delta_{z}^{2}\Phi_{\infty}(Z,0)=0$.

    \section{Related estimates for codimension \texorpdfstring{$2$}{2}}\label{App_D}

    \subsection{Related estimates on $\epsilon^{2}\log\frac{d^{2}}{\epsilon^{2}}$ } \label{E_log}
    In Proposition \ref{cut_S-S} and Proposition \ref{cut_Qlog}, we used the following estimates:
    \begin{align}
        &\| \epsilon^{2}\log\frac{d^{2}}{\epsilon^{2}} \|_{C^{4,\alpha}_{2}(\{ \epsilon^{\beta}\leq d\leq \epsilon^{\theta} \})}\leq C\epsilon^{2-2\beta}\log\frac{1}{\epsilon},  \label{Sec6_need1} \\ 
        &\| \epsilon^{2}\log\frac{d^{2}}{\epsilon^{2}} \|_{\ti{C}^{4,\alpha}_{0,\delta}(\{ \epsilon^{\beta}\leq d\leq \epsilon^{\theta} \})}\leq C\epsilon^{2-\theta\delta}[\log\frac{1}{\epsilon}]^{1-\delta}.    \label{Sec6_need2}
    \end{align}
    Fix $\frac{2}{3}\epsilon^{\beta}\leq r\leq 2\epsilon^{\theta}$, and take  coordinates $(z,w)\in U$ of Lemma \ref{holosubcoor} type. Recall that $d(z,w)=|z|(1+\rho(z,w))^{\frac{1}{2}}$, then the H\"{o}lder coefficient satisfies:
    \begin{equation}
        \begin{split}
            &\|\log\frac{d^{2}}{\epsilon^{2}}(rZ,rW)\|_{C^{\alpha}(B_{2}^{2}\setminus B_{1}^{2}\times B_{2}^{m-2})}\\
            =& \sup\frac{|\log\frac{d^{2}}{\epsilon^{2}}(rZ_{1},rW_{1})-\log\frac{d^{2}}{\epsilon^{2}}(rZ_{2},rW_{2})|}{|(Z_{1},W_{1})-(Z_{2},W_{2})|^{\alpha}}\\
            \leq &  \sup\frac{2|\log|Z_{1}|-\log|Z_{2}||+|\log(1+\rho(rZ_{1},rW_{1}))-\log(1+\rho(rZ_{1},rW_{1}))|}{|(Z_{1},W_{1})-(Z_{2},W_{2})|^{\alpha}}\\
            \leq& \|\log|Z|\|_{C^{\alpha}(B_{2}^{2}\setminus B_{1}^{2})}+\frac{r^{\alpha}}{2}\|\log(1+\rho)\|_{C^{\alpha}(\ol{U})}\leq C_{0}.
        \end{split}
    \end{equation}
    Hence, 
    \begin{equation}
    \begin{split}
        \|\log\frac{d^{2}}{\epsilon^{2}}(rZ,rW)\|_{C^{0,\alpha}(B_{2}^{2}\setminus B_{1}^{2}\times B_{2}^{m-2})}
        \leq& \|\log\frac{d^{2}}{\epsilon^{2}}(rZ,rW)\|_{C^{0}(B_{2}^{2}\setminus B_{1}^{2}\times B_{2}^{m-2})}\\
        &+\|\log\frac{d^{2}}{\epsilon^{2}}(rZ,rW)\|_{C^{\alpha}(B_{2}^{2}\setminus B_{1}^{2}\times B_{2}^{m-2})}\\
        \leq& C_{1}\log\frac{r}{\epsilon}+C_{0}\leq C_{2}\log\frac{1}{\epsilon}.
    \end{split}
    \end{equation}
    Then, notice that $r^{-2}$ is decreasing and $r^{-\delta}$ is increasing when $r\in(\frac{2}{3}\epsilon^{\beta},2\epsilon^{\theta})$, we may derive that
    \begin{align}
        &\| \epsilon^{2}\log\frac{d^{2}}{\epsilon^{2}} \|_{C^{0,\alpha}_{2}(\{ \epsilon^{\beta}\leq d\leq \epsilon^{\theta} \})}
        \leq C_{3}\epsilon^{2-2\beta}\log\frac{1}{\epsilon},\\
        &\| \epsilon^{2}\log\frac{d^{2}}{\epsilon^{2}} \|_{\ti{C}^{4,\alpha}_{0,\delta}(\{ \epsilon^{\beta}\leq d\leq \epsilon^{\theta} \})}
        \leq C_{3}\epsilon^{2-\theta\delta}[\log\frac{1}{\epsilon}]^{1-\delta}.
    \end{align}
    For the estimates on derivatives, we show $C^{1,\alpha}$ estimate as an example.

    Calculate in local coordinates $U$ directly,
    \begin{equation}
        \begin{split}
            \di_{Z^{i}}\big[\log\frac{d^{2}}{\epsilon^{2}}(rZ,rW)\big]
            =& \di_{Z^{i}}\big[2\log|rZ|+\log(1+\rho(rZ,rW)) \big]\\
            =& 2\di_{Z^{i}}\log|Z|+\di_{Z^{i}}\frac{r(\di_{z^{i}}\rho)(rZ,rW)}{1+\rho(rZ,rW)}.
        \end{split}
    \end{equation}
    Hence, recall we are using the change of coordinates $(rZ,rW)=(z,w)\in U$,
    \begin{equation}
        \begin{split}
            &\big\|\di_{Z^{i}}\big[\log\frac{d^{2}}{\epsilon^{2}}(rZ,rW)\big]\big\|_{C^{0,\alpha}(B_{2}^{2}\setminus B_{1}^{2}\times B_{2}^{m-2})}\\
            \leq& 2\|\di_{Z^{i}}\log|Z|\|_{C^{0,\alpha}(B_{2}^{2}\setminus B_{1}^{2})}\\
            &+ r\Big\| \frac{\di_{z^{i}}\rho}{1+\rho} \Big\|_{C^{0}(\ol{U})}+r^{1+\alpha}\Big\| \frac{\di_{z^{i}}\rho}{1+\rho} \Big\|_{C^{\alpha}(\ol{U})}\\
            \leq& C_{4}.
        \end{split}
    \end{equation}
    Take the weight function into consideration, we have
    \begin{align}
        &r^{-2}\big\|\di_{Z^{i}}\big[\epsilon^{2}\log\frac{d^{2}}{\epsilon^{2}}(rZ,rW)\big]\big\|_{C^{0,\alpha}(B_{2}^{2}\setminus B_{1}^{2}\times B_{2}^{m-2})}\leq C_{4}\epsilon^{2-2\beta},\\
        &[\ti{\tau}(r)]^{-\delta}\big\|\di_{Z^{i}}\big[\epsilon^{2}\log\frac{d^{2}}{\epsilon^{2}}(rZ,rW)\big]\big\|_{C^{0,\alpha}(B_{2}^{2}\setminus B_{1}^{2}\times B_{2}^{m-2})}\leq C_{4}\epsilon^{2-\theta\delta}[\log\frac{1}{\epsilon}]^{-\delta}.
    \end{align}
    Using the same analysis, we may derive the estimates of higher order derivatives. Therefore we may conclude the first inequality. Moreover, this kind of estimates suits for various estimates involving $\epsilon^{2}\log\frac{d^{2}}{\epsilon^{2}}$. We just need to be careful about the weight function and domains.

    \subsection{Proof of Lemma \ref{Q2}}\label{PQ2}

    \begin{proof}[Proof of Lemma \ref{Q2}]
    We follow the proof in \cite[Lemma 8.18]{Sze2014} closely. Consider $\chi_{t}=t\vphi+(1-t)\psi, t\in\R$ and denote $\omega_{\chi_{t}}=\omega+\sqrt{-1}\di\dbar\chi_{t}$. Then,
    \begin{equation}
    \begin{split}
        Q_{\omega}(\chi_{t+h})-Q_{\omega}(\chi_{t})
        =& S(\omega+\sqrt{-1}\di\dbar\chi_{t+h})-S(\omega+\sqrt{-1}\di\dbar\chi_{t})-L_{\omega}(\chi_{t+h}-\chi_{t})\\
        =& S(\omega_{\chi_{t}}+\sqrt{-1}h(\vphi-\psi))-S(\omega_{\chi_{t}})-hL_{\omega}(\vphi-\psi).
    \end{split}
    \end{equation}
    Take limit of $h$, we have
    \begin{equation}
        \begin{split}
            \lim_{h\to 0}\frac{Q_{\omega}(\chi_{t+h})-Q_{\omega}(\chi_{t})}{h}
            =& \frac{d}{dh}\Big|_{h=0}S(\omega_{\chi_{t}}+\sqrt{-1}h(\vphi-\psi))-L_{\omega}(\vphi-\psi)\\
            =& L_{\omega_{\chi_{t}}}(\vphi-\psi)-L_{\omega}(\vphi-\psi).
        \end{split}
    \end{equation}
    By the mean value theorem of Fr\'{e}chet derivative (\cite[Theorem 7.2-1 on p.~466]{Ciarlet}), 
    \begin{equation}
    \begin{split}
        \|Q_{\omega}(\vphi)-Q_{\omega}(\psi)\|_{\ti{C}^{0,\alpha}_{-4,\delta}(\{r_{1}\leq d\leq r_{2} \})}
        \leq& \sup_{t\in[0,1]}\|\frac{d}{dt}Q_{\omega}(\chi_{t})\|_{\ti{C}^{4,\alpha}_{0,\delta}\to \ti{C}^{0,\alpha}_{-4,\delta}}\|\vphi-\psi\|_{\ti{C}^{4,\alpha}_{0,\delta}(\{r_{1}\leq d\leq r_{2} \})} \\
        =& \sup_{t\in[0,1]}\|(L_{\omega_{\chi_{t}}}-L_{\omega})\|_{\ti{C}^{4,\alpha}_{0,\delta}\to \ti{C}^{0,\alpha}_{-4,\delta}} \|\vphi-\psi\|_{\ti{C}^{4,\alpha}_{0,\delta}(\{r_{1}\leq d\leq r_{2} \})},
    \end{split}
    \end{equation}
    As we have shown in Proposition \ref{UniEst}, when $c$ is taken sufficiently small, there exists a constant $C_{0}$ independent of $r_{1},r_{2}$ and $\epsilon$ such that,
    \begin{equation}
        \| g_{\omega_{\chi_{t}}}^{-1}-g_{\omega}^{-1} \|_{C^{2,\alpha}_{0}(\{r_{1}\leq d\leq r_{2} \})},\| g_{\omega_{\chi_{t}}}-g_{\omega} \|_{C^{2,\alpha}_{0}(\{r_{1}\leq d\leq r_{2} \})}\leq C_{0}\| \chi_{t} \|_{C^{4,\alpha}_{2}(\{r_{1}\leq d\leq r_{2} \})}
    \end{equation}
    Then, for any $f\in \ti{C}^{4,\alpha}_{0,\delta}(\{r_{1}\leq d\leq r_{2} \})$, the main term of $(L_{\omega_{\chi_{t}}}-L_{\omega})f$ in local coordinates of Lemma \ref{holosubcoor} type is given by bi-laplacian:
    \begin{equation}
    \begin{split}
        (g_{\omega_{\chi_{t}}}^{i\ol{j}} \di_{i}\di_{\ol{j}} g_{\omega_{\chi_{t}}}^{k\ol{l}}\di_{k}\di_{\ol{l}} -g_{\omega}^{i\ol{j}}\di_{i}\di_{\ol{j}} g_{\omega}^{k\ol{l}}\di_{k}\di_{\ol{l}})f
        =& ([g_{\omega_{\chi_{t}}}^{i\ol{j}}-g_{\omega}^{i\ol{j}}] \di_{i}\di_{\ol{j}} g_{\omega_{\chi_{t}}}^{k\ol{l}}\di_{k}\di_{\ol{l}})f \\
        &+(g_{\omega}^{i\ol{j}} \di_{i}\di_{\ol{j}} [g_{\omega_{\chi_{t}}}^{k\ol{l}}- g_{\omega}^{k\ol{l}}]\di_{k}\di_{\ol{l}})f.
    \end{split}
    \end{equation}
    We estimate a typical term provided by the second term in the R.H.S. as an example. Fix $\frac{2}{3}r_{1}<r<r_{2}$,
    \begin{equation}
        \begin{split}
            &[\ti{\tau}(r)]^{-\delta}r^{4}\|g_{\omega}^{i\ol{j}} (\di_{i}\di_{\ol{j}} [g_{\omega_{\chi_{t}}}^{k\ol{l}}- g_{\omega}^{k\ol{l}}])(\di_{k}\di_{\ol{l}}f)(rZ,rW)\|_{C^{0,\alpha}(B_{2}^{2}\setminus B_{1}^{2}\times B_{2}^{m-2})}
            \\
            \leq & \|g_{\omega}\|_{C^{0,\alpha}(M)}
            r^{2}\|(\di_{i}\di_{\ol{j}} [g_{\omega_{\chi_{t}}}^{k\ol{l}}- g_{\omega}^{k\ol{l}}])(rZ,rW)\|_{C^{0,\alpha}(B_{2}^{2}\setminus B_{1}^{2}\times B_{2}^{m-2})}
            \\
            &\cdot [\ti{\tau}(r)]^{-\delta}
            r^{2}\|(\di_{k}\di_{\ol{l}}f)(rZ,rW)\|_{C^{0,\alpha}(B_{2}^{2}\setminus B_{1}^{2}\times B_{2}^{m-2})}
            \\
            \leq& C_{1}\|[g_{\omega_{\chi_{t}}}^{k\ol{l}}- g_{\omega}^{k\ol{l}}](rZ,rW)\|_{C^{2,\alpha}_{0}(B_{2}^{2}\setminus B_{1}^{2}\times B_{2}^{m-2})}
            \|f\|_{\ti{C}^{4,\alpha}_{0,\delta}(\{r_{1}\leq d\leq r_{2}\})}
            \\
            \leq& C_{1}C_{0}\| \chi_{t} \|_{C^{4,\alpha}_{2}(\{r_{1}\leq d\leq r_{2} \})}\|f\|_{\ti{C}^{4,\alpha}_{0,\delta}(\{r_{1}\leq d\leq r_{2}\})}
        \end{split}
    \end{equation}
    We may derive that 
    \begin{equation}
        \sup_{t\in[0,1]}\|(L_{\omega_{\chi_{t}}}-L_{\omega})\|_{\ti{C}^{4,\alpha}_{0,\delta}\to \ti{C}^{0,\alpha}_{-4,\delta}}\leq C_{2}(\|\vphi\|_{C^{4,\alpha}_{2}(\{r_{1}\leq d\leq r_{2}\})}+\|\psi\|_{C^{4,\alpha}_{2}(\{r_{1}\leq d\leq r_{2}\})}).
    \end{equation}
    Thus, we can conclude that
    \begin{equation}
    \begin{split}
        \|Q_{\omega}(\vphi)-Q_{\omega}(\psi)\|_{\ti{C}^{0,\alpha}_{-4,\delta}(\{r_{1}\leq d\leq r_{2} \})}
       \leq &  
       C_{2}(\|\vphi\|_{C^{4,\alpha}_{2}(\{r_{1}\leq d\leq r_{2}\})}+\|\psi\|_{C^{4,\alpha}_{2}(\{r_{1}\leq d\leq r_{2}\})})
       \\
       &\cdot \|\vphi-\psi\|_{\ti{C}^{4,\alpha}_{0,\delta}(\{r_{1}\leq d\leq r_{2} \})}.
    \end{split}
    \end{equation}
\end{proof}

    \subsection{Related estimates in the proof of Proposition \ref{F}}\label{EF}
    We mainly estimate the weighted norm of $F_{\epsilon}$ when $2\epsilon^{\beta}\leq d\leq \epsilon^{\theta}$. According to calculation, we have:
    \begin{equation}
    \begin{split}
        F_{\epsilon}
        =& -\gamma(\frac{2d}{\epsilon^{\theta}})L_{\omega}(\epsilon^{2}\log\frac{d^{2}}{\epsilon^{2}}) \\
        =& -\gamma(\frac{2d}{\epsilon^{\theta}})\epsilon^{2}(\Delta_{\omega}^{2}\log d^{2}+Ric_{\omega}\cdot\di\dbar\log d^{2})\\
        =& -\gamma(\frac{2d}{\epsilon^{\theta}})\epsilon^{2}(\frac{G_{3}}{d^{3}}+\frac{G_{2}}{d^{2}}),
    \end{split}
    \end{equation}
    where $G_{1}$ and $G_{2}$ are smooth functions defined on $M$ that are determined by curvatures of $\omega$. We refer to the proof of \cite[Theorem 3.11, p.~39]{Gray} for a related calculation. Then, the weighted holder norm $\|\cdot\|_{\ti{C}^{\alpha}_{-4,\delta}}$ satisfies:
    \begin{equation}
        \begin{split}
            &\|-\gamma(\frac{2d}{\epsilon^{\theta}})\epsilon^{2}\frac{G_{3}}{d^{3}}\|_{\ti{C}^{\alpha}_{-4,\delta}(\{ 2\epsilon^{\beta}\leq d\leq \epsilon^{\theta} \})}
            \\
            \leq& \|G_{3}\|_{C^{0}}\|\gamma(\frac{2d}{\epsilon^{\theta}})\epsilon^{2}\frac{1}{d^{3}}\|_{\ti{C}^{\alpha}_{-4,\delta}}+\|G_{3}\|_{C^{\alpha}_{0}}\|\gamma(\frac{2d}{\epsilon^{\theta}})\epsilon^{2}\frac{1}{d^{3}}\|_{\ti{C}^{0}_{-4,\delta}}\\
            \leq& \|G_{3}\|_{C^{0,\alpha}_{0}}(\|\gamma(\frac{2d}{\epsilon^{\theta}})\epsilon^{2}\frac{1}{d^{3}}\|_{\ti{C}^{\alpha}_{-4,\delta}}+\|\gamma(\frac{2d}{\epsilon^{\theta}})\epsilon^{2}\frac{1}{d^{3}}\|_{\ti{C}^{0}_{-4,\delta}})\\
            \leq& \|G_{3}\|_{C^{0,\alpha}_{0}}\Big(\|\gamma(\frac{2d}{\epsilon^{\theta}})\|_{C^{0,\alpha}_{0}}2\|\frac{\epsilon^{2}}{d^{3}}\|_{\ti{C}^{0}_{-4,\delta}}+\|\gamma(\frac{2d}{\epsilon^{\theta}})\|_{C^{0}_{0}}\|\frac{\epsilon^{2}}{d^{3}}\|_{\ti{C}^{\alpha}_{-4,\delta}} \Big)\\
            \leq& C_{0}\|\gamma(\frac{2d}{\epsilon^{\theta}})\|_{C^{0,\alpha}_{0}}\|\frac{\epsilon^{2}}{d^{3}}\|_{\ti{C}^{0,\alpha}_{-4,\delta}}.
        \end{split}
    \end{equation}
    On the other hand, we may take $C_{0}$ large enough such that
    \begin{equation}
        \|-\gamma(\frac{2d}{\epsilon^{\theta}})\epsilon^{2}\frac{G_{3}}{d^{3}}\|_{\ti{C}^{0}_{-4,\delta}(\{ 2\epsilon^{\beta}\leq d\leq \epsilon^{\theta} \})}
        \leq 
        C_{0}\|\gamma(\frac{2d}{\epsilon^{\theta}})\|_{C^{0}_{0}}\|\frac{\epsilon^{2}}{d^{3}}\|_{\ti{C}^{0}_{-4,\delta}}
        \leq C_{0} \|\frac{\epsilon^{2}}{d^{3}}\|_{\ti{C}^{0,\alpha}_{-4,\delta}}.
    \end{equation}
    In coordinate neighborhood $U$ of Lemma \ref{holosubcoor} type and fix $\frac{4}{3}\epsilon^{\beta}\leq r\leq 2\epsilon^{\theta}$ (covers $\{2\epsilon^{\beta}\leq d\leq \epsilon^{\theta}\}$),
    \begin{equation}
        \begin{split}
            &\|\gamma(\frac{2d}{\epsilon^{\theta}})(rZ,rW)\|_{C^{\alpha}(B_{2}^{2}\setminus B_{1}^{2}\times B_{2}^{m-2})}\\
            =& \sup \frac{|\gamma(\frac{2d}{\epsilon^{\theta}})(rZ_{1},rW_{1})-\gamma(\frac{2d}{\epsilon^{\theta}})(rZ_{2},rW_{2})|}{|(Z_{1},W_{1})-(Z_{2},W_{2})|^{\alpha}} \\
            =& \sup \frac{|\gamma'(\frac{2d_{0}}{\epsilon^{\theta}})|\cdot\frac{2}{\epsilon^{\theta}}|d(rZ_{1},rW_{1})-d(rZ_{2},rW_{2})|}{|(Z_{1},W_{1})-(Z_{2},W_{2})|^{\alpha}}\\
            \leq & \sup |\gamma'(\frac{2d_{0}}{\epsilon^{\theta}})|  \frac{\frac{2}{\epsilon^{\theta}}(1+\rho(rZ_{1},rW_{1}))^{\frac{1}{2}}||rZ_{1}|-|rZ_{2}||}{|(Z_{1},W_{1})-(Z_{2},W_{2})|^{\alpha}} \\
            &+ \sup |\gamma'(\frac{2d_{0}}{\epsilon^{\theta}})|\frac{\frac{2}{\epsilon^{\theta}}|rZ_{2}|\cdot|(1+\rho(rZ_{1},rW_{1}))^{\frac{1}{2}}-(1+\rho(rZ_{2},rW_{2}))^{\frac{1}{2}}|}{|(Z_{1},W_{1})-(Z_{2},W_{2})|^{\alpha}}\\
            \overset{r\leq 2\epsilon^{\theta}}{\leq}& \big[C_{1}\| |Z| \|_{C^{\alpha}(B_{2}^{2}\setminus B_{1}^{2})}+C_{2}(\epsilon^{\theta})^{\alpha}\|(1+\rho)^{\frac{1}{2}}\|_{C^{\alpha}(\ol{U})}\big]\leq C_{3}.
        \end{split}
    \end{equation}
    Now, we have
    \begin{equation}
        \|-\gamma(\frac{2d}{\epsilon^{\theta}})\epsilon^{2}\frac{G_{3}}{d^{3}}\|_{\ti{C}^{0,\alpha}_{-4,\delta}(\{ 2\epsilon^{\beta}\leq d\leq \epsilon^{\theta} \})}
        \leq 
        C_{4}\|\frac{\epsilon^{2}}{d^{3}}\|_{\ti{C}^{0,\alpha}_{-4,\delta}(\{ 2\epsilon^{\beta}\leq d\leq \epsilon^{\theta} \}}.
    \end{equation}
    Back to neighborhood $U$ and fix $\frac{4}{3}\epsilon^{\beta}\leq r\leq 2\epsilon^{\theta}$ again,
    \begin{equation}
    \begin{split}
        & [\ti{\tau}(r)]^{-\delta}r^{4}\|\frac{\epsilon^{2}}{d^{3}}(rZ,rW)\|_{C^{\alpha}(B_{2}^{2}\setminus B_{1}^{2}\times B_{2}^{2})}\\
        =& [\ti{\tau}(r)]^{-\delta}r^{4}\epsilon^{2}\sup \frac{|(|rZ_{1}|\sqrt{1+\rho(rZ_{1},rW_{1})})^{-3}-|(|rZ_{2}|\sqrt{1+\rho(rZ_{2},rW_{2})})^{-3}|}{|(Z_{1},W_{1})-(Z_{2},W_{2})|^{\alpha}}\\
        \leq& [\ti{\tau}(r)]^{-\delta}r^{4}[\sup_{\ol{U}}(1+\rho)^{\frac{3}{2}}]\sup\frac{||rZ_{1}|^{-3}-|rZ_{2}|^{-3}|}{|(Z_{1},W_{1})-(Z_{2},W_{2})|^{\alpha}}\\
        &+[\ti{\tau}(r)]^{-\delta}r^{4}
        \epsilon^{2} r^{-3+\alpha}\|(1+\rho)^{\frac{1}{2}}\|_{C^{\alpha}(\ol{U})}\\
        \leq& C_{5} \epsilon^{2}r^{1-\delta}[\log\frac{1}{\epsilon}]^{-\delta}\\
        \leq& C_{6}\epsilon^{2+\theta(1-\delta)}(\log\frac{1}{\epsilon})^{-\delta}.
    \end{split}
    \end{equation}
    On the other hand, 
    \begin{equation}
        [\ti{\tau}(r)]^{-\delta}r^{4}\|\frac{\epsilon^{2}}{d^{3}}(rZ,rW)\|_{C^{0}(B_{2}^{2}\setminus B_{1}^{2}\times B_{2}^{2})}
        \leq 
        (\frac{3}{2})^{3}\epsilon^{2}[\log\frac{1}{\epsilon}]^{-\delta}r^{1-\delta}
        \leq (\frac{3}{2})^{3}\epsilon^{2+\theta(1-\delta)}(\log\frac{1}{\epsilon})^{-\delta}.
    \end{equation}
    Finally, we may derive the following.
    \begin{equation}
        \|-\gamma(\frac{2d}{\epsilon^{\theta}})\epsilon^{2}\frac{G_{3}}{d^{3}}\|_{\ti{C}^{0,\alpha}_{-4,\delta}(\{ 2\epsilon^{\beta}\leq d\leq \epsilon^{\theta} \})}
        \leq 
        C \epsilon^{2+\theta(1-\delta)}(\log\frac{1}{\epsilon})^{-\delta}
    \end{equation}
    Similarly, the estimate of $-\gamma(\frac{2d}{\epsilon^{\theta}})\epsilon^{2}\frac{G_{2}}{d^{2}}$ can be formulated by
    \begin{equation}
        \|-\gamma(\frac{2d}{\epsilon^{\theta}})\epsilon^{2}\frac{G_{2}}{d^{2}}\|_{\ti{C}^{0,\alpha}_{-4,\delta}(\{ 2\epsilon^{\beta}\leq d\leq \epsilon^{\theta} \})}
        \leq
        C \epsilon^{2+\theta(2-\delta)}(\log\frac{1}{\epsilon})^{-\delta}
        \leq
        C \epsilon^{2+\theta(1-\delta)}(\log\frac{1}{\epsilon})^{-\delta}
    \end{equation}
    Therefore,
    \begin{equation}\label{F_epsi_est1}
        \|F_{\epsilon}\|_{\ti{C}^{\alpha}_{-4,\delta}(\{ 2\epsilon^{\beta}\leq d\leq \epsilon^{\theta} \})}\leq C \epsilon^{2+\theta(1-\delta)}(\log\frac{1}{\epsilon})^{3-\delta}=C\epsilon^{2-\delta}\epsilon^{\theta+(1-\theta)\delta}(\log\frac{1}{\epsilon})^{-\delta}.
    \end{equation}
    Hence, when $|\delta|$ is sufficiently small, we have the desired estimate in Proposition \ref{F}.

\begin{comment}
\section{A review on an alternative natural K\"ahler metric on blowups}    
\end{comment}

%\bibliographystyle{acm}

\bibliographystyle{amsplain}

\bibliography{scalar_curv}

\end{document}